\documentclass[12pt, a4paper, parskip=half, abstracton, bibliography=totoc]{scrartcl}
\pdfoutput=1

\usepackage{array}
\usepackage{marginnote}
\usepackage[dvipsnames]{xcolor}
\usepackage{amscd,amssymb,amsfonts,amsmath,latexsym,amsthm}
\usepackage{hyperref}
\usepackage[capitalize]{cleveref}
\usepackage[all,cmtip]{xy}
\usepackage{mathrsfs}
\usepackage{graphicx}
\usepackage{bm} 
\usepackage{mathtools} 
\usepackage{rotating}

\usepackage{etoolbox}

\let\counterwithout\relax

\usepackage{chngcntr} 

\usepackage{defs_pp1}
\usepackage{slashed}
\usepackage[utf8]{inputenc}
\usepackage{microtype}
\usepackage[english]{babel}

\title{Equivariant coarse homotopy theory  and coarse algebraic $\boldsymbol{K}$-homology}
\author{
Ulrich Bunke\thanks{Fakult{\"a}t f{\"u}r Mathematik,
Universit{\"a}t Regensburg,
93040 Regensburg,
Germany\newline
ulrich.bunke@mathematik.uni-regensburg.de} 
\and
Alexander Engel\thanks{Fakult{\"a}t f{\"u}r Mathematik,
Universit{\"a}t Regensburg,
93040 Regensburg,
Germany\newline
alexander.engel@mathematik.uni-regensburg.de}
\and
Daniel Kasprowski\thanks{
	Rheinische Friedrich-Wilhelms-Universit\"at Bonn, Mathematisches Institut, Endenicher Allee 60,\newline 53115 Bonn, Germany\newline
	kasprowski@uni-bonn.de
}
\and
Christoph Winges\thanks{
	Rheinische Friedrich-Wilhelms-Universit\"at Bonn, Mathematisches Institut, Endenicher Allee 60,\newline 53115 Bonn, Germany\newline
	winges@math.uni-bonn.de}
}

\numberwithin{equation}{section}
\counterwithout{footnote}{section}

\newtheorem{theorem}{Theorem}[section] 
\newtheorem{prop}[theorem]{Proposition}
\newtheorem{lem}[theorem]{Lemma}
\newtheorem{kor}[theorem]{Corollary}

\theoremstyle{remark}
\theoremstyle{definition}
\newtheorem{ddd-alt}[theorem]{Definition}
\newtheorem{ex-alt}[theorem]{Example}
\newtheorem{rem-alt}[theorem]{Remark}

\newenvironment{ddd}    
{%
	\pushQED{\qed}\begin{ddd-alt}}
	{\popQED\end{ddd-alt}}

\newenvironment{ex}    
{%
	\pushQED{\qed}\begin{ex-alt}}
	{\popQED\end{ex-alt}}

\newenvironment{rem}    
{%
	\pushQED{\qed}\begin{rem-alt}}
	{\popQED\end{rem-alt}}

\newcommand{\CoarseHomologyTheories}{\mathbf{CoarseHomologyTheories}}
\newcommand{\All}{\mathbf{All}}

\newcommand{\UBC}{\mathbf{UBC}}

\newcommand{\Rips}{\mathrm{Rips}}
\newcommand{\free}{\mathrm{free}}

\newcommand{\indd}{\mathrm{ind}}

\newcommand{\Mor}{\mathrm{Mor}}
\newcommand{\Yo}{\mathrm{Yo}}

\newcommand{\yo}{\mathrm{yo}}

\newcommand{\Res}{\mathrm{Res}}

\newcommand{\Orb}{\mathbf{Orb}}

\newcommand{\BC}{\mathbf{BornCoarse}}

\newcommand{\cN}{\mathcal{N}}

\newcommand{\Fin}{\mathbf{Fin}}

\newcommand{\Cofib}{\mathrm{Cofib}}

\newcommand{\cP}{\mathcal{P}}

\newcommand{\cZ}{{\mathcal{Z}}}

\newcommand{\cL}{{\mathcal{L}}}

\newcommand{\PSh}{{\mathbf{PSh}}}

\newcommand{\Add}{{\mathtt{Add}}}

\newcommand{\bA}{{\mathbf{A}}}

\newcommand{\cO}{{\mathcal{O}}}
\newcommand{\cU}{{\mathcal{U}}}
\newcommand{\cY}{{\mathcal{Y}}}

\newcommand{\cD}{{\mathcal{D}}}

 \DeclareMathOperator{\Cone}{Cone}

\newcommand{\cE}{{\mathcal{E}}}

\newcommand{\homolg}{\mathrm{hlg}}
\newcommand{\wfl}{\mathrm{wfl}}

\newcommand{\Coarse}{\mathbf{Coarse}}
\newcommand{\Spc}{\mathbf{Spc}}

\newcommand{\IN}{\mathbb{N}}
\newcommand{\IZ}{\mathbb{Z}}

\renewcommand{\Add}{\mathbf{Add}}
\renewcommand{\tilde}{\widetilde}

\DeclareMathOperator{\Nat}{Nat}

\begin{document}
\maketitle

\begin{abstract}
We study equivariant coarse homology theories through an axiomatic framework. To this end we introduce the category of equivariant bornological coarse spaces and construct the universal equivariant coarse homology theory with values in the category of equivariant coarse motivic spectra.

As examples of equivariant coarse homology theories we discuss equivariant coarse ordinary homology and equivariant coarse algebraic $K$-homology.

Moreover, we discuss the cone functor, its relation with equivariant homology theories in equivariant topology, and assembly and forget-control maps. This is a preparation for applications in subsequent papers aiming at split-injectivity results for the Farrell--Jones assembly map. 
\end{abstract}

\tableofcontents

\section{Introduction}

In this paper we study equivariant coarse homology theories.
We start with the equivariant generalization of the coarse homotopy theory developed by Bunke--Engel \cite{buen}. To this end we introduce the category of equivariant bornological coarse spaces and construct the universal equivariant coarse homology theory with values in the category of equivariant coarse motivic spectra.

As examples of equivariant coarse homology theories we discuss
equivariant ordinary coarse homology and equivariant coarse algebraic $K$-homology of an additive category.

An important application of equivariant coarse  homotopy theory is in the study of assembly maps which appear in isomorphism conjectures of Farrell--Jones or Baum--Connes type. The main tools for the transition between equivariant homology theories and equivariant coarse homology theories are the cone functor and the process of coarsification.  In this paper we give a detailed account of the cone functor and the construction of equivariant homology theories from equivariant coarse homology theories.  Then we introduce the coarsification functor and the forget-control map, and discuss its relation with the assembly maps.

The third part of the present paper provides the technical background for subsequent papers:
\begin{enumerate}
\item In \cite{transb} we show that a certain large scale geometric condition called finite decomposition complexity implies that a motivic version of the forget-control map is an equivalence.
\item In  \cite{desc} we study the descent principle. It states that under certain conditions the fact that the forget-control map becomes an equivalence after restriction of the action to all finite subgroups implies that it is split injective for the original (in general infinite) group. 

{We combine this with the results of \cite{transb} and the technical results of the present paper} to deduce split injectivity results for the original Farrell--Jones assembly map.
\item In \cite{ass} we study more formal aspects of the process of 
coarsification of homology theories and provide a general account for coarse assembly maps.
\end{enumerate}

A bornological coarse space is a set equipped with a coarse and a bornological structure such that these structures are compatible with each other. The category of bornological coarse spaces $\BC$ was introduced in \cite{buen} as a general framework for coarse geometry and {coarse} topology.  Interesting invariants of bornological coarse spaces up to {coarse} equivalence are coarse homology theories. In   \cite{buen} the examples of coarse ordinary homology and various coarse versions of topological $K$-homology {were discussed.} In order to study general properties of coarse homology theories the category of motivic coarse spectra $\Sp\cX$ {was constructed} as the target of the universal coarse homology theory
\[\Yo^{s}\colon\BC\to \Sp\cX\ .\]

One of the  motivations to consider equivariant coarse {algebraic topology} is that it appears as a building block of proofs that certain assembly maps are equivalences (Farrell--Jones conjecture \cite{Bartels:2011fk}, \cite{MR2030590} and Baum--Connes conjecture \cite{yu_baum_connes_conj_coarse_geom}).

 In \cref{iureghfiuwewefwefwefewf} we define for every group $\Gamma$  the  category {$\Gamma\BC$} of  $\Gamma$-bornological coarse spaces.  We provide various constructions of $\Gamma$-bornological coarse spaces from $\Gamma$-sets, metric spaces with isometric $\Gamma$-action, or $\Gamma$-simplicial complexes. We further show that the category of $\Gamma$-bornological coarse spaces admits coproducts and fiber products. It also has an interesting symmetric monoidal structure $\otimes$.

In \cref{eiwjweofewfewfew9} we introduce the notion of an equivariant coarse homology theory, and the terminology necessary to state its defining {properties:}
\begin{enumerate}
\item coarse invariance,
\item {coarse excision},
\item vanishing on flasques, and
\item $u$-continuity.
\end{enumerate}

Strongness {(introduced in \cref{secuiwe23g3r})} is an additional property which an equivariant coarse homology might have. It is important in order to interpret the forget-control map as a transformation between equivariant coarse homology theories.

In the present paper the most important additional property is continuity. We introduce this notion in 
\cref{oiefweiofwfewfiewfewfewf}.  Continuity is crucial if one wants to relate the forget-control map with the assembly map for the family of finite subgroups.

Following the line of thought of \cite{buen}, in \cref{e9werlgforeg99} we construct the universal equivariant coarse homology theory \[\Yo^{s}\colon\Gamma\BC\to \Gamma\Sp\cX\] with values in the category of equivariant coarse motivic spectra. Similarly, 
in \cref{oiefweiofwfewfiewfewfewf} we construct a universal continuous equivariant coarse homology theory 
\[\Yo_{c}^{s}\colon\Gamma\BC\to \Gamma\Sp\cX_{c}\ .\]

As examples of equivariant coarse homology theories, in \cref{weoihjowefewfewfewww} we introduce equivariant {coarse} ordinary homology $H\cX^{\Gamma}$ and equivariant coarse algebraic $K$-{homology} $K\bA\cX^{\Gamma}$ of an additive category $\bA$.  Our main results are the verification that the definitions indeed 
  satisfy the four defining properties of an equivariant coarse 
 homology theory. We also show that these examples have the additional properties of being strong, strongly additive and continuous.
 We calculate the evaluations of  these equivariant homology theories  on simple $\Gamma$-bornological coarse spaces. These calculations are important if one wants to understand which equivariant homology theories they induce after pull-back with the cone functor.

In most of the present paper we consider the theory for a  fixed group~$\Gamma$.
But in \cref{roigurgoregregreg} we  consider a homomorphism of groups $H\to \Gamma$ and provide various transitions from $H$-bornological coarse spaces to $\Gamma$-bornological coarse spaces and back. The most important examples are restriction and induction. 
Our  examples of equivariant coarse homology theories are defined for all groups, so in particular for $H$ and $\Gamma$. The group-change construction on the level of bornological coarse spaces are accomplished by natural transformations relating the evaluations of the $H$- and $\Gamma$-equivariant versions of the equivariant homology theories.  

In \cref{rgergiujoi345346456456} we introduce the category of $\Gamma$-uniform bornological coarse spaces $\Gamma\UBC$ and the cone functor
\[\cO\colon\Gamma\UBC\to \Gamma\BC\ .\] The construction of the cone is motivated by Bartels--Farrell--Jones--Reich \cite{MR2030590} and Mitchener \cite{MR1834777, mit}, and its  
main ingredient in the construction is the hybrid coarse structure first introduced by Wright \cite{nw} and studied in detail in \cite{buen}.
Our main technical results are homotopy invariance and excisiveness of the cone.  While the cone $\cO(X)$ depends on the coarse structure on $X$, its germ at $\infty$, {denoted by $\cO^{\infty}(X)$,} is essentially  independent of the coarse structure. So $\cO^{\infty}$ is very close to an equivariant homology theory. But it is still defined on $\Gamma\UBC$ and does not satisfy a wedge axiom.

In \cref{feijwifoewfffwfewfewf} we first review some general features of equivariant homotopy theory and then derive an equivariant  homology theory
\[\cO^{\infty}_{\homolg}\colon\Gamma\Top\to \Gamma\Sp\cX\]
from the functor $\cO^{\infty}$. We then introduce the classifying space $E_{\cF}\Gamma$ for a family $\cF$ of subgroups of $\Gamma$ and define the motivic assembly map as the map
\[\alpha_{E_{\cF}\Gamma}\colon\cO^{\infty}_{\homolg}(E_{\cF}\Gamma)\to 
\cO^{\infty}_{\homolg}(*)\]
induced by the morphism $E_{\cF}\Gamma\to *$. We discuss some conditions on a $\Gamma$-bornological coarse space $Q$ implying that  the twisted version $\alpha_{E_{\cF}}\otimes \Yo^{s}(Q)$   of the assembly map     becomes an equivalence.

In \cref{secjnk232332} we introduce the universal coarsification functor \[F^{\infty}\colon\Gamma\BC\to \Gamma\Sp\cX\] and the forget-control map
\[\beta\colon F^{\infty}\to \Sigma F^{0}\ ,\]
a natural transformation of functors from
$\Gamma\BC$ to $\Gamma\Sp\cX$.
The coarse geometry approach to the isomorphism conjectures provides conditions on a $\Gamma$-bornological coarse space $X$  
implying that the forget-control map $\beta_{X}\colon F^{\infty}(X)\to \Sigma F^{0}(X)$ is an equivalence, or becomes an equivalence after application of a suitable equivariant coarse homology theory. Since one is {also} interested in the assembly maps of equivariant homotopy theory like $\alpha_{E_{\cF}\Gamma}$ we provide a comparison between this assembly map and the forget-control~map.

\paragraph{Acknowledgements}
U.~Bunke and A.~Engel were supported by the SFB 1085 ``Higher Invariants''
funded by the Deutsche Forschungsgemeinschaft DFG. A.~Engel was furthermore supported by the Research Fellowship EN 1163/1-1 ``Mapping Analysis to Homology'', also funded by the DFG.
D.~Kasprowski and C.~Winges acknowledge support by the Max Planck Society. 

Parts of the present work were obtained during the
Junior Hausdorff Trimester Program “Topology” at the Hausdorff Research Institute for Mathematics (HIM) in Bonn. 

We would also like to thank Mark Ullmann for helpful discussions.

\part{General constructions}

\section{Equivariant bornological coarse spaces}

\subsection{Basic notions} \label{iureghfiuwewefwefwefewf}

In this section we introduce the equivariant version of the category of bornological coarse spaces introduced in \cite{buen}. We assume familiarity with \cite[Section 2]{buen}.

Let $\Gamma$ be a group. If $\Gamma$ acts on a bornological coarse space $X$ by automorphisms, then it acts on the set of coarse entourages $\cC$ of $X$. We let $\cC^{\Gamma}$ denote the partially ordered subset of  $\cC$ of entourages of $X$ which are fixed set-wise.

\begin{ddd}\label{vegieirt456456}
A \emph{$\Gamma$-bornological coarse space} is a bornological coarse space $X$  together with an action of $\Gamma$ by  automorphisms such that  $\cC^{\Gamma}$  
is cofinal in  $\cC$.

A  morphism  between $\Gamma$-bornological coarse spaces is a morphism of bornological coarse spaces  which is in addition $\Gamma$-equivariant.
\end{ddd}

We let $\Gamma\BC$ denote the category of $\Gamma$-bornological coarse spaces and morphisms. By considering a bornological coarse space  as a $\Gamma$-bornological coarse space with the trivial action we get a fully faithful functor
\begin{equation}\label{vreivjhoeivjoieveeerverv}
C\colon\BC\to \Gamma\BC\ .
\end{equation}

\begin{ex}
Let $X$ be a set with an  action of $\Gamma$ and $A\subseteq \cP(X\times X)^{\Gamma}$ be a family of $\Gamma$-invariant subsets. Then we can form  the  coarse structure $\cC := \cC\langle A\rangle$ generated by $A$. In this case $\cC^{\Gamma} $ is cofinal in $\cC$. Hence a $\Gamma$-coarse structure can be generated by a family of $\Gamma$-invariant entourages.
\end{ex}

\begin{rem}\label{ekfhwekjfkjewfewfwfwfewf}
Let $\Coarse $ denote the category of coarse spaces (i.e., sets with coarse structures) and controlled morphisms. We can consider $\Gamma$ equipped with  the minimal coarse structure $\cC\langle \diag_{\Gamma}\rangle$ as a group object in $\Coarse$. 

Let $X$ be a set with an action of a group $\Gamma$ and $\cC$ be a coarse structure {on $X$}.
The following conditions are equivalent:
\begin{enumerate}
\item $\cC^{\Gamma}$ is cofinal in $\cC$.
\item For every entourage $U$ in $\cC$ the set
$\bigcup_{\gamma\in \Gamma}(\gamma\times \gamma)(U)$ also belongs to $\cC$.
\item The action is a morphism $\Gamma\times X\to X$ in $\Coarse$.
\end{enumerate}
The proof is straightforward.

 We denote the category   of $\Gamma$-coarse spaces consisting of coarse spaces with a $\Gamma$-action satisfying the above conditions and equivariant and controlled maps by $\Gamma\Coarse$.
\end{rem}

\begin{ex}\label{eifjwfiowefewfe32453534}
We consider $\Gamma$ as a $\Gamma$-set with the left action. We furthermore let $\cB_{min}$ be the minimal bornology on $\Gamma$ consisting of the finite subsets. Finally,  we let  the coarse structure $\cC_{can}$ on $\Gamma$ be generated by the $\Gamma$-invariant sets
$\Gamma(B\times B)$ for all $B$ in $ \cB_{min}$.  

Then $(\Gamma,\cC_{can},\cB_{min})$ is a $\Gamma$-bornological coarse space called the  canonical $\Gamma$-bornological coarse space associated to $\Gamma$. We will denote it by $\Gamma_{can,min}$.
\end{ex}

\begin{ex}
Let $X$ be a set with an action of $\Gamma$. It gives rise to the $\Gamma$-bornological coarse space  $X_{min,min}$ with the minimal structures $\cB_{min}$ consisting of the finite subsets of $X$ and $\cC_{min}:=\cC\langle \diag_{X}\rangle$. We will use the notation $X_{min,min}$ for this $\Gamma$-bornological coarse space. 

For example, the identity of the underlying set of $\Gamma$ is  a morphism $\Gamma_{min,min}\to \Gamma_{can,min}$ of $\Gamma$-bornological coarse spaces.
\end{ex}

\begin{ex}
Let $X$ again be a set with an action of $\Gamma$. Then we can equip $X$ with the maximal bornological and coarse structures. In this way we get a $\Gamma$-bornological coarse space $X_{max,max}$.

Our notation convention is such that the first subcript indicates the coarse structure, while the second subscript reflects the bornological structure.

For example, we also have a $\Gamma$-bornological coarse space $X_{min,max}$ and morphisms of $\Gamma$-bornological coarse spaces
$X_{min,max}\to X_{max,max}$ {and $X_{min,max}\to X_{min,min}$} given by the identity of the underlying set of $X$.

Note that in general $X_{max,min}$ does not make sense since the minimal bornology is not compatible with the maximal coarse structure. 
\end{ex}

\begin{ex}
{Let} $(X,d)$ {be} a metric space with an isometric $\Gamma$-action. 
For $r$ in $(0,\infty)$ we consider the invariant entourages
\[U_{r}:=\{(x,y)\in X\times X\:|\:d(x,y) \le r\}\ .\]
The coarse structure associated to the metric is defined  by these entourages, i.e., given by
\[\cC_{d}:=\cC\langle \{U_{r}\:|\: r\in (0,\infty)\}\rangle\ .\]
Furthermore, the bornology associated to the metric is generated by the metrically bounded subsets, i.e., given by
\[\cB_{d}:=\cB\langle \{B(x,r)\:|\: x\in X\ , r\in (0,\infty)\}\rangle\ ,\]
where $B(x,r)$ denotes the metric ball of radius $r$ centered at $x$.

The  associated bornological coarse space $X_{d}:=(X,\cC_{d},\cB_{d})$ is a $\Gamma$-bornological coarse space. 

The identity of the underlying set of $X$ is a morphism {$X_{min,max}\to X_{d}$} of $\Gamma$-bornological coarse spaces.
\end{ex}

\begin{rem}
Let $\Gamma$ be a countable group equipped with a proper left invariant metric $d$. Then the $\Gamma$-bornological coarse spaces $\Gamma_d$ and $\Gamma_{can,min}$ are equal.
\end{rem} 

\begin{ex}\label{wefoiweofiewf}
 Let $X$ be a $\Gamma$-bornological coarse space with the coarse structure $\cC$ and the bornology $\cB$. For every invariant entourage $U$ in $\cC$ we consider the coarse structure
 $\cC_{U}:=\cC\langle \{U\}\rangle$. The coarse structure $\cC_{U}$ is compatible with $\cB$. We let
 \[X_{U}:=(X,\cC_{U},\cB)\]
 denote the resulting $\Gamma$-bornological coarse space. 
 The identity of the underlying set is a morphism of $\Gamma$-bornological coarse spaces $X_{U}\to X$. If $U^{\prime}$ is a second invariant entourage such that $U\subseteq U^{\prime}$, then we also have a morphism $X_{U}\to X_{U^{\prime}}$.
This construction is important for the formulation of the $u$-continuity condition in \cref{fewoijfewoifewofww45535345345}.
\end{ex}

\begin{ex}\label{exjk22332}
Let $X$ be a  $\Gamma$-bornological coarse space with coarse structure $\cC$ and bornology $\cB$, and let $Z$ be a $\Gamma$-invariant subset of~$X$.
Then we define the induced coarse structure and bornology on $Z$ as follows:
\begin{enumerate}
\item $\cC_{Z}:=\{(Z\times Z)\cap U\:|\: U\in \cC\}$
\item $\cB_{Z}:=\{Z\cap B\:|\: B\in \cB\}$.
\end{enumerate}
Then $Z_{X}:=(Z,\cC_{Z},\cB_{Z})$ is a $\Gamma$-bornological coarse space. The inclusion $Z_{X}\to X$ is a morphism of $\Gamma$-bornological coarse spaces.
\end{ex}

\begin{ex}
If $\Gamma$ acts on the underlying set of a bornological coarse space $ (X,\cC,\cB)$, then we can define a $\Gamma$-bornological coarse space $\Gamma X:=(X,\Gamma \cC,\Gamma \cB)$, where
\[  \Gamma\cC:=\cC\left\langle\left\{{\bigcup_{\gamma\in\Gamma}(\gamma\times\gamma)(U)}\:|\: U\in \cC\right\}\right\rangle\] and \begin{equation}\label{cjdhciuchiuchwc}  \Gamma\cB:=\cB\langle\{U[{\gamma} B]\:|\: U\in \Gamma\cC, {\gamma\in\Gamma} \text{ and } B\in \cB\}\rangle\ .\end{equation}
In general we must enlarge the bornology $\cB$ to $\Gamma\cB$ as described above in order to keep it compatible with the new coarse structure.
 
If $X$ was a $\Gamma$-bornological coarse spaces, then $\Gamma X=X$.
\end{ex}

\begin{ddd}\label{goihreiugiuuz34r34t3t3t}
Let $\Gamma$ denote a group and let $X$ be a set with an action of $\Gamma$. If $\cB$ is a bornology on $X$, then we let $\cB_{\Gamma}$ denote the bornology on $X$ which is generated by the sets $\Gamma B$ for all $B$ in $\cB$.
 
Let $(X,\cC,\cB)$ be a $\Gamma$-bornological coarse space.  Then the  new  bornology $\cB_{\Gamma}$ is compatible with the original coarse structure. The \emph{$\Gamma$-completion} of $(X,\cC,\cB)$ is defined to be the $\Gamma$-bornological coarse space $(X,\cC,\cB_{\Gamma})$.

Let $(X,\cC,\cB)$ be a $\Gamma$-bornological coarse space and $Y$ be a  subset of $X$. The subset $Y$ is called \emph{$\Gamma$-bounded} if it belongs to $\cB_{\Gamma}$.
\end{ddd}

\begin{ex}\label{fewkfjwklefjewlkfewfewe}
Let $(X,\cB)$ be a bornological space with an action of $\Gamma$ by proper maps.  We say that $\Gamma$ acts \emph{properly} if for every $B$ in $\cB$ the set $\{\gamma\in \Gamma\:|\: \gamma B\cap B\not=\emptyset\}$ is finite.

We define the coarse structure $\cC_\cB$ on $X$ {to be} generated by the $\Gamma$-invariant  entourages $U_{B}:=\Gamma(B\times B)$ for all $B$ in $\cB$.
If $\Gamma$ acts properly, then the bornological structure is compatible with this coarse structure and we get a $\Gamma$-bornological coarse space ${(X,\cC_\cB,\cB)}$.
\end{ex}

\begin{ex}
Let $X$ be a $\Gamma$-complete $\Gamma$-bornological coarse space with coarse structure $\cC$ and bornology $\cB$.
We equip the quotient set $\bar X:=\Gamma\backslash X$ with the maximal bornology $\bar \cB$ such that the projection $q\colon X\to \bar  X$ is proper, i.e.,
\[\bar \cB=\cB\langle\{\bar B\subseteq \bar X\:|\: q^{-1}(\bar B)\in \cB\}\rangle\ .\]
We furthermore equip $\bar X$ with the minimal coarse structure $\bar\cC$ such that $q$ is controlled, i.e.,
\[\bar \cC:=\cC\langle\{{(q\times q)}(U)\:|\: U\in \cC\}\rangle\ .\]
Then $\bar \cC$ and $\bar \cB$ are compatible and we obtain a bornological coarse space $(\bar X,\bar \cC,\bar \cB)$.

This construction produces a functor
\[Q\colon\Gamma\BC\to \BC\ , \quad Q(X,\cC,\cB):=(\bar X,\bar \cC,\bar \cB)\ .\]
It is easy to check that the morphism of bornological coarse spaces
$q\colon(X,\cC,\cB)\to (\bar X,\bar \cC,\bar \cB)$ can be interpreted as the unit of an adjunction
\[Q\colon\Gamma\BC\leftrightarrows \BC: C\ ,\]
where $C$ is the inclusion \eqref{vreivjhoeivjoieveeerverv}.
\end{ex}

\begin{lem}
The category $\Gamma\BC$ admits arbitrary coproducts and cartesian products.
\end{lem}

\begin{proof}
The coproduct of  a family of  $\Gamma$-bornological coarse spaces is represented by the coproduct of the underlying  bornological coarse spaces with the induced $\Gamma$-action.

Similarly, the cartesian product of a family of $\Gamma$-bornological coarse spaces is represented by the cartesian product of the family of   the underlying bornological coarse spaces with the induced $\Gamma$-action.
\end{proof}
 
For more information about limits and colimits in $\Gamma\BC$ we refer to Section~\ref{griojoregjegegg}.

\begin{ex}\label{efwifuhwfowefewfewfwef}
We consider a family $(X_{i})_{i\in I}$ of $\Gamma$-bornological coarse spaces.  We define the free union
$\bigsqcup_{i\in I}^{\free}X_{i}$ as follows:
\begin{enumerate}
\item The  underlying $\Gamma$-set of the free union  is the disjoint union of $\Gamma$-sets $\bigsqcup_{i\in I}X_{i}$.
\item The coarse structure of the free union is generated by entourages $\bigsqcup_{i\in I} U_{i}$ for all families $(U_{i})_{i\in I}$, where $U_{i}$ is an entourage of $X_{i}$ for every $i$ in $I$.
\item The bornology is generated by the set
$\{B\mid B\in \cB(X_i), i\in I\}$ of subsets of $\bigsqcup_{i\in I}X_{i}$.
\end{enumerate}
If $I$ is finite, then the free union is the coproduct of the family. In general, we have a morphism of $\Gamma$-bornological coarse spaces
\[\coprod_{i\in I} X_{i}\to \bigsqcup_{i\in I}^{\free}X_{i}\]
{induced by the identity of the underlying sets.}
\end{ex}

\begin{ex}\label{fijwefiewofejfoiewjfoijojoi}
Let $X$ and $X^{\prime}$ be two $\Gamma$-bornological coarse spaces. Then we can form the $\Gamma$-bornological coarse space $X^{\prime}\otimes X$ whose coarse structure is the one of the cartesian product and the bornology is generated by the products $B^{\prime}\times B$ for bounded subsets $B^{\prime}$ of $X^{\prime}$ and $B$ of $X$.
This construction defines a symmetric monoidal structure
\[-\otimes -\colon\Gamma\BC\times \Gamma\BC\to \Gamma\BC\]
with tensor unit given by the one-point space.

Let $Y$ be a $\Gamma$-set. We can form the space $Y_{min,min}\otimes  X$. For a second $\Gamma$-set $Y^{\prime}$ we have a canonical isomorphism
\[(Y^{\prime}\times Y)_{min,min}\otimes X\cong Y_{min,min}^{\prime}\otimes (Y_{min,min}\otimes X)\ .\qedhere\]
\end{ex}

\subsection{Limits and colimits in \texorpdfstring{$\Gamma\BC$}{GammaBornCoarse}}\label{griojoregjegegg}

In this section we show that the category $\Gamma \BC$ admits all limits of non-empty diagrams and various colimits.
We furthermore discuss some special cases.

Let $\Gamma $ be a group.
In the following arguments we let
\begin{equation}
\label{eq_forget_to_set}
\iota\colon\Gamma \BC\to \Gamma \Set
\end{equation}
be the forgetful functor.
For a $\Gamma $-bornological coarse space $X$ we let $\cB_{X}$ and $\cC_{X}$ denote its bornology and coarse structure.

\begin{prop}\label{grioeo34jigreergreg}\mbox{}\begin{enumerate}\item 
The category  $\Gamma\Coarse$ admits all small limits. \item The category   $\Gamma \BC$ admits all limits of diagrams indexed by non-empty small categories.\end{enumerate}
\end{prop}
\begin{proof}
 We give the proof for $\Gamma \BC$. The statement for $\Gamma\Coarse$ can be obtained by ignoring all comments pertaining to bornologies and allowing in addition  the index set $I$ below to be empty. 

Note that the non-emptiness assumption on $I$ only enters into the part of the proof concerning bornologies.   See also \cref{houhgoihbgioanionioghio}.

Let $I$ be a non-empty small category and \[X\colon I\to \Gamma \BC\] be a functor.
We will show that $\lim_{I}X$ exists.

The category $\Gamma \Set$ is complete.
In a first step we form the $\Gamma $-set 
\[ \tilde Y:=\lim_{I}\iota X\ .\]

For every $i$ in $I$ we have a map of $\Gamma $-sets
$e_{i}\colon   \tilde Y\to \iota X(i)$. 

On $\tilde Y$ we define the bornology \[\cB_{Y}:=\cB\langle \{e_{i}^{-1}(B)\:|\: i\in I\: \text{and} \: B\in \cB_{X(i)}\} \rangle\ .\]
 
Since $\cB_{X(i)}$ is $\Gamma $-invariant for every $i$ in $I$, we see that $\cB_{Y}$ is $\Gamma $-invariant. 
 
We can view $\tilde Y$ as a subset of $\prod_{i\in I} \iota X(i) $. 
 {On $\tilde Y$ we define the coarse structure}
\[\cC_{Y}:=\cC \big\langle \big\{\big(\prod_{i\in I} U_{i}\big)\cap(\tilde Y\times \tilde Y) \:|\:(U_{i})_{i\in I}\in \prod_{i\in I} \cC_{X(i)} \big\} \big\rangle\ .\]
   
Using that $\cC_{X(i)}$ has a cofinal subset of invariant entourages for every $i$ in $I$, we see that the coarse structure $\cC_{Y}$  is generated by invariant entourages and is hence a $\Gamma $-coarse structure.

Finally, the relation
\[\big(\prod_{i\in I} U_{i}\big)[e_{j}^{-1}(B)]=e_{j}^{-1}( {U_{j}}[B])\]
for $j$ in $I$ and $B$ a subset of $X_{j}$
shows that the compatibility of $\cB_{X(j)}$ with $\cC_{X(j)}$ for all $j$ in $I$ implies that $\cB_{Y}$ and $\cC_{Y}$ are compatible.

We therefore have defined an object    $Y:=(\tilde Y,\cC_{Y},\cB_{Y}) $ in $\Gamma \BC$.
We now show that $e_{j}\colon Y\to X(j)$ is a morphism of $\Gamma $-bornological coarse spaces for all $j$ in $I$.
For every family $(U_{i})_{i\in I}$   we have  $e_{j}((\prod_{i\in I} U_{i})\cap(\tilde Y\times \tilde Y))\subseteq U_{j}$.
This implies that $e_{j}$ is controlled. Furthermore, $\cB_{Y}$ is defined such that $e_{j}$ is proper for every $j$ in $I$.

The morphisms $e_{i}\colon Y\to X(i)$ give a transformation
$e\colon \underline{Y}\to X$ in $\Fun(I,\Gamma \BC)$,  {where $\underline{Y}$ is the constant functor with value $Y$.}
We now show that $(Y,e)$ has the universal property of the limit .

We consider a pair $(Z,f)$ with $Z$ in $\Gamma \BC$ and with $f\colon \underline{Z}\to X$ a morphism in  $\Fun(I,\Gamma \BC)$.
Because the underlying $\Gamma $-set of $Y$ is the limit of the diagram $\iota X$ 
 there is a unique map $  h\colon \iota Z\to \iota Y$ of $\Gamma $-sets such that $\iota f=\iota e\circ  \underline{  h}$.
 
It suffices to show that $h$ is a morphism in $\Gamma \BC$.
Let $U$ be an entourage of $Z$. Then $f_{i}(U)$ is an entourage of $X(i)$ for every $i$ in $I$.
We have
\[h(U)\subseteq \big(\prod_{i\in I}  {(f_{i}\times f_i)}(U)\big)\cap (\tilde Y\times \tilde Y)\ ,\]
i.e., $h(U)$ is contained in one of the generating entourages of $\cC_{Y}$.
Hence $h$ is a  {controlled} map.

In order to show that $h$ is proper, we consider a generating bounded
subset $e_{i}^{-1}(B)$ for $i $ in $I$ and $B$ in $\cB_{X(i)}$.
Then $h^{-1}(e_{i}^{-1}(B))=f_{i}^{-1}(B)$ is bounded in $Z$.
Since $I$ is non-empty, the set of subsets $e_{i}^{-1}(B)$ for all $i$ in $I$ and $B$ in $\cB_{X(i)}$ cover $Y$.
Therefore, every element of $\cB_{Y}$ is contained in a finite union of such subsets.
We conclude that   $h$ is proper.
\end{proof}

\begin{rem}\label{houhgoihbgioanionioghio}
 Note that the last piece of the argument goes wrong if the index  category of the  diagram is empty.
 Then $Y=*$, and we need the subset $\{*\}$ (which is not of the form $e_{i}^{-1}(B)$) to generate the bornology.
 The empty limit does not exist since the category $\Gamma \BC$ does not have a final object.
\end{rem}

We now turn to colimits.
Let $I$ be a small category and \[X\colon I\to \Gamma \BC\] be a functor. The category $\Gamma \Set$ is cocomplete.
We form the $\Gamma $-set 
\[ \tilde Y:=\colim_{I}\iota X\ ,\]
 {where $\iota$ is the forgetful functor from \eqref{eq_forget_to_set}.}  
For every $i$ in $I$ we have a map $e(i)\colon \iota X(i)\to \tilde Y$ of $\Gamma $-sets.
\begin{ddd}
We say that the diagram $X$ is colim-admissible if   we have
\[e_{i}^{-1}((e_{j_{1}}(U_{1})\circ \dots\circ e_{j_{k}}(U_{k}))[\{y\}])\in \cB_{X(i)}\] for every $y$ in $\tilde Y$,
every $i$ in $I$, every $r$ in $\nat$, every family $(j_{1},\dots,j_{r})$ of objects of $I$, and every  family of entourages $U_{k}$ in $\cC_{X(j_{k})}$ for $k$ in $\{1,\dots,r\}$.
\end{ddd}

\begin{prop}\label{efiu9ewofwfwewfwef}\mbox{}\begin{enumerate}\item
The category  $\Gamma\Coarse$ admits all small colimits. \item The category  $\Gamma \BC$ admits colimits for all colim-admissible diagrams.\end{enumerate}
\end{prop}
\begin{proof}
 We give the proof for $\Gamma \BC$. The statement for $\Gamma\Coarse$ can be obtained by ignoring all comments pertaining to bornologies.
Note that the assumption of colim-admissibility only enters to see that the bornology and the coarse structure on the colimit are compatible. 

Assume that \[X\colon I\to \Gamma \BC\] is a colim-admissible diagram. 
We show that  $\colim_{I}X$ exists.

In a first step we form the $\Gamma $-set 
\[ \tilde Y:=\colim_{I}\iota X\ .\]   
On $\tilde Y$ we  define the coarse structure
\[\cC_{Y}:=\cC\langle \{ {(e_{i}\times e_i)}(U)\:|\: i\in I \text{ and } U\in \cC_{X(i)}\}\rangle \ .\]
 Using the fact that $\cC_{X(i)}$ has a cofinal subset of invariant entourages we see that the coarse structure $\cC_{Y}$  is generated by invariant entourages
and is hence a $\Gamma $-coarse structure.

We define the bornology $\cB_{Y}$ to be the subset
  of $\cP(\tilde Y)$ consisting of the sets $B$
satisfying
\[e_{i}^{-1}((e_{j_{1}}(U_{1})\circ \dots\circ e_{j_{k}}(U_{k}))[B])\in \cB_{X(i)}\]
for every $i$ in $I$, every $r$ in $\nat$, every family $(j_{1},\dots,j_{r})$ of objects of $I$, and every  family of entourages $U_{k}$ in $\cC_{X(j_{k})}$ for $k$ in $\{1,\dots,r\}$.  
Since the diagram is colim-admissible, all one-point sets belong  to  $\cB_{Y}$. Furthermore $\cB_{Y}$ is obviously closed under forming finite unions and subsets. Consequently, $\cB_{Y}$ is a bornology on $Y$.
 Since  $\cB_{X(i)}$ and $\cC_{X(i)}$ are $\Gamma $-invariant 
for every $i$ in $I$ we see that $\cB_{Y}$ is $\Gamma $-invariant. 
We finally observe that  $\cC_{Y}$ and $\cB_{Y}$ are compatible by construction.

We  define now the  object $Y:=(\tilde Y,\cC_{Y},\cB_{Y})$ of $\Gamma \BC$.
By construction the maps $e_{i}\colon X(i)\to Y$ are morphisms in $\Gamma \BC$ for all $i$ in $I$.
The family of morphisms $(e_{i})_{i\in I}$ provides a morphism
$e\colon X\to \underline{Y}$ in $\Fun(I,\Gamma \BC)$, {where $\underline{Y}$ is the constant functor with value $Y$.}

We now show that $(Y,e)$ has the universal property of a colimit.

Consider a pair $(Z,f)$ with $Z$ in $\Gamma \BC$ and $f\colon X\to \underline{Z} $. 
Since the underlying $\Gamma $-set of $Y$ is the colimit of the diagram $\iota X$ 
 there is a unique map $h\colon \iota Y\to \iota Z$ of $\Gamma $-sets such that $\iota f=  \underline{h}\circ \iota e$.

 It suffices to show that $h$ is a morphism.
 
Let $i$ in be $I$ and $U$ be  in $\cC_{X(i)}$. Then
$h(e_{i}(U))=f_{i}(U_{i})$ is an entourage of $Z$. This implies that
$h$ is controlled.

Let now $B$ be a bounded subset of $Z$. Since $\cC_{Z}$ and $\cB_{Z}$ are  compatible and $f_{j}$ is controlled,  $f_{j}(U)[B]$ is bounded for every $j$ in $I$ and $U$ in $\cC_{X_{j}}$.
We now have\footnote{Here we use the general relation $U[f^{-1}(B)]\subseteq f^{-1}(f(U)[B])$ for a map $f\colon X\to Y$, entourage $U$ of~$X$ and subset $B$ of $Y$.}
\[e_{i}^{-1}(e_{j}(U)[h^{-1}(B)])\subseteq e_{i}^{-1}(h^{-1}(h(e_{j}(U))[B]))=f_{i}^{-1}(f_{j}(U)[B])\ .\]
Since $f_{i}$ is proper, we conclude that $e_{i}^{-1}(e_{j}(U)[h^{-1}(B)])$ is bounded. Since $i,j$ in $I$ and $U$ in $\cC_{X(j)}$ were arbitrary this shows that
$h^{-1}(B)$ is bounded in $Y$.
We conclude that $h$ is a proper map.
\end{proof}

\begin{ex}
If $(\cY,Z)$ is an equivariant  complementary pair (see Definition \ref{gerojogregege}) on $X$, then we have a push-out
\[\xymatrix{\colim \cY\cap Z\ar[r]\ar[d]&Z\ar[d]\\\colim \cY\ar[r]&X}\]
 {We first note that $\colim \cY$ is admissible since it is a filtered colimit of inclusions.} 
It is straightforward to check that the diagram is colim-admissible and that the bornology
on the space $X$ is the one of the colimit.
The only non-trivial fact to check is that the coarse structure on $X$ given by the colimit is not too small.
Let $U$ be an entourage of $X$. Assume that $i$ is in $I$ such that $Y_{i}\cup Z=X$. Since the family is big, there
exists $j$ in $I$ such that $U[Y_{i}]\subseteq Y_{j}$.  Let $e\colon Y_{j}\to X$ and $f\colon Z\to X$ be the inclusions.
Then
\[U\subseteq e(U\cap (Y_{j}\times Y_{j}))\cup f(U\cap (Z\times Z))\ .\qedhere\]
\end{ex}

\begin{ex}
By Proposition \ref{grioeo34jigreergreg} the category $\Gamma\BC$ admits fiber products.
Here we give an explicit description.
 We consider a diagram
 \[\xymatrix{ &X\ar[d]^{a}\\Y\ar[r]^{b}&Z}\]
 of $\Gamma$-bornological coarse spaces. We then form the cartesian product $X\times Y$ in the category $\Gamma\BC$. We define the $\Gamma$-bornological coarse space  $X\times_{Z}Y$ to be the subset of $X\times Y$ of pairs   $(x,y)$ with $a(x)=b(y)$ with the induced bornological and coarse structure. It is straightforward to check that the square 
\[\xymatrix{
X\times_{Z}Y\ar[d]\ar[r]&X\ar[d]^{a}\\
Y\ar[r]^{b}&Z
}\]
is cartesian in $\Gamma\BC$.
\end{ex}

\begin{ex}
Assume that $(Y,Z)$ is a coarsely excisive pair (see Definition \ref{dddi34546}) on $X$. Then we have a push-out
\[\xymatrix{Y\cap Z\ar[r]\ar[d]&Z\ar[d]\\Y\ar[r]&X}\]
It is straightforward to check that the diagram is colim-admissible and that the bornology
on the space $X$ is the one of the colimit.
The only non-trivial fact to check is that the coarse structure on $X$ given by the colimit is not too small.
Let $U$ be an entourage of $X$. Then there is an entourage $W$ of $X$ such that $U\subseteq W$ and $U[Z]\cap U[Y]\subseteq W[Z\cap Y]$.  Let $e\colon Y\to X$ and $f\colon Z\to X$ be the inclusions.
Then
\[U\subseteq e(W^{2}\cap (Y \times Y ))\cup f(W^{2}\cap (Z\times Z))\ .\qedhere\]
\end{ex}

\begin{ex}
	\label{ex:coeq}
Let $H$ be a group acting on a $\Gamma$-bornological coarse space $X$ such that the set of $H$-invariant entourages of $X$ 
is cofinal in all entourages. 

We let $B_{H}(X)$ denote 
$H$-completion of $X$ obtained from $X$ by replacing the bornology $\cB$ of $X$ by the bornology $B_{H}(\cB)$
generated by the subsets $HB$ for all $B$ in $\cB$. Then the coequalizer
\[(H_{min,min}\otimes X)_{max-\cB}\rightrightarrows B_{H}(X) \stackrel{\pi}{\to} X/H\]
for the $H$-action 
exists.  Here the index $-_{max-\cB}$ indicates that we replaced the bornology by the maximal bornology.
The two arrows are given by
$(h,x)\mapsto x$ and $(h,x)\mapsto hx$.
One checks easily that they are both morphisms of $\Gamma$-bornological coarse spaces: they are both proper since their domain has the maximal bornology, and if $U$ is a $H$-invariant entourage of~$X$ then both maps send
$\diag(H)\times U$ to $U$ and hence the maps are controlled in view of our assumption on the coarse structure of $X$.

Finally, we check that the coequalizer diagram is $\colim$-admissible.
It suffices to check that for every $H$-invariant entourage $U$ of $X$ 
and point $Hx$ in $X/H$ the set $U[\pi^{-1}(Hx)]$ is bounded in $B_{H}(X)$. This is the case since $U[x]$ belongs to $\cB$ and so
$U[\pi^{-1}(Hx)]=HU[x]$ belongs to $B_{H}(\cB)$.
\end{ex}

\begin{ex}
Let
\[ M \colon \Gamma\BC\to \Gamma\Set \]
be the functor which sends a $\Gamma$-bornological coarse space to its underlying $\Gamma$-set. In view of \cref{grioeo34jigreergreg},
it  preserves all limits over non-empty small index categories. It is in fact the right-adjoint of an adjunction  
\[ (-)_{min,max} \colon \Gamma\Set\leftrightarrows \Gamma\BC:M\ ,\]
where $(-)_{min,max}$ sends a $\Gamma$-set $S$ to the $\Gamma$-bornological coarse space obtained by equipping $S$ with the minimal coarse structure and maximal bornology. 
Indeed, for a $\Gamma$-bornological coarse space $X$ we have a natural identification
\[ \Hom_{\Gamma\BC}(S_{min,max},X)\cong \Hom_{\Gamma\Set}(S,M(X))\ .\qedhere\]
\end{ex}

\section{Equivariant coarse homology theories}
\label{eiwjweofewfewfew9}
  
The following notions are the obvious generalizations from the non-equivariant situation considered in \cite{buen}.

Two morphisms $f,f'\colon X\to X'$ between $\Gamma$-bornological coarse spaces are \emph{close} to each other if the subset  $\{(f(x),f'(x))\mid x\in X\}$ of $X^{\prime}\times X^{\prime}$ is an entourage, i.e. if they are close as morphisms between the underlying bornological coarse spaces.
\begin{ddd}\label{goreoigegerg}
 A  morphism   between $\Gamma$-bornological coarse spaces is an \emph{equivalence} if  it admits an   inverse morphism up to closeness.      \end{ddd}

\begin{ex}\label{sdfj89231}
We consider a $\Gamma$-bornological coarse space $X$, a $\Gamma$-invariant subset $A$ of $X$, and a $\Gamma$-invariant entourage $U$  of $X$. Then we can form the $U$-thickening $U[A]$ which is again $\Gamma$-invariant. We now assume that $U$ contains the diagonal. Then we have a natural inclusion $i\colon A\to U[A]$. This inclusion is in general not an equivalence of $\Gamma$-bornological coarse spaces.

For example, let the group $\Z$ act on $\C$ by $(n,z)\mapsto e^{2\pi i\theta n}z$, where $\theta$ is an irrational real number. Then the subset $\C\setminus \{0\}$ of $\C$ is $\Z$-invariant. Every non-trivial thickening of this subset contains the point $0$. This point is fixed by the action, but $\C\setminus \{0\}$ does not contain any fixed point which could serve as the image of $0$ under a potential inverse of the inclusion $\C\setminus \{0\}\to \C$.
\end{ex}

Let $X$ be a $\Gamma$-bornological coarse space and $A $ be a $\Gamma$-invariant subset of $X$.
\begin{ddd}
The subset $A$ is called \emph{nice} if for every invariant entourage $U$ of $X$ {containing the diagonal} the inclusion $A\to U[A]$ is an equivalence.
\end{ddd}

\begin{ex}\label{fuhewiufhuewffwefewf}
Let $X$ be a $\Gamma$-bornological coarse space and {let} $Y$ be a bornological coarse space considered as a $\Gamma$-bornological coarse space with the trivial $\Gamma$-action.

For every subset $A$ of $ Y$ the subset $A\times X$ of the $\Gamma$-bornological coarse space $Y\times X$  (or of $Y\otimes X$) is nice. 
\end{ex}

A filtered family of subsets of a set $X$ is a family $(Y_{i})_{i\in I}$ of subsets indexed by a filtered partially ordered set $I$  such that the map $I\to \cP(X)$ given by  $i\mapsto Y_{i}$ is order-preserving.

Let $X$ be a $\Gamma$-bornological coarse space. Recall from \cite{buen} that a \emph{big family} on $X$ is a filtered family of subsets $(Y_{i})_{i\in I}$ of $X$  such that for every entourage $U$ of $X$ and $i$ in $I$ there exists $j$ in $I$ such that $U[Y_{i}]\subseteq Y_{j}$.
 
\begin{ddd}
 An \emph{equivariant big family} on   $X$ is a big family      consisting of $\Gamma$-invariant subsets.  \end{ddd}
 
\begin{ex}\label{pokpogergeerge}
Let $X$ be a $\Gamma$-bornological coarse space and $A$   be a $\Gamma$-invariant subset of~$X$.
Then the family
\[\{A\}:=(U[A])_{U\in \cC^{\Gamma}}\]
is an equivariant big family.
\end{ex}

Let $X $ be a $\Gamma$-bornological coarse space. 
Recall from \cite{buen} that a complementary pair $(Z,\cY)$ on $X$ is a pair of a subset $Z$ of $X$ and a big family $\cY=(Y_{i})_{i\in I}$ on $X$ such that there exists $i$ in $I$   with $Z\cup Y_{i}=X$.
 
 \begin{ddd}\label{gerojogregege}
 An \emph{equivariant complementary pair}   on  $X$ is a complementary pair 
  $(Z,\cY)$ such that  $Z$ is a  $\Gamma$-invariant subset and   $\cY$ is an equivariant big family.
  \end{ddd}
  
  \begin{ddd}\label{hckjhckhwhuiechiu1}   A $\Gamma$-bornological coarse space $X$ is \emph{flasque} if it admits a morphism $f\colon X\to X$ such that
  \begin{enumerate}
  \item $f$ is close to $\id_{X}$.
  \item\label{fjweoifjewoijeoiejfiwjeiofjewfewfewf} For every entourage $U$ of $X$ the subset $\bigcup_{n\in \nat}(f^{n}\times f^{n})(U)$ is an entourage of $X$.
  \item\label{hckjhckhwhuiechiu} For every bounded subset $B$ of $X$ there exists an integer  $n$ such that
  $\Gamma B\cap f^{n}(X)=\emptyset$.
  \end{enumerate}
  
  We say that flasqueness of $X$ is \emph{implemented} by $f$.
   \end{ddd}
\begin{rem}\label{hckjhckhwhuiechiu2}
In Condition \ref{hckjhckhwhuiechiu} above one could require  the weaker condition 
$B\cap f^{n}({X})=\emptyset$ instead of  $\Gamma B\cap f^{n}({X})=\emptyset$.
Then much of the theory would go through, but we lose the possibility of descending the
group change functors ``$H$-completion'', ``quotient'' and ``induction'' to the motivic level.
\end{rem}

 Let $\bC$ be a cocomplete stable $\infty$-category.
 We consider {a functor}
 \[E\colon\Gamma\BC\to \bC\ .\]
 If $\cY=(Y_{i})_{i\in I}$ is a filtered family of  $\Gamma$-invariant subsets of $X$, then we set
   \begin{equation}\label{vrevkjervnjknvevvveverv}
E(\cY):=\colim_{i\in I} E(Y_{i})\ .
\end{equation} 
In this formula we consider the subsets $Y_{i}$ as $\Gamma$-bornological coarse spaces with the structures induced from $X$.
 
The set $\{0,1\}_{max,max}$ is a $\Gamma$-bornological coarse space with the trivial $\Gamma$-action.   
 
 \begin{ddd}\label{fewoijfewoifewofww45535345345}
 A \emph{$\Gamma$-equivariant $\bC$-valued coarse homology theory} is a functor
 \[E\colon\Gamma\BC\to \bC\] with the following properties:
 \begin{enumerate}
 \item\label{efjwefjewfewiofjewof}({C}oarse invariance) For all $X\in \Gamma\BC$ the projection $\{0,1\}_{max,max}\otimes  X\to X$ is sent by $E$ to an equivalence.
 \item({E}xcision) $E(\emptyset)\simeq 0$ and for every equivariant complementary pair   $(Z,\cY)$ on  a  $\Gamma$-bornological coarse space $X$ the   square  
 \[\xymatrix{E(Z\cap \cY)\ar[r]\ar[d]&E(Z)\ar[d]\\E(\cY)\ar[r]&E(X)}\]
 is  a push-out.
 \item ({Flasqueness}) If  a  $\Gamma$-bornological coarse space $X $ is  flasque, then $E(X)\simeq 0$.
 \item({u-Continuity}) For every  $\Gamma$-bornological coarse space $X $ the natural map
 \[\colim_{U\in \cC^{\Gamma}} E(X_{U})\xrightarrow{\simeq}E(X)\]
 is an equivalence (see \cref{wefoiweofiewf} for notation).
\end{enumerate}
If the group $\Gamma$ is clear from the context, then we will often just speak of an equivariant coarse homology theory.
\end{ddd}

\begin{rem}
Condition \ref{efjwefjewfewiofjewof} in the above definition is equivalent to the condition that $E$
sends equivalences (\cref{goreoigegerg}) of $\Gamma$-bornological coarse spaces to equivalences in $\bC$:

The projection $\{0,1\}_{max,max}\otimes X\to X$ is an equivalence of $\Gamma$-bornological coarse spaces. If $E$ preserves equivalences, then it sends this morphism to an equivalence.

Vice versa, if $E$ satisfies Condition \ref{efjwefjewfewiofjewof}, then it sends pairs of close maps to equivalent maps. If $f$ is an equivalence with inverse $g$ up to closeness, then $E(f\circ g)$ and $E(g\circ f)$ are  equivalent to $E(\id)$ and therefore  themselves equivalences. This implies by functoriality that $E(f)$ is an equivalence.
\end{rem}

{Let the cocomplete stable $\infty$-category $\bC$ have all small products.}
Let $(X_{i})_{i\in I}$ be a family of $\Gamma$-bornological coarse spaces. If $E$ is a $\bC$-valued equivariant coarse homology theory, then by excision for every index $i$ in $I$ we have a projection
$E(\bigsqcup_{i\in I}^{\free}X_{i})\to E(X_{i})$. The collection of these projections induces a morphism
\begin{equation}\label{iojoiergegerg}
E\big(\bigsqcup_{i\in I}^{\free}X_{i}\big)\to \prod_{i\in I} E(X_{i})
\end{equation}
 \begin{ddd}\label{rgfou894ut984t3kfnrekjf}
$E$ is called \emph{strongly additive} if \eqref{iojoiergegerg} is an equivalence for every family  $(X_{i})_{i\in I}$ of $\Gamma$-bornological coarse spaces.
\end{ddd}

Let $E$ be an equivariant coarse homology theory and let $S$ be a $\Gamma$-set.

\begin{lem}
If $E$ is strongly additive, then the twist $E(-\otimes S_{max,max})$ is strongly additive.
\end{lem}

\begin{proof} 
This follows from the fact that for every family $(X_{i})_{i\in I}$ of $\Gamma$-bornological coarse spaces we have an isomorphism
\[\big(\coprod_{i\in I}^{\free}X_{i}\big)\otimes S_{max,max}\cong \coprod_{i\in I}^{\free}(X_{i}\otimes S_{max,max})\]
of $\Gamma$-bornological coarse spaces.
\end{proof}

\section{Equivariant coarse motivic spectra}
\label{e9werlgforeg99}

\subsection{Construction}
{In this section we} define the stable $\infty$-category of coarse motives $\Gamma\Sp\cX$. This is completely analogous to \cite[Sec.~3 \& 4]{buen}. The category $\Gamma\Sp\cX$ is designed such that  equivariant $\bC$-valued coarse homology theories (\cref{fewoijfewoifewofww45535345345}) are the same as colimit-preserving functors $\Gamma\Sp\cX\to \bC$. The precise formulation is \cref{lkjiooiwoigewgewgwegfw123}.

Let $\Spc$ be the $\infty$-category of spaces, i.e., the universal presentable $\infty$-category generated by $*$.
We start with the category  \[\PSh(\Gamma\BC):=\Fun(\Gamma\BC^{op},\Spc)\] of $\Spc$-valued presheaves on $\Gamma\BC$.
Let \begin{equation}\label{huiciuhciuewcwecewcwecwecwc}
\yo\colon\Gamma\BC\to \PSh(\Gamma\BC)
\end{equation} be the Yoneda embedding. Precomposition with it induces an equivalence
\[\Fun^{\mathrm{lim}}(\PSh(\Gamma\BC)^{op},\Spc)\simeq \PSh(\Gamma\BC)\ ,\]
where $\Fun^{\mathrm{lim}}$ denotes limit-preserving  functors  \cite[Thm. 5.1.5.6]{htt} (see also \cite[Rem. 3.9]{buen}).  
We use this equivalence in order to evaluate presheaves on other presheaves.   

For  a filtered  family $\cY=(Y_{i})_{i\in I}$ of invariant subsets on some $\Gamma$-bornological coarse space~$X$ we write
\begin{equation}\label{huiciuhciuewcwecewcwecwecwc1}
\yo(\cY):=\colim_{i\in I} \yo(Y_{i})\in \PSh(\Gamma\BC)\ .
\end{equation}
For $E$ in $\PSh(\Gamma\BC)$ we set
\[E(\cY):=E(\yo(\cY))\ .\]

\begin{rem}
For a $\Gamma$-bornological coarse space $X$ we have the equivalence \[E(X)\simeq E(\yo(X))\ .\]
Furthermore we have
\begin{equation}\label{vervheoioij}
E(\cY)\simeq \lim_{i\in I}E(Y_{i})
\end{equation}
for a filtered family $\cY=(Y_{i})_{i\in I}$ of invariant subsets on $X$.
\end{rem}

Let $E$ be an object of  $\PSh(\Gamma\BC)$.
\begin{ddd}\label{ewlfjwefewoi2r543555}
We say that
$E $ \emph{satisfies descent} if
\begin{enumerate}
\item $E(\emptyset)\simeq *$, and
\item for every equivariant complementary pair $(Z,\cY)$ on a $\Gamma$-bornological coarse space $X$ the square
\begin{equation}\label{wefkewjfio32ru23r23r23r32r}
\xymatrix{E(X)\ar[r]\ar[d]&E(Z)\ar[d]\\E(\cY)\ar[r]&E(Z\cap \cY)} 
\end{equation}
is cartesian.
\end{enumerate}
Presheaves which satisfy descent are called \emph{sheaves}.
\end{ddd}

\begin{rem}
One can show that
there is a subcanonical Grothendieck topology  $\tau_{\chi}$ on $\Gamma\BC$ such that
the $\tau_{\chi}$-sheaves are exactly the presheaves which satisfy descent  for equivariant complementary pairs. Since we will not be using this fact in this paper we will omit the arguments.
\end{rem}

We let $\Sh (\Gamma\BC)$ denote the full subcategory of $ \PSh(\Gamma\BC)$ of sheaves. We can characterize sheaves as presheaves which are local
with respect to the morphisms
\begin{eqnarray} \yo(\cY)\sqcup_{\yo(Z\cap \cY)}\yo(Z)&\to& \yo(X)\label{wklwecjwjciowecjoweicjwoiec} \\
\yo(\emptyset)&\to&*\nonumber
 \ .\end{eqnarray}

\begin{rem}
In order to fix set-theoretic issues we assume that all $\Gamma$-bornological coarse spaces and the index sets $I$ for the big families belong to some Grothendieck universe of small sets. The class of local objects is then generated by a small set of morphisms. The category $\Sh(\Gamma\BC)$ then belongs to a bigger universe.
\end{rem}

 We have a sheafification adjunction   \[\PSh (\Gamma\BC)\leftrightarrows \Sh(\Gamma\BC)\colon inclusion\ .\]

\begin{rem}
For a $\Gamma$-bornological coarse space $X$ the presheaf $\yo(X)$   is a compact object of $\PSh (\Gamma\BC)$.
If $\cY$ is a big family, then $\yo(\cY)$ is an in general infinite colimit of compact objects and hence not compact anymore.  Consequently, the morphisms \eqref{wklwecjwjciowecjoweicjwoiec} are not morphisms between compact objects. The localization $\Sh(\Gamma\BC)$ is therefore a presentable $\infty$-category, but it is not compactly generated.
\end{rem}

Let $E$ be an object of $\Sh (\Gamma\BC)$.
\begin{ddd}\label{iojoijfoiwejfowejffewf98u98u98uu94234234234} \mbox{}\begin{enumerate}
\item 
 $E$ is \emph{coarsely invariant}  if it is local with respect to the morphisms
 \[\yo (\{0,1\}_{max,max}\otimes X)\to \yo(X)\]  induced by the projection for every $\Gamma$-bornological coarse space $X$.  
 \item  $E$ \emph{vanishes on flasques} if it is local with respect to  the morphisms  
\[{\yo(\emptyset)} \to \yo(X)\]  for every 
flasque $\Gamma$-bornological coarse space $X$.
\item  $E$  is \emph{$u$-continuous}  if it is local
for the morphism
\[{\colim_{U\in \cC^{\Gamma}}\yo(X_{U})} \to \yo(X)\]
for every $\Gamma$-bornological coarse space $X$ (where $\cC^{\Gamma}$ denotes the invariant entourages of the space $X$). 
\end{enumerate}
The above notions are just the equivariant analogues of the corresponding notions from \cite[Sec.~3]{buen}.
\end{ddd}

\begin{ddd}\label{fewowijewoifjoewfjowefwefewfew}
We define the $\infty$-category of \emph{$\Gamma$-equivariant motivic coarse spaces} $\Gamma\Spc\cX$ as the full localizing subcategory of $\Sh(\Gamma\BC)$ of 
coarsely invariant, $u$-continuous sheaves which vanish on flasques. \end{ddd} 

The locality condition is generated by a small set of morphisms. Therefore
we have a localization adjunction
\begin{equation}\label{erglkhoi3hto34t34t43t34t3t}
\cL\colon\PSh(\Gamma\BC)\leftrightarrows \Gamma\Spc\cX: inclusion\ .
\end{equation}
 We define
\[\Yo:=\cL\circ \yo\colon\Gamma\BC\to \Gamma\Spc\cX\ .\]

The $\infty$-category $\Gamma\Spc\cX$ is a presentable $\infty$-category.

\begin{ddd}
We define the category of \emph{equivariant motivic coarse spectra} as the stabilization
 \[\Gamma\Sp\cX:=\Gamma \Spc\cX_{*}[\Sigma^{-1}]\]
 in the realm of presentable $\infty$-categories.
\end{ddd}
Then $\Gamma \Sp\cX$ is a stable presentable $\infty$-category which fits into an adjunction
\[\Sigma^{mot}_{+}\colon\Gamma\Spc\cX\leftrightarrows\Gamma \Sp\cX:\Omega^{mot}\ .\]
We further define the Yoneda functor
\[\Yo^{s}:=\Sigma^{mot}_{+}\circ \Yo\colon\Gamma\BC\to \Gamma\Sp\cX\ .\]

\begin{ddd}
We call $\Yo^{s}\colon\Gamma\BC\to \Gamma\Sp\cX$ the universal equivariant coarse homology theory.
\end{ddd}

For a $\Gamma$-bornological coarse space $X$ we consider the object $\Yo^{s}(X)$ of  $\Gamma\Sp\cX$ as the motive of $X$.
 
Let $\bC$ be a cocomplete stable $\infty$-category. Let $\Gamma\CoarseHomologyTheories_{\bC}$  denote the full subcategory of $\Fun(\Gamma\BC,\bC)$ of functors which are $\Gamma$-equivariant $\bC$-valued coarse homology theories in the sense of \cref{fewoijfewoifewofww45535345345}. By $\Fun^{\mathrm{colim}}(\Gamma\Sp\cX,\bC)$ we denote the full subcategory of $\Fun(\Gamma\Sp\cX,\bC)$ of colimit preserving functors.

The construction of $\Gamma\Sp\cX$ has the following consequence (see \cite[Cor.~4.6]{buen}):
\begin{kor}\label{lkjiooiwoigewgewgwegfw123}
The functor $\Yo^{s}$ is a $\Gamma\Sp\cX$-valued equivariant coarse homology theory. Furthermore,
precomposition with $\Yo^{s}$ induces an equivalence of $\infty$-categories
\[\Fun^{\mathrm{colim}}(\Gamma\Sp\cX,\bC)\to \Gamma\CoarseHomologyTheories_{\bC}\ .\]   \end{kor}

\subsection{Properties}\label{riojreoirjgoiggoiergergege}

If $\cY=(Y_{i})_{i\in I}$ is an equivariant 
big family on a $\Gamma$-bornological coarse space $X$, then we define the equivariant motivic coarse spectrum
\begin{equation}\label{wefweew254}
\Yo^{s}(\cY):=\Sigma_+^{mot}\circ\cL\circ \yo(\cY)\ .
\end{equation}
{Note that we have $\Sigma_+^{mot}\circ\cL\circ \yo(\cY)\simeq \colim_{i\in I} \Yo^{s}(Y_{i})$.}

We  will use the notation
\begin{equation}\label{fkhwiufuiz823zr824234242424234234234234}
\Yo^{s}(X,\cY):=\Cofib(\Yo^{s}(\cY)\to \Yo^{s}(X))\ .
\end{equation}

By construction we have the following {properties}:
\begin{kor}\label{kjeflwfjewofewuf98ewuf98u798798234234324324343}\mbox{}
\begin{enumerate}
\item\label{ifjweifjewiojwefw231} We have a fiber sequence
\[\Yo^{s}(\cY)\to \Yo^{s}(X)\to \Yo^{s}(X,\cY)\to \Sigma \Yo^{s}(\cY)\]
\item\label{fwejiofjweiofuewofewf234} For an equivariant  complementary pair $(Z,\cY)$ on $X$  the natural morphism
\[\Yo^{s}(Z,Z\cap \cY)\to \Yo^{s}(X,\cY)\] is an equivalence.
\item \label{iweufhf89wfu89ewfew245} If $X\to X^{\prime}$ is an equivalence of $\Gamma$-bornological coarse spaces, then {the induced morphism} $\Yo^{s}(X)\to \Yo^{s}(X^{\prime})$ is an equivalence in $\Gamma\Sp\cX$.
\item\label{ofjewofefoewiufewfiewf09i23423434234} If $X$ is a flasque $\Gamma$-bornological coarse space, then $\Yo^{s}(X)\simeq 0$.
\item \label{efwijfiewfoiefe3u40934332r} For every  $\Gamma$-bornological coarse space $X$ with coarse structure $\cC$    {the natural map} \[\colim_{U\in \cC^{\Gamma}} \Yo^{s}(X_{U})\xrightarrow{\simeq}\Yo^{s}(X)\]
is an equivalence.\end{enumerate}
\end{kor}

Let $X$ be a $\Gamma$-bornological coarse space and $A$  be a $\Gamma$-invariant subset of $X$. Recall that  $\{A\}$ denotes the equivariant big family generated by $A$ (\cref{pokpogergeerge}).
\begin{kor}\label{239023fsdsf}
If $A$ is  nice, then we the natural map
$\Yo^{s}(A)\to \Yo^{s}(\{A\})$ is an equivalence.
\end{kor}

\begin{proof}
Since $A$ is nice, 
for every invariant  entourage $U$  of $X$  the inclusion $A\to U[A]$ is an equivalence. The assertion now follows since
$\Yo^{s}$ preserves equivalences.
\end{proof}

Let $X$ be a $\Gamma$-bornological coarse space and $Y,Z$ be invariant subsets such that $Y\cup Z=X$.
\begin{ddd}\label{dddi34546}
We say that  $(Y,Z)$ is a \emph{coarsely excisive pair}, if:
\begin{enumerate}
\item For every
entourage
$U$ of $X$ there exists an entourage $W$ of $X$ such that \[U[Y]\cap U[Z]\subseteq W[Y\cap Z]\ .\]
\item  \label{wfopkwopfeef}There exists a cofinal  set of invariant entourages $V$ of $X$ such that
$V[Y]\cap Z$ is nice.
\end{enumerate}
Note that Condition~\ref{wfopkwopfeef} is a new aspect of the equivariant theory.
\end{ddd}

Let $X$ be a $\Gamma$-bornological coarse space and $Y,Z$ be invariant subsets such that $Y\cup Z=X$.
\begin{kor}\label{cbdbbg}
{If $(Y,Z)$ is a coarsely excisive pair, then we have a cocartesian square
\[\xymatrix{\Yo^{s}(Y\cap Z)\ar[r]\ar[d]&\Yo^{s}(Z)\ar[d]\\\Yo^{s}(Y)\ar[r]&\Yo^{s}(X)}\]}
\end{kor}

\begin{proof}
The proof of \cite[Lem.~3.38]{buen} goes through literally. In the proof we need the equivalence 
\[\Yo^{s}(V[Y]\cap Z)\simeq \Yo^{s}(\{V[Y]\cap Z\})\]
for sufficiently large invariant entourages $V$ of $X$. This is ensured by Condition \ref{wfopkwopfeef} in the \cref{dddi34546} of coarse excisiveness.
\end{proof}

Let $X$ be a $\Gamma$-bornological coarse space.
Given two bornological and $\Gamma$-invariant maps $p=(p_{-},p_{+})$ with  $p_{-}\colon X\to (-\infty,0]$ and $p_{+}\colon X\to [0,\infty)$ we can form the  coarse cylinder $I_{p}X$ as in the non-equivariant case {\cite[Sec.~4.3]{buen}}. With its natural $\Gamma$-action it is  a $\Gamma$-bornological coarse space. The projection $I_{p}X\to X$ is a morphism.
We will call it an equivariant cylinder in order to stress that the datum $p$ was $\Gamma$-invariant.

Let $X$ be a $\Gamma$-bornological coarse space and $I_{p}X$ be an equivariant coarse cylinder.

\begin{kor}
The projection $I_{p}X\to X$ induces an equivalence
$\Yo^{s}(I_{p}X)\to \Yo^{s}(X)$. 
\end{kor}

\begin{proof}
We observe that the proof of \cite[Prop.~4.16]{buen} goes through.  At all places in the argument where \cref{239023fsdsf} is used the corresponding subset is nice, see \cref{fuhewiufhuewffwefewf}.
\end{proof}

We say that two morphisms $f_{+},f_{-}\colon X\to X^{\prime}$ between $\Gamma$-bornological coarse spaces are \emph{homotopic} if there exists a cylinder $I_{p}X$ such that $p_{\pm}$ are $\Gamma$-invariant, bornological and in addition controlled, and if there exists a morphism $h\colon I_{p}X\to X^{\prime}$ such that $f_{\pm}=h\circ i_{\pm}$. This leads to an extension of the notion of coarse invariance.

\begin{kor}
If $f_+$ and $f_-$ are homotopic, then $\Yo^s(f_+)$ and $\Yo^s(f_-)$ are equivalent.
\end{kor}

\subsection{Symmetric monoidal structure}
\label{secni2039}

{Recall that} the category
$\Gamma\BC$ has a symmetric monoidal structure
\begin{equation}
\label{eq_43t5rgeseses34t}
- \otimes  -\colon\Gamma\BC\times \Gamma\BC\to \Gamma\BC
\end{equation}
with tensor unit~$*$.

\begin{lem}\label{lemjine2}
$\Gamma\Sp\cX$ has an induced closed symmetric monoidal structure $\otimes $ such that the functor $\Yo^s \colon \Gamma\BC \to \Gamma\Sp\cX$ is symmetric monoidal. The functor $\otimes$ commutes with colimits in each variable separately.

\end{lem}

\begin{proof}
We get an induced symmetric monoidal structure on $\PSh(\Gamma\BC)$ by the Day convolution product. The unit is given by $\yo(*)$, the Yoneda embedding is a strong symmetric monoidal functor, and $\PSh(\Gamma\BC)$ is closed symmetric monoidal.

For a $\Gamma$-bornological coarse space $Q$ the    functor
\[-\otimes  Q\colon\Gamma\BC\to \Gamma\BC\]
maps big families to big families and complementary pairs to complementary pairs. So the monoidal structure of $\PSh(\Gamma\BC)$ restricts to one on $\Sh(\Gamma\BC)$ and the sheafification adjunction is a symmetric monoidal adjunction.

The functor $-\otimes  Q$ furthermore respects closeness of morphisms and therefore coarse equivalences, and it respects flasqueness and $u$-continuity. So we get an induced  symmetric monoidal structure on $\Gamma\Spc\cX$ and $\Yo \colon \Gamma\BC \to \Gamma\Spc\cX$ is symmetric monoidal.

Since $\Gamma\Spc\cX$ is presentable, we can equip its stabilization $\Gamma\Sp\cX$ with a unique symmetric monoidal structure $\otimes$ such that stabilization $\Sigma^{mot}_{+} \colon \Gamma\Spc\cX \to \Gamma\Sp\cX$ is symmetric monoidal \cite[Thm.~5.1]{ggm}. 

It follows from the construction that  $\otimes$ commutes with colimits in each variable separately. 
\end{proof}

\subsection{Strong version}
\label{secuiwe23g3r}

In this section we discuss an additional property (strongness) which an equivariant coarse homology theory can have. Another condition (continuity) will be discussed in \cref{oiefweiofwfewfiewfewfewf}.

By definition, a  flasque $\Gamma$-bornological coarse space $X$ admits a morphism $f\colon X\to X$ satisfying the conditions listed in \cref{hckjhckhwhuiechiu1}.  
The first condition is the condition that $f$ is close to the identity.  This fact is usually used in order to deduce that $\Yo^{s}(f)\simeq \id_{\Yo^{s}(X)}$. In the following we will  use this weaker condition in order to define a more general notion of flasqueness.

\begin{ddd}\label{iogegergeger}
 $X$ is called \emph{weakly flasque} if it admits a morphism $f\colon X\to X$  satisfying \begin{enumerate}
\item\label{iogegergeger1}    $\Yo^{s}(f)\simeq  \id_{\Yo^{s}(X)}$. 
\item For every entourage $U$ of $X$ the subset $\bigcup_{n\in \nat} (f^{n}\times f^{n})(U)$ is again an entourage of~$X$.
\item For every bounded subset $B$ of $X$ there exists an integer $n$ such that ${\Gamma}B\cap f^{n}(X)=\emptyset$.
\end{enumerate}
We say that $f$ \emph{implements weak flasqueness} of $X$.
\end{ddd}

Let $\bC$ be a {cocomplete} stable $\infty$-category and consider a $\bC$-valued equivariant coarse homology theory $E$.
\begin{ddd}\label{foijofifwefwefewfw}
$E$ is called \emph{strong}  if $E(X)\simeq 0$ for all   weakly flasque $\Gamma$-bornological coarse spaces $X$.
\end{ddd}

{Let us incorporate now the condition of strongness on the motivic level.}

\begin{ddd}\label{def:wfl}
We define the version  of equivariant motivic spectra
$ \Gamma\Sp\cX_{\wfl}$ as the localization  of the category $\Gamma\Sp\cX$ at the  set of morphisms $0\to \Yo^{s}(X)$ for all weakly flasque $\Gamma$-bornological coarse spaces $X$.
\end{ddd}

The corresponding Yoneda functor   is denoted by
\[\Yo^{s}_{\wfl}\colon \Gamma\BC\to \Gamma\Sp\cX_{\wfl}\ .\]
 
We consider now the $\infty$-category of strong $\Gamma$-equivariant coarse homology theories. The construction of $\Gamma\Sp\cX_{\wfl}$ has the following immediate consequence:
\begin{kor}\label{lkjiooiwoi123gewgewgwegfw1}
The functor $\Yo_{\wfl}^{s}$ is a $\Gamma\Sp\cX_{\wfl}$-valued equivariant coarse homology theory. Furthermore,
precomposition with $\Yo_{\wfl}^{s}$ induces an equivalence of $\infty$-categories
\[\Fun^{\mathrm{colim}}(\Gamma\Sp\cX_{\wfl},\bC)\to\mathbf{strong} \Gamma\CoarseHomologyTheories_{\bC}\ .\]
\end{kor}

\section{Continuity}\label{oiefweiofwfewfiewfewfewf}

{The purpose of this section is to introduce the notion of continuity for equivariant coarse homology theories. This property will be crucially needed in \cref{reoirgioerigueog3453455}.  We will first introduce the notion of  trapping exhaustions in \cref{secjk223}. \cref{sec9000} contains the actual definition of continuous equivariant coarse homology theories, and \cref{sec02303223} incorporates continuity motivically. In the last \cref{secjk1222} we will show how one can  force   continuity for an equivariant coarse homology theory.}

\subsection{Trapping exhaustions}
\label{secjk223}

 In this section we will introduce the notion of a trapping exhaustion of a $\Gamma$-bornological space and discuss some examples and basic properties of this notion. We will also introduce the stronger notion of a co-$\Gamma$-bounded exhaustion.

Let $X$ be a bornological space and let $F$ be a subset of $X$.

\begin{ddd}
The  subset $F$ is called \emph{locally finite} if $B\cap F$ is finite for every  bounded subset $B$ of $X$.
\end{ddd}

\begin{ex}
Every finite subset of $X$ is locally finite. 
\end{ex}

\begin{ex}
If $X$ has the minimal  bornology on $X$, i.e., a subset is bounded if and only if it is finite, then every subset of $X$ is locally finite.
\end{ex}

\begin{ex}
If $X$ has the maximal bornology on $X$, i.e., every subset of $X$ is bounded, then the locally finite subsets of $X$ are exactly the finite subsets.
\end{ex}

Let $f\colon X\to X^{\prime}$ be a proper map between bornological spaces.
Let $F$ be a subset of $X$.
\begin{lem}\label{lirjfiowjfoifjoejfewfewf}
If $F$ is locally finite, then $f(F)$ is locally finite.
\end{lem}

\begin{proof}
We use the relation $f(F) \cap B \subseteq f(F \cap f^{-1}(B))$.
\end{proof}

We consider in the following a $\Gamma$-bornological  space $X$ and  a filtered family of invariant subsets $\cY=(Y_{i})_{i\in I}$.

\begin{ddd} \label{fiuwefziufewfwefwefewwefw}
The family $\cY$ is called a \emph{trapping exhaustion} if
for every locally finite, invariant subset $F$ of $X$ there exists $i$ in $I$ such that $F\subseteq Y_{i}$.
\end{ddd}

\begin{ex}\label{egriohreiohgiuhiuerhggre}
The family consisting of all {locally} finite, invariant subsets is a trapping exhaustion.

It might happen that a $\Gamma$-bornological coarse space does note admit any non-empty invariant locally finite subset. Consider e.g. $\Gamma$ with the maximal bornology.
In this case the empty family is a trapping exhaustion.
\end{ex}

{In the following we will introduce a particular kind of trapping exhaustions which we call co-$\Gamma$-bounded exhaustions.}

We consider a $\Gamma$-bornological  space $X$ and  a filtered family of invariant subsets $\cY = (Y_{i})_{i\in I}$.
We use the notation {and terminology} introduced in \cref{goihreiugiuuz34r34t3t3t}.

\begin{ddd}\label{dddjk2332}
The family $\cY$ is called a \emph{co-$\Gamma$-bounded exhaustion} if
\begin{enumerate}
\item $\cY$ is an exhaustion, i.e., $\bigcup_{i\in I}Y_{i}=X$, and
\item $\cY$ is co-$\Gamma$-bounded, i.e., there exists $i$ in $I$ such that $X\setminus Y_{i}$ is $\Gamma$-bounded.\qedhere
\end{enumerate}
\end{ddd}

We consider $\Gamma$-bornological   spaces $X$ and $Z$ and a filtered family $\cY:=(Y_{i})_{i\in I}$ of invariant subsets of $X$.
In the following we denote by $Z\otimes X$ the $\Gamma$-bornological space whose bornology is generated by
products $A\times B$ for all bounded subsets $A$ of $Z$ and $B$ of $X$.

\begin{lem}\label{greughiuui43gt3gg}
If $Z$ is bounded  and $\cY$ is a  co-$\Gamma$-bounded ({resp.,} trapping) exhaustion of $X$, $(Z\times Y_{i})_{i\in I}$ is a  co-$\Gamma$-bounded ({resp.,} trapping) exhaustion of $Z\otimes  X$.
\end{lem}

\begin{proof}
The co-$\Gamma$-bounded case is {straightforward.}

For the trapping case assume that $F$ is an invariant, locally finite subset of $Z\otimes  X$. Since the projection $p\colon Z\otimes  X\to X$ is proper, by  \cref{lirjfiowjfoifjoejfewfewf}  the subset $p(F)$ of $X$ is locally finite. Hence there exists $i$ in $I$ such that $p(F)\subseteq Y_{i}$, and therefore $F\subseteq Z\times Y_{i}$. 
\end{proof}

\begin{lem}\label{lemkjnsd222}
If $\cY$ is a co-$\Gamma$-bounded exhaustion of a $\Gamma$-bornological space $X$ then it is a trapping exhaustion.
\end{lem}

\begin{proof}
Let $F$   be an invariant, locally finite subset of $X$. Since $\cY$ is a {co-$\Gamma$-bounded} exhaustion of $X$ there exists an index $i$ in $I$ and a bounded subset $B$ of $X$ such that ${\Gamma B \cup Y_{i}} = X$. Since the union of  $\cY$  is $X$ and {$F\cap B$} is finite  there exists an index $j$ in $I$ such that $i\le j$ and {$F\cap B \subseteq Y_{j}$.} Since {$Y_{j}$} is invariant, then also {$F\subseteq Y_{j}$.}
\end{proof}

\begin{ex}\label{fwoiejoweijfoiwefewfewfew}
Let $Z$ be a  $\Gamma$-bounded $\Gamma$-bornological space and let $\cZ=(Z_{i})_{i\in I}$ be an exhaustion by not necessarily $\Gamma$-invariant subsets. 
For every $i$ in $I$ we consider the subset
\[D_{i}:=\Gamma(Z_{i}\times \{1\})\] of $  Z\times  \Gamma$.
We consider the $\Gamma$-bornological space $Z\otimes  \Gamma$,
where   $\Gamma$  has any $\Gamma$-invariant bornology.
The family
\[\cD:=(D_{i})_{i\in I}\]
is  a co-$\Gamma$-bounded (and hence trapping) exhaustion of $Z\otimes \Gamma $.
\end{ex}

We consider  $[0,\infty)$ as a $\Gamma$-bornological space with the trivial action and  the bornology generated by the subsets $[0,n]$ for all integers $n$.
In the following we will construct an interesting trapping exhaustion of the $\Gamma$-bornological space
\[[0,\infty)\otimes  Z\otimes  \Gamma\]
which will play an important role in \cref{reoirgioerigueog3453455}.

Note that in general $([0,\infty)\times D_{i})_{i\in I}$ is not trapping.

We consider the set of functions $I^{\nat}$ with its partial order induced from $I$. Then the partially ordered set
$I^{\nat}$ is filtered. For a function $\kappa$  in $I^{\nat}$ we define the set
\[Y_{\kappa}:=\bigcup_{n\in \nat}{[n-1,n]}\times D_{\kappa{(n)}}\ .\]

\begin{lem}\label{fkwehfueuiewiuu423}
If $Z$ is $\Gamma$-bounded, then 
$\cY:=(Y_{\kappa})_{\kappa\in I^{\nat}}$ is a trapping exhaustion of the space $[0,\infty)\otimes  Z\otimes  \Gamma$.
\end{lem}
Note that the exhaustion $\cY$ is not co-$\Gamma$-bounded.

\begin{proof}
The members of $\cY$ are $\Gamma$-invariant subsets.
For every integer $n $ the family given by
$([{n-1},n]\times D_{i})_{i\in I}$ is a co-$\Gamma$-bounded exhaustion of $[{n-1},n]\otimes Z\otimes \Gamma$, since $[{n-1},n]$ is bounded and $(D_{i})_{i\in I}$ is a co-$\Gamma$-bounded {exhaustion} of $Z\otimes \Gamma$.
So it is trapping. 

Let $F$ be a $\Gamma$-invariant locally finite  subset of $ [0,\infty)\times  Z\times  \Gamma$.
 Then  $F\cap [{n-1},n]\times Z\times \Gamma$ is also locally finite and $\Gamma$-invariant. For every integer $n $ we can choose  $\kappa(n)$ in $I$ such that $(F\cap [{n-1},n]\times Z\times \Gamma) \subseteq [{n-1},n]\times D_{\kappa(n)}$.
 This describes a function   $\kappa $ in $I^{\nat}$ such that by  construction
 $F\subseteq Y_{\kappa}$.
\end{proof}

Let $f\colon X^{\prime}\to X$ be a proper map between $\Gamma$-bornological   spaces.
\begin{lem}
If $\cY$ is a co-$\Gamma$-bounded  (resp., trapping)  exhaustion of $X$, then $f^{-1}\cY$ is a co-$\Gamma$-bounded (resp., trapping) exhaustion of $X^{\prime}$.
\end{lem}

\begin{proof}
The co-$\Gamma$-bounded case is a direct consequence of properness and the fact that forming preimages
commutes with forming complements.

For the trapping case one uses in addition \cref{lirjfiowjfoifjoejfewfewf}.
\end{proof}

\subsection{Continuous equivariant coarse homology theories}
\label{sec9000}

In this section we will introduce an additional continuity condition on an equivariant coarse homology theory. We then verify that a continuous equivariant coarse homology theory preserves coproducts.

\begin{rem}
In \cref{seci232323} we show continuity of equivariant coarse ordinary homology theory and in \cref{prop:coarseK.continuous}  continuity of equivariant coarse algebraic $K$-homology. In \cref{gfoihwjgoij32rtgergergerge} we describe a version of the coarse algebraic $K$-homology theory which is not continuous.
\end{rem}

We can extend the notions of locally finite subsets and trapping / co-$\Gamma$-bounded exhaustions to $\Gamma$-bornological coarse spaces by just considering the underlying $\Gamma$-bornological spaces.

Let $\bC$ be a stable cocomplete $\infty$-category and \[E\colon\Gamma\BC\to \bC\] an equivariant coarse homology theory.
We use the convention \eqref{vrevkjervnjknvevvveverv} for the evaluation $E(\cY)$ on a  filtered family $\cY$ of invariant subsets of a $\Gamma$-bornological coarse space. We have a natural  morphism
\begin{equation}\label{fwkjfhweiufhiewfwefewf}
E(\cY)\to E(X)\ .
\end{equation}

\begin{ddd}\label{fweoiuiowefwefeve}
 $E$ is called \emph{continuous} if
for every  trapping exhaustion  $\cY$ of some $\Gamma$-bornological coarse space $X$ the morphism  \eqref{fwkjfhweiufhiewfwefewf} is an equivalence.
\end{ddd}

\begin{rem}\label{feoijewoif234345}
A continuous equivariant coarse homology theory $E$ is determined by its values on {locally finite, invariant} spaces. More precisely, let $\cF(X)$ be the filtered partially ordered set of {locally finite, invariant} subsets of $X$. In view of \cref{egriohreiohgiuhiuerhggre} we have a trapping exhaustion $\cY:=(F)_{F\in \cF(X)}$ of $X$. So we get
\[E(X)\simeq E(\cY)=  \colim_{F\in \cF(X)} E(F)\]
showing the claim.

If $\cF(X)$ is empty, i.e., $X$ does not admit non-empty locally finite $\Gamma$-invariant subsets (see \cref{egriohreiohgiuhiuerhggre}), and $E$ is continuous, then $E(X)\simeq 0$.
\end{rem}

By excision an equivariant coarse homology theory preserves coproducts of finite families of $\Gamma$-bornological coarse spaces. By the following lemma, for a continuous  equivariant coarse homology theory, {we can drop the word \emph{finite}}.
 
\begin{lem}\label{foifjoijoiiofiuewoifuoweffewf}
A continuous equivariant coarse homology theory preserves coproducts.
\end{lem}

\begin{proof}
Let $(X_{i})_{i\in I}$ be a family of $\Gamma$-bornological coarse spaces.
We must show that the natural map
\[\bigoplus_{i\in I} E(X_{i})\to E \big( \coprod_{i\in I} X_{i} \big)\]
is an equivalence.
All invariant subsets of $\coprod_{i\in I} X_{i}$ have the
induced bornological coarse structures.

{Consider the following diagram, where the horizontal maps are equivalences by continuity and where $\cF(-)$ is the trapping exhaustion consisting of all locally finite, invariant subsets.
\[\xymatrix{
\bigoplus_{i\in I}\colim_{F_i\in\cF(X_i)}E(F_i)\ar[d]_{!}\ar[r]^-\simeq& \bigoplus_{i\in I}E(X_i)\ar[d]\\
\colim_{F\in\cF(\coprod_{i\in I}X_i)}E(F)\ar[r]^-\simeq&E(\coprod_{i\in I}X_i)	
}\]}
	
It remains to show that the map marked with ! is an equivalence. The bornology of the coproduct is described in \cite[Lemma 2.24]{buen}.
 A subset of the coproduct is bounded if and only if its intersection with $X_{i}$ for every $i$ in $I$ is bounded. This implies that  for
 every invariant, locally finite subset $F$ of $ \coprod_{i\in I} X_{i}$ there exists a minimal finite subset $J(F)$ of $ I$ such that $F\subseteq \coprod_{i\in J{(F)}} X_{i}$.  We write $F_{i}:=F\cap X_{i}$. Then
\[F=\bigcup_{i\in J(F)}F_{i}\] is a finite, coarsely disjoint decomposition. 
Hence
\[E(F)\simeq \bigoplus_{i\in J(F)} E(F_{i})\simeq \bigoplus_{i\in I}E(F_i)\ .\]
That the map marked with ! is an equivalence now follows from the equivalence
\[\bigoplus_{i\in I}\colim_{F_i\in \cF(X_i)}E(F_i)\simeq \colim_{(F_i)_i\in \cF(\coprod_{i\in I}X_i)}\bigoplus_{i\in I}E(F_i)\ .\qedhere\]
\end{proof}

\subsection{Continuous coarse motives}
\label{sec02303223}

In this section we will explain how to incorporate continuity on the motivic level and we will discuss basic properties of this procedure.
Recall the   Yoneda embedding \eqref{huiciuhciuewcwecewcwecwecwc}  and  the notation \eqref{huiciuhciuewcwecewcwecwecwc1}.
We have a natural morphism 
\begin{equation}\label{weflkewjfoi243r54}
\yo(\cY)\to \yo(X)\ .
\end{equation}

Let $E$ be an object of $\PSh(\Gamma\BC)$.
\begin{ddd}
We call $E$ \emph{continuous} if it is local with {respect to} the morphisms  \eqref{weflkewjfoi243r54}   for all  trapping exhaustions $\cY$ of $\Gamma$-bornological {coarse} spaces $X$.
\end{ddd}

\begin{rem}
Let $E$ be an object of   $\PSh(\Gamma\BC)$ and recall \eqref{vervheoioij}.
 The collection of restriction morphisms $E(X)\to E(Y_{i})$ for all $i$ in $I$ induce a natural morphism
\begin{equation}\label{eflkwfoewofjewfoiewjf}
E(X)\to E(\cY)\ .
\end{equation}

Then $E$ is continuous if and only if the morphism \eqref{eflkwfoewofjewfoiewjf} is an equivalence   for  every  trapping exhaustion  $\cY$ of a $\Gamma$-bornological {coarse} space  $X$.
\end{rem}

We now incorporate continuity on the motivic level by adding this relation to the list in \cref{fewowijewoifjoewfjowefwefewfew}.

\begin{ddd}\label{fewowijewoifjoewfjowefwefewfew22} We define the $\infty$-category of continuous $\Gamma$-equivariant motivic coarse spaces $\Gamma\Spc\cX_{c}$ to be the full localizing subcategory of $\Sh(\Gamma\BC)$ of 
  coarsely invariant, continuous and $u$-continuous sheaves which vanish on flasques.
  \end{ddd} 

The locality condition is generated by a small set of morphisms. Therefore
we have a localizing adjunction
\begin{equation}\label{erglkho83453t34t3t}
\cL_{c}\colon\PSh(\Gamma\BC)\leftrightarrows \Gamma\Spc\cX_{c}: inclusion\ .
\end{equation}
 We define
\[\Yo_{c}:=\cL_{c}\circ \yo\colon\Gamma\BC\to \Gamma\Spc\cX_{c}\ .\]
 
 We furthermore have a localizing adjunction
\[C\colon\Gamma\Spc\cX\leftrightarrows \Gamma\Spc\cX_{c}: inclusion\]
and the relations
\[\cL_{c}\simeq C\circ  \cL\ , \quad  \Yo_{c}\simeq C\circ \Yo\ .\]
where $\cL$ is as in \eqref{erglkhoi3hto34t34t43t34t3t}.

The $\infty$-category $\Gamma\Spc\cX$ is a presentable $\infty$-category.

\begin{ddd}\label{orjoeggregrege}
We define the category of continuous equivariant motivic coarse spectra as the stabilization
 \[\Gamma\Sp\cX_{c}:=\Gamma \Spc\cX_{c,*}[\Sigma^{-1}]\]
 in the realm of presentable $\infty$-categories.
\end{ddd}
Then $\Gamma \Sp\cX_{c}$ is a stable presentable $\infty$-category which fits into an adjunction
\[\Sigma^{mot}_{c,+}\colon\Gamma\Spc\cX_{c}\leftrightarrows {\Gamma \Sp\cX_{c}}:\Omega^{mot}_{c}\ .\]
We further define the following stable continuous version of the Yoneda functor:
\begin{equation}\label{jhjhfkjhewiufhwifuwefwefewfwef}
\Yo^{s}_{c}:=\Sigma^{mot}_{c,+}\circ \Yo_{c}\colon\Gamma\BC\to \Gamma\Sp\cX_{c}\ .
\end{equation} 

We obtain the functor $C^{s}$ in the following commuting square from the universal property of the stabilization:
\begin{equation}\label{oihiuhfiuwheiuewf23}
\xymatrix{\Gamma\Spc\cX\ar[r]^{\Sigma^{mot}_{+}}\ar[d]^{C}&\Gamma\Sp\cX\ar[d]^{C^{s}}\\
\Gamma\Spc\cX_{c}\ar[r]^{\Sigma^{mot}_{c,+}}&\Gamma\Sp\cX_{c}}
\end{equation}
We furthermore have the relation
\begin{equation}\label{eq3t4rger435}
\Yo^{s}_{c}\simeq C^{s}\circ \Yo^{s}\ .
\end{equation}

 Let $\bC$ be a cocomplete stable $\infty$-category.
We let $\mathbf{Cont}\Gamma\CoarseHomologyTheories_{\bC}$  denote the full subcategory of $\Fun(\Gamma\BC,\bC)$ of functors which are continuous $\Gamma$-equivariant coarse homology theories in the sense of \cref{fewoijfewoifewofww45535345345}.

The construction of $\Gamma\Sp\cX_{c}$ has the following immediate consequence:
\begin{kor}\label{lkjiooiwoigewgewgwegfw}
Precomposition with $\Yo_{c}^{s}$ induces an equivalence of $\infty$-catgeories
\[\Fun^{\mathrm{colim}}(\Gamma\Sp\cX_{c},\bC)\to \mathbf{Cont}\Gamma\CoarseHomologyTheories_{\bC}\ .\]
\end{kor}

The Yoneda functor $\Yo^{s}_{c}$ has all the properties listed in \cref{riojreoirjgoiggoiergergege}.
In addition it satisfies:

\begin{kor} We have
\[\Yo^{s}_{c}(\cY)\simeq \Yo_{c}^{s}(X)\] for every trapping exhaustion $\cY$ of a $\Gamma$-bornological coarse space $X$.
\end{kor}

So in particular, $\Yo^{s}_{c}(X)$ is determined by the collection of invariant, locally finite subsets of $X$:
\[\Yo^{s}_{c}(X)\simeq \colim_{F\in  \cF(X)}\Yo^{s}_{c}(F)\ .\]

Note that $\Yo^{s}_{c}$ is a $\Gamma\Sp\cX_{c}$-valued   continuous $\Gamma$-equivariant {coarse} homology theory. Hence \cref{foifjoijoiiofiuewoifuoweffewf} implies:
\begin{kor}
The functor $\Yo^{s}_{c}\colon\Gamma\BC\to \Gamma\Spc\cX_{c}$ preserves coproducts.
\end{kor}

If $\cY=(Y_{i})_{i\in I}$ is a co-$\Gamma$-bounded (or trapping, respectively) exhaustion of a $\Gamma$-bornological space $X$ and $Z$ is a second $\Gamma$-bornological coarse space, then $\cY\times Z:=(Y_{i}\times Z)_{i\in I}$ is not necessarily  trapping   in $X\otimes Z$. As a consequence the symmetric monoidal structure $\otimes $ does not descend to continuous motivic coarse spectra.
But if $Z$ is bounded, then $\cY\times Z$ is again a trapping exhaustion of $X\otimes Z$ by \cref{greughiuui43gt3gg}.  We can conclude:

\begin{kor}\label{fliio2fo2fjoijewfwefewfw}
If $Z$ is a bounded $\Gamma$-bornological coarse space, then the functor
\[-\otimes Z\colon\Gamma\BC\to \Gamma\BC\] descends to a functor
\[-\otimes^{mot}Z\colon\Gamma\Sp\cX_{c}\to \Gamma\Sp\cX_{c}\] such that \begin{equation}\label{f3ih3iou4z89z8zt}
\Yo^{s}_{c}(X\otimes Z)\simeq \Yo^{s}_{c}(X)\otimes^{mot} Z\ .
\end{equation}
 \end{kor}

\subsection{Forcing continuity}
\label{secjk1222}

To every $\bC$-valued equivariant coarse homology theory $E$ we can naturally associate a continuous version $E_{cont}$. This is actually the best approximation of $E$ by some continuous $\bC$-valued equivariant coarse homology theory. We will show that there is an adjunction
	\[\mathclap{(-)_{cont}\colon\Gamma\CoarseHomologyTheories_{\bC} \leftrightarrows  \mathbf{Cont}\Gamma\CoarseHomologyTheories_{\bC}: inclusion\ .}\]

Let us first construct the functor $(-)_{cont}$. For simplicity of the presentation we will only describe it on objects. The honest construction of the functor is similar and just involves more complicated diagrams.

{Denote by $\Gamma\BC^{mb}$ the full subcategory of $\Gamma\BC$ spanned by the $\Gamma$-bornological coarse spaces which carry the minimal bornology.
Given a functor \[E \colon \Gamma\BC \to \bC\] with target a stable cocomplete $\infty$-category,
we define the functor $E_{cont}$ by left Kan extension:
\[\xymatrix{
 \Gamma\BC^{mb}\ar[r]^-{E}\ar[d] & \bC \\
 \Gamma\BC\ar@{..>}[ur]_-{\, E_{cont}}&}
\]
Since the image of every morphism originating in a $\Gamma$-bornological coarse space with minimal bornology is locally finite, the point-wise formula for the left Kan-extension implies that the canonical map
\begin{equation}\label{f2iuhfiu23f7832fhbfhjejfbhe}
\colim_{ F\in\cF(X) } E(F_{X}) \xrightarrow{\sim} E_{cont}(X)
\end{equation}
is an equivalence.}

\begin{lem}\label{feiuhziur34rt34rf}
If $E$ is an equivariant coarse homology theory, then $E_{cont}$ is a continuous equivariant coarse homology theory.
\end{lem}

\begin{proof}
We start with showing that $E_{cont}$ is coarsely invariant.
Let $X$ be a $\Gamma$-bornological space. 
Since the subsets
$\{0,1\}\times F$ of $\{0,1\}\times X$ for all locally finite  invariant subsets $F$  of $X$ are cofinal in all locally finite invariant subsets of $\{0,1\}_{max,max}\otimes X$ we get the second equivalence in the chain
\begin{eqnarray*}E_{cont}(\{0,1\}_{max,max}\otimes X)&\simeq &\colim_{F^{\prime}\in\cF(\{0,1\}\otimes X)  } E( F^{\prime}_{\{0,1\}_{max,max}\otimes X})\\
&\simeq& \colim_{F\in\cF(X)  }E(\{0,1\}_{max,max}\otimes F_X)\\
&\simeq&  \colim_{F\in\cF(X)  } E( F_{X})\\&\simeq& E_{cont}(X)\ .
\end{eqnarray*}
{See \cref{exjk22332} for the notation $F_X$ (the induced structures on the subset $F$).} The third equivalence {in the above chain of equivalences} follows from the coarse invariance of $E$.
 
Next we show that $E_{cont}$ satisfies excision.
If $(Z,\cY)$ is an invariant complementary pair on the $\Gamma$-bornological {coarse} space  $X$, then $(F \cap Z, F \cap \cY)$ is an invariant  complementary pair on~$F_{X}$.
Hence
\[\xymatrix{E_{cont}(Z\cap \cY)\ar[r]\ar[d]&E_{cont}(\cY)\ar[d]\\E_{cont}(Z)\ar[r]&E_{cont}(X)}\]
is the colimit of the push-out diagrams over $F$ in $\cF(X)$:
\[\xymatrix{E (F\cap Z\cap \cY)\ar[r]\ar[d]&E (F\cap \cY)\ar[d]\\E (F\cap Z)\ar[r]&E (F_{{X}})}\] (all subsets of $F$  are equipped with the bornological coarse structures induced from $F_{X}$)
and hence itself a push-out diagram.

We now show that $E_{cont}$ vanishes on flasques.
Assume that $X$ is a flasque $\Gamma$-bornological coarse space and that flasqueness is implemented by the morphism $f\colon X\to X$.
If $F$ is an invariant, locally finite subset of $X$, then $ \tilde F:=\bigcup_{n\in \nat} f^{n}(F)$ is again an invariant,  locally finite subset of $X$. Furthermore, $\tilde F_{X}$ is flasque with flasqueness implemented by the restriction $f|_{\tilde F}$. 
The inclusion $F\to \tilde F$ belongs to the structure maps for the colimit over $\cF(X)$.
 Since $E(\tilde F)\simeq 0$ this finally implies that $E_{cont}(X)\simeq 0$.

Finally, we have $u$-continuity by
\begin{eqnarray*} \colim_{U\in \cC} E_{cont}(X_{U})&\simeq&
\colim_{U\in \cC} \colim_{F\in\cF(X) } E(F_{X_{U}})\\&\simeq&
\colim_{U\in \cC} \colim_{F\in\cF(X) } E(F_{(F\times F)\cap U})\\
&\simeq&  \colim_{F\in\cF(X) }  \colim_{U\in \cC} E(F_{(F\times F)\cap U})\\
&\simeq& \colim_{F\in\cF(X) }   E(F_{X} )\\
&\simeq& E_{cont}(X)\ .
\end{eqnarray*}
This finishes the proof that $E_{cont}$ is an equivariant coarse homology theory.

We now argue that $E_{cont}$ is continuous. Let $X$ be a $\Gamma$-bornological coarse space and $\cY=(Y_{i})_{i\in I}$ be a trapping exhaustion of $X$. For every invariant, locally finite subset $F$ of~$X$ exists an index  $i$ in $I$ such that $F\subseteq Y_{i}$. Furthermore, for every $i$ in $I$ we have an inclusion $\cF(Y_{i})\subseteq \cF(X)$. This implies that
\[E_{cont}(\cY)\simeq \colim_{i\in I} \colim_{F\in \cF(Y_{i})}E(F_{Y_{i}}) \simeq \colim_{F\in \cF(X)} E(F_{X})\simeq E_{cont}(X)\ .\]
This finishes the proof of \cref{feiuhziur34rt34rf}.
\end{proof}

\begin{prop}\label{prop0200}
There exists an adjunction
\[\mathclap{(-)_{cont}\colon\Gamma\CoarseHomologyTheories_{\bC} \leftrightarrows  \mathbf{Cont}\Gamma\CoarseHomologyTheories_{\bC}: inclusion\ .}\]
\end{prop}
\begin{proof}
The collection of maps $F\to X$ for all locally finite invariant subsets $F$ of $X$ induces a transformation of functors 
  \[\eta\colon(-)_{cont}\to \id\] on $\Gamma\CoarseHomologyTheories_{\bC}$. By \cref{feoijewoif234345}, if $E$ is continuous, then the transformation $\eta_{E}\colon E_{cont}\to E$ is an equivalence. Moreover the transformations $\eta_{E_{cont}}$ and $(\eta_{E})_{cont}$ (i.e., the functor $(-)_{cont}$ applied to $\eta_{E}$) are equivalences. By \cite[Prop.~5.2.7.4]{htt} we get the desired adjunction.
\end{proof}

Let the cocomplete stable $\infty$-category $\bC$ admit all small products.  Let $E$ be a $\bC$-valued equivariant coarse homology theory.
\begin{lem}
If $E$ is strongly additive, then so is $E_{cont}$.
\end{lem}

\begin{proof}
Let $(X_{i})_{i\in I}$ be a family of $\Gamma$-bornological coarse spaces and set
$X:=\bigsqcup_{i\in I}^{\free}X_{i}$ (\cref{efwifuhwfowefewfewfwef}). For a subset $F$ of $X$ and $i$	 in $I$ we write $F_{i}:=F\cap X_{i}$.
Then $F$ is locally finite in $X$ if and only if $F_{i}$ is locally finite in $X_{i}$ for every $i$ in $I$. Furthermore, we have an isomorphism of $\Gamma$-bornological coarse spaces
$F_{X}\cong \bigsqcup_{i\in I}^{\free} F_{i,X_{i}}$.
This implies \[\colim_{F\in \cF(X)}\prod_{i\in I} E ( F_{i,X_{i}})\simeq\prod_{i\in I}\colim_{F_{i}\in \cF(X_{i})} E(F_{i,X_{i}})\ .\]
We must show that the right vertical map in the diagram \[\xymatrix{
\colim_{F\in\cF(X)}E(F_X)\ar[r]^-\simeq\ar[d]^\simeq&E_{cont}(X)\ar[dd]\\
\colim_{F\in \cF(X)}\prod_{i\in I} E ( F_{i,X_{i}} )\ar[d]^\simeq&\\
\prod_{i\in I}\colim_{F_i\in \cF(X_i)}E(F_{i,X_{i}})\ar[r]^-\simeq&\prod_{i\in I}E_{cont}(X_i)
}\]   is an equivalence. Indeed, the horizontal maps are equivalences by the definition of $E_{cont}$ and the upper vertical map is an equivalence since $E$ is strongly additive.
\end{proof}

\section{Change of groups}\label{roigurgoregregreg}
In this section we describe various change of groups constructions. They induce  adjunctions between
the corresponding categories of  equivariant motivic  coarse spectra which are very similar to the base change functors in motivic homotopy theory.

Assume that we are given an equivariant coarse homology theory defined for every group (as it is the case for all our examples). The compatibility of the equivariant coarse homology theory with the change of groups functors is expressed by natural transformations. These transformations are additional data and will be discussed for every example of equivariant coarse homology theory separately.

What we describe here is the beginning of a story which should finally capture all group-change functors in a sort of spectral Mackey functor  {formalism (see, e.g., Barwick \cite{Barwick:2014aa} and Barwick--Glasman--Shah \cite{Barwick:2015aa}).}
The obvious task here is to capture the relations between these functors  on the motivic level (like iterated restrictions or inductions and the Mackey relation, \cref{oieioefuew9843534535345}) together with all their higher coherences in a proper way. 

All change of groups transformations are  associated to a 
homomorphism of groups
\[ \iota\colon H\to \Gamma\ .\]

\subsection{Restriction}\label{rekijgerogoijo34544566} \label{roijgeogergrgrege}
 Every $\Gamma$-bornological coarse space  gives rise to an $H$-bornological coarse space, where the action of $H$ is induced from the action of $\Gamma$ via $\iota$.    In this way we get a restriction functor
\[\Res^{\Gamma}_{H}\colon\Gamma\BC\to H\BC\ .\] If the homomorphism $\iota$ is not clear from the context, then we 
  add it to the  notation and  write
$\Res^{\Gamma}_{H}(\iota)$.

The functor
$\Res^{\Gamma}_{H} $  induces a pull-back functor $\Res^{\Gamma}_{H,pre}$ for presheaves, which preserves  {all} limits and colimits. Since $\PSh( \Gamma\BC  )$ is a presentable $\infty$-category, by   Lurie \cite[Cor.\ 5.5.2.9]{htt} the functor $\Res^{\Gamma}_{H,pre}$ is the right-adjoint of  an adjunction 
\[\Res^{\Gamma,pre}_{H }\colon\PSh( \Gamma\BC  )\leftrightarrows \PSh(H\BC ):\Res^{\Gamma}_{H,pre}\ .\]
The functor $\Res^{\Gamma}_{H}$ sends equivariant complementary pairs on a $\Gamma$-bornological coarse space $X $ to 
equivariant complementary pairs on the $H$-bornological coarse space $\Res^{\Gamma}_{H}(X) $.
Consequently, the restriction functor $\Res^{\Gamma}_{H,pre}$ preserves sheaves. 
In the following, we decorate the Yoneda functors by the relevant group.
Using $\Res^{\Gamma,pre}_{H}\circ \yo_{\Gamma}\simeq \yo_{H}\circ \Res^{\Gamma}_{H}$, we see that  $\Res^{\Gamma,pre}_{H}$ sends the generators of the localization listed in \cref{iojoijfoiwejfowejffewf98u98u98uu94234234234} for the group $\Gamma$ to corresponding generators for $H$. We conclude  that
$\Res^{\Gamma}_{H,pre}$ preserves coarsely invariant sheaves, sheaves which vanish on flasque spaces,  and $u$-continuous sheaves. 
Hence we get an adjunction
\[\Res^{\Gamma,*}_{H } \colon \Gamma\Spc\cX\leftrightarrows H\Spc\cX:\Res^{\Gamma}_{H,*} \ .\]
Here 
 $
\Res^{\Gamma}_{H,*}$ is given by the restriction of $ \Res^{\Gamma}_{H,pre} $ to coarse motivic spaces, and its left adjoint satisfies
\[\Res^{\Gamma,*}_{H }\simeq \cL_{H}\circ \Res^{\Gamma,pre}_{H }\ ,\] where $\cL_{H} $ is  the localization as in \eqref{erglkhoi3hto34t34t43t34t3t} (we have added a subscript $H$ in order to indicate the relevant group).

Hence by {passing} to the stabilizations we get an adjunction
\[\Res^{\Gamma,Mot}_{H} \colon\Gamma\Sp\cX \leftrightarrows H\Sp\cX :\Res^{\Gamma}_{H,Mot} \ ,\]
where
$\Res^{\Gamma}_{H,Mot}$ is defined by  the  extension of the functor $\Res^{\Gamma}_{H,pre }$  to stable objects. The obvious equivalence 
\[\Res^{\Gamma,pre}_{H}\circ \yo_{\Gamma}\simeq \yo_{H}\circ \Res^{\Gamma}_{H}\] implies
 the equivalence
\[\Res^{\Gamma,Mot}_{H}\circ \Yo_{\Gamma}^{s} \simeq \Yo_{H}^{s}\circ  \Res^{\Gamma}_{H}\ .\]
where we have decorated the Yoneda functors by the relevant group.

\subsection{Completion}
We consider a $\Gamma$-bornological coarse space $X$. Let $\cC$ and $\cB$ denote the coarse structure and the bornology of $X$.
  We define a new compatible bornology $\cB_{H}$ on $X$ generated by the {$\iota(H)$}-completions {$\iota(H)B$}
  of the bounded subsets $B$ of $X$. We observe that $\cC$ and $\cB_{H}$ are compatible and that $\cB_{H}$ is $N_\Gamma(\iota(H))$-invariant (as a subset of $\cP(X)$), {where $N_{\Gamma}( \iota(H))$ denotes the normalizer of  the subgroup  $ \iota(H) $  in $\Gamma$.}
  
\begin{ddd} The \emph{$H$-completion} of $X$ is the $N_\Gamma(\iota(H))$-bornological coarse space defined by $B_{H}(X):=(X,\cC,\cB_{H})$.
\end{ddd}

In this way we define a functor
\[B_{H}\colon\Gamma\BC\to N_\Gamma(\iota(H))\BC\ .\]
The pull-back along $B_{H}$ induces an adjunction
\[B_{H}^{pre}\colon \PSh(\Gamma\BC) \leftrightarrows \PSh(N_\Gamma(\iota(H))\BC) \colon B_{H,pre}\ .\]
It is easy to see that
$B_{H,pre}$ preserves sheaves, $u$-continuity and coarse invariance, since its left-adjoint adjoint $B_{H}^{pre}$ preserves the corresponding generating morphisms.  See \cref{roijgeogergrgrege} for a similar argument.

If $X$  is a  flasque $\Gamma$-bornological coarse space with flasqueness implemented by $f \colon X\to X$, then $B_{H}(X)$ is flasque
with flasqueness implemented by the same map. Here   it is important to define flasqueness with Condition \ref{hckjhckhwhuiechiu}  in \cref{hckjhckhwhuiechiu1} and not the weaker one discussed in 
\cref{hckjhckhwhuiechiu2}. 
Consequently, $B_{H,pre}$ preserves presheaves which vanish on flasques.

Similarly as in the case of the restriction, 
we get an adjunction

\[B_{H}^{Mot}\colon\Gamma\Sp\cX\leftrightarrows N_\Gamma(\iota(H))\Sp\cX: B_{H,Mot}\]
and an equivalence
\[B_{H}^{Mot}\circ \Yo_{\Gamma}^{s}\simeq \Yo^{s}_{N_{\Gamma}(\iota(H))}\circ B_{H}\ .\]
 
The identity of the underlying set induces a morphism $b_{H}\colon B_{H}(X)\to \Res_{N_\Gamma(\iota(H))}^\Gamma (X)$.
 The transformation $b_{H}$ induces 
 a natural transformation of functors
\[b_{H}\colon B_{H}\to \Res_{N_\Gamma(\iota(H))}^\Gamma \colon\Gamma\BC\to N_\Gamma(\iota(H))\BC\ .\]

By functoriality, we get a transformation
\[ b_{H}\colon E\circ B_{H}\to E \circ \Res_{N_\Gamma(\iota(H))}^\Gamma \] 
for any equivariant coarse homology theory $E$.

 \begin{rem}
  If $H$ is a finite group, then $B_{H}\cong \Res_{N_\Gamma(\iota(H))}^\Gamma$. 
 \end{rem}

\subsection{Quotients}  
We consider a $\Gamma$-bornological coarse space $X$.
We  form the quotient  set $H\backslash X$ which carries an action of the group \[W_{\Gamma}(H):=N_{\Gamma}( \iota(H))/ \iota(H)\ ,\]
where $N_{\Gamma}( \iota(H))$ denotes the normalizer of  the subgroup  $ \iota(H) $  in $\Gamma$. 

Let $\pi\colon X\to H\backslash X$ denote the projection.
 We  equip $H\backslash X$ with the maximal bornology such that the projection $\pi\colon B_{H}(X)\to H\backslash X$ is proper.  
Furthermore, we  equip $H\backslash X$ with the minimal coarse structure such that    $\pi\colon X\to H\backslash X$ is controlled.
In this way we define a functor 
\[Q_{H}\colon\Gamma\BC\to W_{\Gamma}(H)\BC\ .\]

The projection maps define a natural transformation $\pi_H \colon B_H \to \Res_{N_\Gamma(\iota(H))}^{W_\Gamma(H)}(Q_H)$ of $N_\Gamma(\iota(H))\BC$-valued functors.

  For $X$ in $\Gamma\BC$
one can interpret $\Res^{W_{\Gamma}(H)}_{N_{\Gamma}(\iota(H))}Q_{H}(X)$ as a coequalizer.
For a  bornological coarse space $Z$ let $Z_{max-\cB}$ denote the  bornological coarse space obtained by replacing its bornology by the maximal bornology. The two maps $H\times X\to X$  in the lemma below are given by the projection $(h,x)\mapsto x$ and the action $(h,x)\mapsto h x$. 
The diagram and the colimit are considered in $N_{\Gamma}(\iota(H))\BC$, where the action of $\sigma$ in $N_{\Gamma}(\iota(H))$  on $H \times X$ is given by $\sigma(h,x):=(\sigma  h\sigma^{-1},\sigma x)$.

\begin{lem}\label{rgioergerggegergreg}
We have an isomorphism
\[\Res^{W_{\Gamma}(H)}_{N_{\Gamma}(\iota(H))}Q_{H}(X)\cong \colim \big( (H_{min,max}\otimes X)_{max-\cB}\rightrightarrows B_{H}(X)
\big)\ .\]
 \end{lem}
\begin{proof}
One checks that the two morphisms in the coequalizer diagram are controlled. Since we replaced the bornology on the domain by the maximal one they are obviously proper.
The coequalizer diagram is colim-admissible, {see \cref{ex:coeq}}. One checks that the description of the structures on $Q_{H}(X)$ given above coincides with the explicit description of the structures of the colimit given in the proof of \cref{efiu9ewofwfwewfwef}.
\end{proof}

 The functor 
$Q_{H} $
induces a pull-back in presheaves
$ Q_{H,pre}$. 
This functor preserves  {all} limits and colimits. Again by {Lurie} \cite[Cor.\ 5.5.2.9]{htt} 
it fits into an adjunction 
\[Q^{pre}_{H} \colon\PSh(\Gamma \BC  )\leftrightarrows \PSh( W_{\Gamma}(H)\BC ): Q_{H,pre} \ .\]

Note that the functor $Q_{H}$ induces a bijection between equivariant complementary pairs on a $\Gamma$-bornological coarse space  $X\ $ and on the $W_{\Gamma}(H) $-bornological coarse space 
$Q_{H}(X) $. Consequently, $Q_{H,pre} $ preserves sheaves.

It is furthermore clear that $Q_{H,pre} $ preserves coarsely invariant sheaves and $u$-continuous sheaves.   
 If the $\Gamma$-bornological coarse space $X $ is flasque with flasqueness implemented by $f\colon X\to X$, then the induced map $\bar f\colon Q_{H}(X)\to Q_{H}(X)$ implements flasqueness of $Q_{H}(X)$. It follows that
$Q_{H,pre}$ preserves coarse motivic spaces. 

Hence $Q_{H,pre} $ restricts to  equivariant coarse motivic spaces. Similar as before we get an adjunction
\[Q_{H}^{Mot}\colon \Gamma\Sp\cX
\leftrightarrows  W_{\Gamma}(H)\Sp\cX: Q_{H,Mot}\ .\]
We have the relation
\[\Yo_{W_{\Gamma}(H)}^{s}\circ Q_{H}\simeq Q_{H}^{Mot}\circ \Yo^{s}_{\Gamma}\ .\]

\subsection{Products}

We consider two groups $\Gamma$ and $\Gamma^{\prime}$. For a $\Gamma$-bornological coarse space $X $ and a $\Gamma^{\prime}$-bornological coarse space  $X^{\prime}  $ we can form the product
$X\otimes X^{\prime}$ which is a $ (\Gamma\times \Gamma^{\prime})$ bornological coarse space.
We fix the $\Gamma$-bornological coarse space $X$ and consider the functor
\[P_{X}:=X\otimes -\colon\Gamma^{\prime}\BC\to (\Gamma\times \Gamma^{\prime})\BC\ .\]
As in the preceding cases one can check that the restriction \[P_{X,pre}\colon\PSh ((\Gamma\times \Gamma^{\prime})\BC)\to \PSh(\Gamma^{\prime}\BC)\] along $P_{X}$ preserves coarse motivic spaces and
induces
an adjunction 
\[P_{X}^{Mot}\colon\Gamma^{\prime}\Sp\cX\leftrightarrows(\Gamma\times \Gamma^{\prime})\Sp\cX: P_{X,Mot}\]
such that \[P_{X}^{Mot} \circ \Yo^{s}_{\Gamma^{\prime}} \simeq \Yo^{s}_{(\Gamma\times \Gamma^{\prime})}\circ P_{X}\ .\]

\subsection{Induction}\label{wefiwuofwefwegwgewrgwe}

We now come back to our original situation and consider the homomorphism $\iota\colon H\to \Gamma$ of groups.
We define a $(\Gamma\times H)$-bornological coarse space $\hat \Gamma$ as follows: \begin{enumerate}
\item The underlying bornological coarse space of $\hat \Gamma$ is $\Gamma_{min,min}$.
\item The group $(\Gamma\times H)$ acts on the set $\hat \Gamma$ by \[(\gamma,h)\gamma^{\prime}:=\gamma\gamma^{\prime}\iota(h)^{-1}\ .\]
\end{enumerate}
 We define the functor
\[\hat P_{\Gamma}:=\Res^{\Gamma\times H\times H}_{\Gamma\times H}\circ P_{\hat \Gamma} \colon H\BC\to (\Gamma\times H)\BC\ ,\]
where the restriction is along the homomorphism \[\id_{\Gamma}\times \diag_{H}\colon\Gamma\times H\to \Gamma\times H\times H\ .\]

This functor extends to motives 
\[\hat P_{\Gamma}^{Mot}:=\Res^{\Gamma\times H\times H,Mot}_{\Gamma\times H}\circ P^{Mot}_{\hat \Gamma}\ .\]
By construction we have an adjunction
\[\hat P_{\Gamma}^{Mot}\colon H\Sp\cX\leftrightarrows (\Gamma\times H)\Sp\cX:\hat P_{\Gamma,Mot}\] and the relation
\[\Yo^{s}_{\Gamma\times H}\circ \hat P_{\Gamma}\simeq \hat P_{\Gamma}^{Mot}\circ \Yo^{s}_{H}\ .\]

We  consider the canonical embedding $\kappa\colon H\to \Gamma\times H$ into the second factor. Note that then we have $\Gamma\times H= N_{\Gamma\times H}(H)$ and hence
$W_{\Gamma\times H}(H)\cong \Gamma$. 
We thus have the  quotient functor  \[Q_{H}(\kappa)\colon(\Gamma\times H)\BC\to \Gamma\BC\]
(it is useful to add the embedding $\kappa$ as an argument since there is also the other obvious homomorphism $(\iota,\id)\colon H\to \Gamma\times H$). We define the induction functor as the composition
\[\Ind_{H}^{\Gamma}\colon H \BC\to  \Gamma\BC\ , \quad  \Ind_{H}^{\Gamma}:= Q_{H}(\kappa)\circ \hat P_{\Gamma}\ .\]
The induction functor extends to motives
\[\Ind_{H}^{\Gamma,Mot}:=Q_{H}^{Mot}(\kappa)\circ \hat P_{\Gamma}^{Mot}:H\Sp\cX\to \Gamma \Sp\cX\]
such that
\[\Ind_{H}^{\Gamma,Mot}\circ \Yo^{s}_{H}\simeq  \Yo_{\Gamma}^{s}\circ \Ind_{H}^{\Gamma}\ .\]
\begin{rem}
	\label{rem:jtjbzuzgkvrgkieilfu}
The underlying $\Gamma$-set of $\Ind_{H}^{\Gamma}(X)$ is 
 $\Gamma\times_{H} X$ with the left-action of $\Gamma$ on the left factor. Here $\Gamma\times_{H} X$ stands for the quotient set $H \backslash (\Gamma \times X)$ with respect to the action given by
\[h(\gamma,x)=(\gamma \iota(h^{-1}),hx)\ .\] The bornology on $\Gamma\times_{H} X$ is generated by the images of the subsets
$\{\gamma\}\times B$ for all bounded subsets of $X$, and the coarse structure is 
generated by the images of 
$\diag_{\Gamma}\times U$ for all entourages $U$ of $X$.
\end{rem}

Let  $K:=\ker(\iota \colon H\to \Gamma)$. 

\begin{lem}\label{jkhsdf8923}
We have $\Ind_{H}^{\Gamma}\cong \Ind_{H}^{\Gamma}\circ B_{K}$.
\end{lem}

\begin{proof}
Straightforward.
\end{proof}

\begin{rem} Using  Lemma \ref{rgioergerggegergreg} one can easily check that  we have an isomorphism \[\Ind_{H}^{\Gamma}(X)\cong \colim \left((H_{min,max}\otimes \Gamma_{min,min}\otimes X)_{max-\cB}\rightrightarrows B_{H}(\Gamma_{min,min}\otimes X)\right)\]
in $\Gamma \BC$.
The two arrows are given by $(h,\gamma,x)\mapsto (\gamma,x)$ and $(h,\gamma,x)\mapsto (\gamma h^{-1},hx)$. The group $\Gamma$ acts on the factor $\Gamma_{min,min}$ by left multiplication.  
We used the more complicated description of the induction as a composition of various previously defined functors in order to deduce that induction descends to motives.
  \end{rem}

\begin{prop}
If $K$ is finite and the image of $\iota$ has finite index in $\Gamma$, then we have an adjunction
\[\Ind_{H}^{\Gamma}\colon H\BC\leftrightarrows \Gamma\BC: \Res^{\Gamma}_{H}\ .\]
\end{prop}
\begin{proof}
 {As we observed earlier, the underlying set of $\Ind_H^\Gamma(X)$ is given by $\Gamma \times_H X$.
 Recall that the functions \footnote{We use the word ``function'' in order to denote maps between underlying sets}
 \[ X \to \Gamma \times_H X, \quad x \mapsto [e,x] \]
 and
 \[ \Gamma \times_H \Res_H^\Gamma(Y) \to Y, \quad [\gamma,y] \mapsto \gamma y \]
 define the unit and counit, respectively, of an adjunction $H\Set \rightleftarrows \Gamma\Set$.
 We must check that these functions define morphisms of equivariant bornological coarse spaces.
 \\
 Since $K$ is finite and normal in $H$, we have $B_KX \cong X$ for every $H$-bornological coarse space $X$.
 We claim that the function \[X \to \Gamma \times X\ , \quad x \mapsto (e,x)\] defines a natural morphism $B_KX \to B_H(\Gamma_{min,min} \otimes X)$.
 This function is obviously controlled. To see that it is proper, we note that the generating bounded subsets of $B_H(\Gamma_{min,min} \otimes X)$ are of the form $\iota(H)(\{\gamma\} \times B)$.
 The intersection of such a subset with $\{e\} \times X$ is equal to $\{e\} \times KB$, hence bounded because $K$ is finite.
 Then the unit is given by the composition
 \[ X\cong B_K(X) \to \Res^{\Gamma}_{H}(B_H(\Gamma_{min,min} \otimes X)) \xrightarrow{\pi} \Res^{\Gamma}_{H}(Q_H(\Gamma_{min,min} \otimes X)) = \Res^{\Gamma}_{H}(\Ind_H^\Gamma(X)). \]
 Consider now the composition \[F \colon \Gamma \times \Res_H^\Gamma Y \to \Gamma \times_H \Res_H^\Gamma Y \to Y\] of the projection map with the function which will provide the desired counit.
 It suffices to show that this composition is a morphism of bornological coarse spaces $B_H(\Gamma_{min,min} \otimes \Res_H^\Gamma(Y)) \to Y$.
 If $U$ is an entourage of $Y$, then $F(\diag_\Gamma \times U) = \bigcup_{\gamma \in \Gamma} (\gamma \times \gamma)(f \times f)(U)$ is an entourage of $Y$, see \cref{ekfhwekjfkjewfewfwfwfewf}.
 If $B$ is a bounded subset of $Y$, then $F^{-1}(B) = \bigcup_{\gamma \in \Gamma} \{\gamma^{-1}\} \times \gamma B$ is bounded since $\bigcup_{h \in H} \{\gamma \iota(h)^{-1}\} \times \iota(h)B$ is bounded for every $\gamma$ in $\Gamma$ and the image of $\iota$ has finite index in $\Gamma$.
 This shows that $F$, and hence the counit, is a morphism.}
\end{proof}

We now consider homomorphisms  $H \to \Gamma$ and $H' \to \Gamma$, and set $K:=\ker(H\to \Gamma)$ and $K':=\ker(H' \to \Gamma)$. Note that $W_{H}(K)\cong H/K=:\bar H$ and $W_{H'}(K')\cong H'/K'=:\bar H'$.

\begin{lem}\label{oieioefuew9843534535345}
Assume that  $H^{\prime}\backslash \Gamma /H$ is finite.

For an $H$-bornological coarse space  $X$ we have the relation
\begin{eqnarray*}
\Res^{\Gamma}_{H^{\prime}}\circ \Ind^{\Gamma}_{H}(X) & \cong &
\coprod_{[\gamma]\in \bar H^{\prime}\backslash \Gamma/\bar H} \Res^{\bar H^{\prime}}_{H^{\prime}} \circ \Ind^{\bar H^{\prime}}_{\bar H\cap \gamma^{-1} \bar H^{\prime}\gamma} \circ \Res^{\bar H}_{\bar H\cap \gamma^{-1} \bar H^{\prime}\gamma}(c_{\gamma})\circ \Ind^{\bar H}_{H} (X)\ .
\end{eqnarray*}
where $c_{\gamma}\colon \bar H\cap \gamma^{-1} \bar H^{\prime}\gamma\to \bar H$ is given by $\bar h \mapsto \gamma^{-1}\bar h\gamma$. 
\end{lem}

\begin{proof}
One just makes all definitions explicit.
\end{proof}

The induction on the motivic level   is given by
 \[\Ind^{\Gamma, Mot}_{H}:=  Q_{H}^{Mot}\circ \hat P_{\Gamma}^{Mot}\ .\] This functor 
 fits into the adjunction
\[ \Ind^{\Gamma,Mot}_{H}\colon  H\Sp\cX\leftrightarrows  \Gamma\Sp\cX:  \Ind^{\Gamma}_{H,Mot}\]
and is compatible with the Yoneda functor:
\[\Ind^{\Gamma, Mot}_{H}\circ \Yo^{s}_{H}\simeq \Yo^{s}_{\Gamma}\circ \Ind^{\Gamma}_{H}\ .\]
Since the Yoneda functor preserves finite coproducts (since it is excisive), the \cref{oieioefuew9843534535345} implies:

\begin{kor}
If $H^{\prime}\backslash \Gamma /H$ is finite, then
\[ \Res^{\Gamma,Mot}_{H^{\prime}}\circ \Ind^{\Gamma,Mot}_{H}\simeq \bigoplus_{[\gamma]\in \bar H^{\prime}\backslash G/\bar H} 
 \Res^{\bar H^{\prime},Mot}_{H^{\prime}} \circ \Ind^{\bar H^{\prime}, Mot}_{\bar H\cap \gamma^{-1}\bar H^{\prime}\gamma}  \circ \Res^{\bar H,Mot}_{\bar H\cap \gamma^{-1}\bar H^{\prime}\gamma} (c_{\gamma})\circ Q^{Mot}_{K}\ .\]
\end{kor}
 
\part{Examples}\label{weoihjowefewfewfewww}

\section{Equivariant coarse ordinary homology}

\subsection{Construction}

In this section we introduce equivariant coarse ordinary homology theory.
Its construction is completely analogous to the non-equivariant case \cite[Sec.~6.3]{buen}.

Let $X$ be a $\Gamma$-bornological coarse space and let $n$ be a natural number.
An \emph{$n$-chain on $X$} is a function $c \colon X^{n+1} \to \IZ$.
Its \emph{support} is defined by $\supp(c) := \{x \in X^{n+1} \, | \, c(x) \not= 0\}$.
{We typically think of $n$-chains as infinite linear combinations of points in $X^{n+1}$.}

Let $U$ be a coarse entourage of $X$ and let $B$ be a bounded subset.
A point $(x_0, \ldots, x_n)$ in $X^{n+1}$ is \emph{$U$-controlled} if $(x_{i},x_{j}) \in U$ for every $0 \leq i,j \leq n$.
The point $(x_0, \ldots, x_n)$ in $X^{n+1}$ \emph{meets $B$} if there exists $0 \leq i \leq n$ such that $x_{i}$ lies in $B$.

An $n$-chain $c$ is \emph{$U$-controlled} if every point of its support is $U$-controlled.
The $n$-chain $c$ is \emph{controlled} if it is $U$-controlled for some coarse entourage $U$ of $X$.
{Furthermore, $c$ is \emph{locally finite} if for every bounded subset $B$ of $X$ the set of points in $\supp(c)$ which meet $B$ is finite.}
We let $C\cX_{n}(X)$ denote the  abelian group of controlled locally finite $n$-chains.
   
The boundary operator $\partial \colon C\cX_{n}(X)\to C\cX_{n-1}(X)$ (for all $n\ge 1$) is defined to be $\partial := \sum_{i=0}^{n}(-1)^{i}\partial_{i}$,
where $\partial_{i}$ is the linear extension of the operator $X^{n+1} \to X^{n}$ omitting the $i$'th entry.  
One checks that $\partial$ is well-defined  and a differential of a chain complex.
By $C\cX(X)$ we denote the chain complex of locally finite and controlled chains on $X$.

\begin{ddd}
For every natural number $n$ we let  $C\cX^{\Gamma}_{n}(X)$ denote the subgroup of $\ C\cX_{n}(X) $ of locally finite and controlled $n$-chains which are in addition $\Gamma$-invariant.
\end{ddd}

For every natural number $n$ the boundary operator $\partial \colon C\cX_{n+1}(X)\to C\cX_{n}(X)$
restricts to a boundary operator $\partial\colon C\cX^{\Gamma}_{n+1}(X)\to C\cX^{\Gamma}_{n}(X)$ between the subgroups of $\Gamma$-invariants.
Hence we have defined a subcomplex $C\cX^{\Gamma}(X)$ of $ C\cX(X)$.
If $f\colon X\to X^{\prime}$ is a morphism between $\Gamma$-bornological coarse spaces,
then the induced map $C\cX(f) \colon C\cX(X)\to C\cX(X^{\prime})$  of chain complexes preserves the subcomplexes of $\Gamma$-invariants.
Therefore, we obtain a functor
\[ C\cX^{\Gamma}\colon\Gamma\BC\to \Ch\ , \]
where $\Ch$ denotes the category of chain complexes.

In order to go from chain complexes to spectra we use the Eilenberg--MacLane correspondence  \begin{equation}\label{fvihgoirvebvevvreveve}
\cE\cM \colon \Ch \to \Sp\ .
\end{equation}
One way to define this functor is as the composition
\[\cE\cM\colon \Ch\to \Ch[W^{-1}]\to \Sp\ ,\]
where the first functor is the localization of the category  of chain complexes at the quasi-isomorphisms, and the second functor
is the mapping spectrum functor $\map(\Z[0],\dots)$ of the stable $\infty$-category $\Ch[W^{-1}]$,
where $\Z[0]$ is the chain complex with $\Z$ placed in degree zero.

\begin{ddd}
We define the functor {$H\cX^{\Gamma}\colon\BC\to \Sp$}
by
\[H\cX^{\Gamma}:=\cE\cM\circ C\cX^{\Gamma}\ .\qedhere\]
\end{ddd}

\begin{theorem}\label{oifjweoifewfefwefwefewf}
$H\cX^{\Gamma}$ is an equivariant coarse homology theory.
\end{theorem}

\begin{proof}
We observe that the arguments given in the proof of the \cite[Thm.~6.15]{buen} extend word-by-word to the equivariant case.
\end{proof}
 
\subsection{Calculations for spaces of the form \texorpdfstring{$\Gamma_{can,min} \otimes S_{min,max}$}{Gamma-can-min otimes S-min-max}}\label{wifhuifhiwefhieuwhufewfewfwf}

In this section we will do some computations of equivariant coarse ordinary homology groups.
In particular, we will relate it to ordinary group homology.

\begin{ex}
If the $\Gamma$-bornological coarse space $X$ has the trivial $\Gamma$-action, then we have an isomorphism $H\cX^{\Gamma}(X)\cong H\cX(X)$.
\end{ex}

In order to provide more examples we consider the group homology functor
\[ H(\Gamma,-)\colon\Mod(\Z[\Gamma])\to \Sp \]
which can be defined as the composition
\[ \Mod(\Z[\Gamma])\to  \Ch_{\Z[\Gamma]}\to  \Ch_{\Z[\Gamma]}[W^{-1}]\xleftarrow{\simeq}
   \Ch^{\mathit{free}}_{\Z[\Gamma]}[W^{-1}]\xrightarrow{\Z\otimes_{\Z[\Gamma]}-} \Ch_{\Z}[W^{-1}]\xrightarrow{\cE\cM} \Sp\ .\]
The first functor sends a $\Z[\Gamma]$-module to a chain complex of $\Z[\Gamma]$-modules concentrated in degree $0$.
The second functor is the localization at quasi-isomorphisms.
The equivalence in the third step is induced by the inclusion of the full subcategory of chain complexes of free $\Z[\Gamma]$-modules.
It is essentially surjective by the existence of free resolutions.
Finally, the functor $\Z\otimes_{\Z[\Gamma]}-$ is well-defined since it preserves quasi-isomorphisms between chain complexes of free $\Z[\Gamma]$-modules.
   
Since the above definition involves the inverse of an equivalence it does not directly provide an explicit formula.
But for calculations it is useful to choose an explicit model for $H(\Gamma,-)$. The standard choice is as follows.
We consider the chain complex of $\Z[\Gamma]$-modules $C(\Gamma)$ given by 
\[ \dots\to \Z[\Gamma^{n+1}]\to \Z[\Gamma^{n}]\to \dots\to \Z[\Gamma]\ .\]
The differential $C(\Gamma)_{n+1}\to C(\Gamma)_{n}$ is defined as the linear extension of the map
\[ (\gamma_{0},\dots,\gamma_{n+1})\to \sum_{i=0}^{n+1}(-1)^{i} (\gamma_{0},\dots,\widehat{\gamma}_{i},\dots,\gamma_{n+1})\ ,\]
where $\widehat{\gamma}_i$ indicates that this component gets omitted.
The group $\Gamma$ acts diagonally on the products $\Gamma^{n}$ and this induces the $\Z[\Gamma]$-module structure on $C(\Gamma)$.
We now consider the functor
\[ C(\Gamma,-)\colon\Mod(\Z[\Gamma])\to \Ch_{\Z}\ ,\quad  V\mapsto \Z\otimes_{\Z[\Gamma]} (C(\Gamma) \otimes_{\Z} V)\ , \]
which sends a $\Z[\Gamma]$-module  $V$ to its standard complex $C(\Gamma,V)$.
Here $\Gamma$ acts diagonally on $C(\Gamma)\otimes_\Z V$.
Then we have an equivalence of functors
\[H(\Gamma,-)\simeq \cE\cM\circ C(\Gamma,-)\ .\]

If $S$ is a $\Gamma$-set, then we form the $\Gamma$-bornological coarse space $S_{min,max}$
given by the  $\Gamma$-set $S$ with the maximal bornological and the minimal coarse structure.
In this way we get a functor
\[ (-)_{min,max}\colon\Gamma\Set\to \Gamma\BC\ .\]
We have furthermore a functor
\[ \Gamma\Set\to \Mod(\Z[\Gamma])\ , \quad S\mapsto \Z[S]\ .\]
We now have two functors $\Gamma\Set\to \Sp$ given by
\[ S\mapsto H\cX^{\Gamma}(\Gamma_{can,min} \otimes S_{min,max})\quad \text{and} \quad S\mapsto H(\Gamma,\Z[S])\ .\]

\begin{prop}[{cf.~{\cite[Prop.~3.8]{engel_loc}}}]\label{flkewfjwefjeoiwef234}
 There is a natural equivalence
 \[ H\cX^{\Gamma}(\Gamma_{can,min} \otimes S_{min,max})\simeq H(\Gamma, \Z[S])\ .\]
\end{prop}

\begin{proof}
 We claim that there is a natural isomorphism between $C\cX^{\Gamma}(\Gamma_{can,min} \otimes S_{min,max})$ and the standard complex $C(\Gamma, \Z[S])$.
To do so, we identify $\Z[\Gamma^{n+1}] \otimes_{\Z} \Z[S] \cong \Z[\Gamma^{n+1} \times S]$, where $\Gamma^{n+1} \times S$ carries the diagonal $\Gamma$-action.
Then we define the homomorphism
\begin{equation}\label{foijefiojewoiwef2342342424}
 \phi_n \colon C_n(\Gamma, \Z[S]) \cong \Z \otimes_{\Z[\Gamma]} \Z[\Gamma^{n+1} \times S] \to C\cX^{\Gamma}(\Gamma_{can,min} \otimes S_{min,max})
\end{equation}
as the linear extension of
\begin{equation}\label{ewfuwehfikewfewfw5435}
1 \otimes (\gamma_{0},\gamma_{1},\dots,\gamma_{n},s) \mapsto \sum_{\gamma \in \Gamma} ((\gamma\gamma_0,\gamma s), \ldots, (\gamma\gamma_n,\gamma s))\ .
\end{equation}
Note that all summands are different points on $(\Gamma \times S)^{n+1}$ so that the infinite sum makes sense, and it is $\Gamma$-invariant by construction.
Every point $((\gamma\gamma_0,\gamma s), \ldots, (\gamma\gamma_n,\gamma s))$ is controlled by the entourage $\Gamma\{(\gamma_i,\gamma_j) \mid 0 \leq i,j \leq n\} \times \diag_{S}$
of the $\Gamma$-bornological coarse space $\Gamma_{can,min} \otimes S_{min,max}$.
To show that this chain is also locally finite,
it suffices to check that there are only finitely many points in the support of the chain \eqref{ewfuwehfikewfewfw5435} which meet bounded sets of the form $B \times S$,
where $B$ is some finite subset of $\Gamma$.
This is clear since $\Gamma$ acts freely on $\Gamma^{n+1}$.
This finishes the argument for the assertion that \eqref{foijefiojewoiwef2342342424} is well-defined.

It is straightforward to check that the collection $\{ \phi_n \}_n$ is a chain map.
 
We now argue that  the map \eqref{foijefiojewoiwef2342342424} is an isomorphism. To this end we define an inverse
\[ \psi\colon C\cX_{n}(\Gamma_{can,min} \otimes S_{min,max}) \to \Z \otimes_{\Z} \Z[\Gamma^{n+1} \times S] \cong C_n(\Gamma, \Z[S])\ .\]
Let
\[c =\sum_{ x\in (\Gamma \times S)^{n+1}} n_{x}x \]
be an invariant, controlled and locally finite $n$-chain on $\Gamma_{can,min} \otimes S_{min,max}$.
We now define
\[ \psi(c) := \sum_{(\gamma_{1},\dots,\gamma_{n},s)\in  \Gamma^{n} \times S} n_{((1,s),(\gamma_{1},s)\dots,(\gamma_{n},s))} \otimes (1,\gamma_{1},\dots,\gamma_{n},s)\ .\]
Assume that $c$ is $U$-controlled. Then only summands with
$\{\gamma_{1},\dots,\gamma_{n}\}\subseteq U[\{1\}]$ contribute to the sum.
Since $U[\{1\}]$ is bounded and $c$ is locally finite we see that the number of non-trivial summands is finite.
This implies that $\psi(c)$ is well-defined.

It is straightforward to check that $\phi$ and $\psi$ are inverse to each other:
To see that $\psi \circ \phi = \id$, use that
\[ 1 \otimes (\gamma_0,\gamma_1,\ldots,\gamma_n,s) = 1 \otimes (1,\gamma_0^{-1}\gamma_1,\ldots,\gamma_0^{-1}\gamma_n,\gamma_0^{-1}s)\ .\]
The equality $\phi \circ \psi = \id$ follows from the $\Gamma$-invariance of an $n$-chain $c = \sum_{ x\in (\Gamma \times S)^{n+1}} n_{x}x$ together with the observation that
$n_{((\gamma_0,s_0),\ldots,(\gamma_n,s_n))} = 0$ unless $s_0 = \dots = s_n$.
The latter fact is due to $S$ carrying the minimal coarse structure.

One easily checks that $\phi$ is natural for maps between $\Gamma$-sets.
\end{proof}

\begin{rem}
 Assume that $S$ is a transitive $\Gamma$-set. If we fix a point $s$ in $S$ and let $\Gamma_{s}$ denote the stabilizer subgroup  of $s$, then we have an isomorphism of $\Z[\Gamma]$-modules
\[\Z[S]\cong \Ind_{\Gamma_{s}}^{\Gamma}\Z\ .\]
The induction isomorphism in group homology now gives the chain of equivalences
\[H(\Gamma,\Z[S])\simeq H(\Gamma, \Ind_{\Gamma_{s}}^{\Gamma}\Z)\simeq H(\Gamma_{s},\Z)\ .\]
{Since this identification involves the choice of the base point $s$ it is not reasonable to state any naturality (for morphisms of $\Gamma$-sets) of this equivalence.} 
\end{rem}
 
The following definition and proposition depend on definitions and results of the following sections. They are not needed later and can safely be skipped on first reading.

We can apply \cref{flkewfjwefjeoiwef234} in order to study the  equivariant homology theory $H\Z^{\Gamma} $ on $\Gamma$-topological spaces  induced by the equivariant coarse homology $H\cX^{\Gamma}$ and the {twist} $\Yo^{s}(\Gamma_{can,min})$; see \cref{ioefuwo345345345} and Equation \eqref{fwefgzu2382t78t476t476t746t764t76t} for the notation. 

\begin{ddd}
We define the equivariant ordinary homology theory by
\[H\Z^{\Gamma}:=H\cX^{\Gamma}_{\Yo^s(\Gamma_{can,min})}\cO^{\infty}_{{\homolg}} {\colon \Gamma\Top\to\Sp}\ .\qedhere\]
\end{ddd}

By the results from \cref{wfifjweoifjwoifewfwfewf} the associated equivariant homology theory is determined {on $\Gamma$-CW complexes} by the restriction of the functor $H\Z^{\Gamma}$ to 
  transitive $\Gamma$-sets. The following result describes this functor explicitly.

We consider the following two functors
$\Orb(\Gamma)\to \Sp$ given by
\[S\mapsto  H\Z^{\Gamma}(S)\ , \quad  S\mapsto H(\Gamma,\Z[S])\ .\]

\begin{prop}
For every  transitive $\Gamma$-set $S$ 
we have a natural equivalence   \[H\Z^{\Gamma}(S)\simeq \Sigma H(\Gamma,\Z[S])\]
of spectra.
\end{prop}

\begin{proof}
 By definition we have a natural equivalence 
\begin{align*}
H\Z^{\Gamma}(S) & \simeq H\cX^{\Gamma}(\cO^{\infty}(\cM(\cU(S)))\otimes \Yo^{s}(\Gamma_{can,min}))\\
& \simeq H\cX^{\Gamma}(\cO^{\infty}(S_{disc,max,max})\otimes \Yo^{s}(\Gamma_{can,min}))\ .
\end{align*}
 {By \cref{oiefjewoifo23r4243455435345}, $\cO^\infty(S_{disc,max,max})\simeq \Sigma\Yo^s(S_{min,max})$. This gives} the equivalence
\[ H\Z^{\Gamma}(S)\simeq \Sigma H\cX^{\Gamma}(S_{min,max}\otimes \Gamma_{can,min})\ .\]
By \cref{flkewfjwefjeoiwef234} we have a natural equivalence 
$H\cX^{\Gamma}(S_{min,max}\otimes \Gamma_{can,min})\simeq H(\Gamma,\Z[S])$ which finishes this proof. 
\end{proof}

\subsection{{Additional properties}}\label{seci232323}
 
Recall \cref{fweoiuiowefwefeve} of continuity of an equivariant coarse homology theory.

\begin{lem}
The equivariant coarse homology  theory $H\cX^{\Gamma}$ is continuous.
\end{lem}

\begin{proof}
Let $X$ be a $\Gamma$-bornological coarse space and $\cY:=(Y_{i})_{i\in I}$ be a trapping exhaustion. 
If $c$ is a chain in $C\cX^{\Gamma}_{n}(X)$, then
$\supp(c)$ is a $\Gamma$-invariant subset of $X^{n+1}$ which meets every bounded subset of $X$ in a finite set.
For $i$ in $\{0,\dots,n\}$ let $p_{i}\colon X^{n+1}\to X$ be the projection.
Then we consider the $\Gamma$-invariant subset \[F:=\bigcup_{i=0}^{n} p_{i}(\supp(c))\ .\]
Note that $c$ belongs to the image of the map
$C\cX_{n}^{\Gamma}(F_{X})\to C\cX_{n}^{\Gamma}(X)$
induced by the inclusion of $F$ into $X$.

Observe that $F$ is locally finite.
Hence there exists an index $i$  in $I$ such that $F\subseteq Y_{i}$. We conclude that \[C\cX_{n}^{\Gamma}(X)\cong \colim_{i\in I} C\cX_{n}^{\Gamma}(Y_{i})\ .\] 

The argument above implies that 
 we have an isomorphism of chain complexes
\[C\cX^{\Gamma}(X)\cong \colim_{i\in I} C\cX^{\Gamma}(Y_{i})\ .\] Since the Eilenberg--MacLane correspondence \eqref{fvihgoirvebvevvreveve} preserves filtered colimits we conclude that
\[H\cX^{\Gamma}(X)\simeq \colim_{i\in I} H\cX^{\Gamma}(Y_{i})\]
which was to be shown.
\end{proof}

Recall \cref{foijofifwefwefewfw} of strongness of an equivariant coarse homology theory.
\begin{lem}
The equivariant coarse homology  theory
$H\cX^{\Gamma}$ is strong.
\end{lem}

\begin{proof}
We  can essentially repeat the proof of \cite[Prop.~6.18]{buen}.

Let $f\colon X\to X$ implement weak flasqueness of a $\Gamma$-bornological coarse space $X$.
Then   we can define the chain map  \[S:=\sum_{n=0}^{\infty} C\cX(f^{n})\colon C\cX(X)\to C\cX(X)\ .\] We refer to  \cite[Prop.~6.18]{buen} for the verification that this map is well-defined.
 We then have the identity of endomorphisms of $C\cX^{\Gamma}(X)$
\[\id_{C\cX(X)}+C\cX(f)\circ S=S\ .\]
Applying the Eilenberg--MacLane correspondence $\cE\cM$ this gives
\[\id_{H\cX^{\Gamma}(X)}+H\cX^{\Gamma}(f)\circ \cE\cM(S)=\cE\cM(S)\ .\]
Since we already know that $H\cX^{\Gamma}$ is a coarse homology theory we have the equivalence  $H\cX^{\Gamma}(f) \simeq \id_{H\cX^{\Gamma}(X)}$. Hence we get
$\id_{H\cX^{\Gamma}(X)}+\cE\cM(S)\simeq \cE\cM(S)$, and this implies that we must have $H\cX^{\Gamma}(X)\simeq 0$.
\end{proof}

Recall the definition of the free union of a family of $\Gamma$-bornological coarse spaces which was given in \cref{efwifuhwfowefewfewfwef} and the \cref{rgfou894ut984t3kfnrekjf} of the notion of strong additivity for an equivariant coarse homology theory.
 
\begin{lem}\label{roijwoiere54}
The equivariant coarse homology  theory $H\cX^{\Gamma}$ is strongly additive.
\end{lem}

\begin{proof}
Let $(X_{i})_{i\in I}$ be a family of $\Gamma$-bornological coarse spaces. 
By inspection of the definitions,
\[C\cX^{\Gamma} \Big(  \bigsqcup_{i\in I}^{\free} X_{i} \Big) \cong \prod_{i\in I} C\cX^{\Gamma}(X_{i})\ .\]
We then use that the Eilenberg--MacLane correspondence \eqref{fvihgoirvebvevvreveve} preserves products.
\end{proof}

\subsection{Change of groups}

In this section we provide natural transformations which relate the equivariant coarse homology theory with the change of group functors considered in \cref{roigurgoregregreg}.
 
Let $\iota\colon H\to \Gamma$ be a homomorphism of groups. Let $X$ be a $\Gamma$-bornological coarse space.
 Since every $\Gamma$-invariant chain is $H$-invariant we 
 have an inclusion \[C\cX^{\Gamma}(X)\hookrightarrow C\cX^{H}(\Res^{\Gamma}_{H}X)\ .\] This gives a natural transformation between  equivariant coarse homology theories
\begin{equation}
\res^{\Gamma}_{H}\colon H\cX^{\Gamma}\to H\cX^{H}\circ \Res^{\Gamma}_{H}\ .
\end{equation}

We let the subset $X^{\prime}$ of $X$ be a set of representatives of $H$-orbits in $X$. Then we have a restriction map
\[r\colon C\cX_{n}(X)\to \Z^{X^{\prime}\times X^{n}}{\ ,\quad c \mapsto c|_{X^\prime \times X^n}}\ .\] 
We want to define a homomorphism
\begin{equation}\label{hcwekchewuchiuh}C\cX^{\Gamma}_{n}(X)\to C\cX^{W_{\Gamma}(H)}_{n}(Q_{H}(X))\end{equation} by linear extension of the  projection map $X^{\prime}\times X^{n}\to (H\backslash X)^{n+1}$ composed with the restriction $r$.
We will argue now that this is well-defined, i.e., that the sums appearing in this extension are finite. For $x$ in $X$ we denote by $[x]$    its orbit in $ H\backslash X$.
We consider a point $([x_{0}],\dots,[x_{n}])$ in  $(H\backslash X)^{n+1}$. Then $[x_{0}]\cap X^{\prime}$ consists of a unique point $x_{0}$.  

Let $c$ be in $C\cX_{n}(X)$ and assume that $c$ is $U$-controlled for  some entourage $U$ of $X$. Then we have \[\supp(c)\cap \{x_{0}\}\times X^{n}\subseteq 
(U[x_{0}])^{n+1}\ .\] Since  $U[x_{0}]$ is bounded the number of points of $\supp(c)$ which meet
$U[x_{0}]$ is finite. 

One checks that the  homomorphism \eqref{hcwekchewuchiuh} does not depend on the choice of the set of representatives $X^{\prime}
$ and has values in $W_{\Gamma}(H)$-invariant chains.
 Furthermore, it is compatible
with the differential and takes values in controlled and locally finite chains. 

We therefore get a natural transformation
\begin{equation}
q_{H}(\iota)\colon H\cX^{\Gamma}\to H\cX^{W_{\Gamma}(H)}\circ Q_{H}(\iota)\ .
\end{equation}

 Recall from \cref{wefiwuofwefwegwgewrgwe}   that $P_{\hat \Gamma}(X)= \Gamma_{min,min}\otimes X$ with the action of $\Gamma\times H$ given  by
 $(\gamma,h)(\gamma^{\prime},x)=(\gamma \gamma^{\prime}\iota(h)^{-1},hx)$.
We define a morphism  of chain complexes
\[ C\cX_{n}^{H}\to C\cX_{n}^{\Gamma\times H}(P_{\hat \Gamma}(X)) \]
by linear extension of the map
\[(x_{0},\dots,x_{n})\mapsto \sum_{\gamma^{\prime}\in \Gamma} ((\gamma^{\prime},x_{0}),(\gamma^{\prime},x_{1}),\dots, (\gamma^{\prime},x_{n}))\ .\]
In this way we get a transformation
\begin{equation}
\hat p_{\Gamma}\colon H\cX^{H}\to H\cX^{\Gamma\times H}\circ \hat P_{\Gamma}\ .
\end{equation}

We finally get a natural transformation \begin{equation}\label{regeg43t34434} \indd_{H}^{\Gamma}\simeq q_{H}(\kappa)\circ \hat p_{\Gamma}\colon H\cX^{H}\to H\cX^{\Gamma}\circ \Ind_{H}^{\Gamma}\ ,\end{equation}
 
where $\kappa\colon H\to \Gamma\times H$ is the inclusion of the second factor.

\begin{prop}If $\iota\colon H\to \Gamma$ is injective, then the transformation \eqref{regeg43t34434}
 is an equivalence of $H$-equivariant coarse homology theories.
\end{prop}
 \begin{proof}	
	By \cref{rem:jtjbzuzgkvrgkieilfu}  the map
	$X\to  \Gamma\times_{H}X=\Ind_H^\Gamma (X)$ given by $x\mapsto [1,x]$	is an embedding of an $H$-invariant coarse component.
	Restriction along this map gives a map of chain complexes  $C\cX^\Gamma(\Ind_H^\Gamma (X)) \to C\cX^H(X)$ which induces the inverse  to \eqref{regeg43t34434}.
\end{proof}

\section{Equivariant coarse algebraic \texorpdfstring{$\boldsymbol{K}$}{K}-homology}

In this section, we define for every additive category $\bA$ with $\Gamma$-action its $\Gamma$-equivariant coarse algebraic $K$-homology
\[ K\bA\cX^\Gamma \colon \Gamma\BC \to \Sp\ .\]
The construction associates to a $\Gamma$-bornological coarse space $X$ an additive category  of equivariant $X$-controlled $\bA$-objects $\bV_{\bA}^{\Gamma}(X)$, and defines $K\bA\cX^\Gamma$ to be the (non-connective) algebraic $K$-theory spectrum of this category.

\subsection{The algebraic \texorpdfstring{$K$}{K}-theory functor}

We describe the properties   of the $K$-theory functor that we will use subsequently. See \cite[Sec.~2.1]{MR2030590} for similar statements.
Let $\Add$ denote the $1$-category of small additive categories and exact functors.
In the following, all additive categories will be small so that we can omit this adjective safely.

The $K$-theory functor is a functor 
\[ K \colon \Add \to \Sp \]
which has the following properties (we will recall the occurring notions further below):
\begin{enumerate}
\item (Normalization) \label{it:Kprop.normalization} It sends (a skeleton of) the additive category of finitely generated free modules over a ring $R$ to the non-connective $K$-theory (see e.g. \cite{MR2079996}) of that ring.
\item(Invariance) \label{it:Kprop.invariance} It sends  isomorphic exact functors to equivalent maps.
\item (Colimits) \label{it:Kprop.colimits} If $\bA = \colim_{i} \bA_{i}$ is a filtered colimit of additive subcategories, then the natural map $\colim_{i} K(\bA_{i}) \xrightarrow{\simeq} K(\bA)$ is an equivalence.
\item (Additivity) \label{it:Kprop.additivity} If $\Phi,\Psi \colon \bA \to \bA^{\prime}$ are exact functors between additive categories, then we have an equivalence $K(\Phi) + K(\Psi) \simeq K(\Phi \oplus \Psi)$ of morphisms between $K$-theory spectra.
\item(Exactness)\label{it:Kprop.exactness} If $\bA$ is a Karoubi filtration of $\bC$, then we have a fiber sequence
\[ K(\bA) \to K(\bC) \to K(\bC/\bA) \xrightarrow{\partial} K(\bA) \ .\]
\item(Products)\label{it:Kprop.products} If we have a family $(\bA_i)_{i \in I}$ of additive categories, then the natural map $K(\prod_{i \in I} \bA_i) \to \prod_{i \in I} K(\bA_i)$ is an equivalence.
\item (Flasqueness) \label{it:Kprop.flasqueness} It sends flasque additive categories to zero.
\end{enumerate}
Let us recall the definition of some notions appearing above.

A category is additive if admits a zero object and biproducts such that the operation $\Hom(A,B) \times \Hom(A,B) \to \Hom(A,B)$ which sends $f,g$ to
\[ f + g \colon A \xrightarrow{\Delta} A \oplus A \xrightarrow{f \oplus g} B \oplus B \xrightarrow{\id + \id} B \]
defines an abelian group structure on morphism sets.
In applications, it is useful to consider the equivalent characterization of additive categories as 
categories  which are enriched over abelian groups, have a zero object and admit finite coproducts {Lurie \cite[Sec.~1.1.2]{HA}.}

A morphism between additive categories is a  functor between the underlying categories which preserves the zero object and finite coproducts. Equivalently, one can require that the functor is compatible with the enrichement in abelian groups, see Mac Lane \cite[Prop.~VIII.2.4]{maclane}.

Given morphisms $\Phi,\Psi \colon \bA \to \bA^{\prime}$ between additive categories, we define a new morphism $\Phi \oplus \Psi \colon \bA \to \bA^{\prime}$ by choosing for every object $A$ of $\bA$ an object of $\bA^{\prime}$ representing the sum $\Phi(A) \oplus \Psi(A)$.
Since the sum of two functors is unique up to unique isomorphism, there is an essentially unique map $K(\Phi \oplus \Psi)$ by virtue of Property \eqref{it:Kprop.invariance}.

\begin{ddd}\label{def:flasquecat}
An additive category $\bA$ is called \emph{flasque} if there exists a functor $\Sigma \colon \bA \to \bA$ such that $\id_{\bA} \oplus \Sigma \cong \Sigma$.
\end{ddd}

Note that Property \eqref{it:Kprop.additivity} implies Property \eqref{it:Kprop.flasqueness}:
Assume that  $\bA$ is flasque and that $\Sigma \colon \bA \to \bA$ is a functor  satisfying $\id_\bA \oplus \Sigma \cong \Sigma$.  By Property \ref{it:Kprop.additivity} we then have an equivalence  $K(\Sigma) + \id_{K(\bA)} \simeq K(\Sigma)$, which implies that
$K(\bA)\simeq 0$.

Let $\bA \subseteq \bC$ be a full additive subcategory.
For $C,D$ in $\bC$, let $\Hom_\bC(C,\bA,D)$ denote the set of all morphisms in $\Hom_\bC(C,D)$ which factor through some object in $\bA$.
Then $\Hom_\bC(C,\bA,D)$ is a subgroup of $\Hom_\bC(C,D)$.
We let $\bC/\bA$ be the category with the same objects as $\bC$ and whose morphisms are given by
\[ \Hom_{\bC/\bA}(C,D) := \Hom_\bC(C,D) / \Hom_\bC(C,\bA,D)\ .\]
Note that $\bC/\bA$ is again an additive category.
It has the universal property that exact functors on $\bC/\bA$ correspond bijectively to exact functors on $\bC$ which vanish on $\bA$.

Let $\bC$ be an  additive category.
\begin{ddd}\label{def:karoubi}
 The inclusion $\bA \subseteq \bC$ of a full additive subcategory is a \emph{Karoubi filtration}
 if every diagram 
 \[ A \xrightarrow{f} C \xrightarrow{g} B \]
 in $\bC$, where $A,B$ are objects of $\bA$, admits an extension to a commutative diagram
 \[\xymatrix{
  A \ar[r]^{f}\ar[d] & C\ar[r]^{g}\ar[d]^{\cong} & B \\
  D\ar@{>->}[r]^-{inc} & D \oplus D^{\perp}\ar@{->>}[r]^-{\pr} & D\ar[u]
 }\]
 for some object $D$ of $\bA$.	
\end{ddd}

In \cite[Lem.~5.6]{KasFDC} it is shown that  \cref{def:karoubi} is equivalent to the standard definition of a Karoubi filtration as considered in \cite{MR1469141}.

An algebraic $K$-theory functor as described here can be furnished by restricting the $K$-theory functor constructed by Schlichting \cite{MR2206639} to additive categories.
For Property~\eqref{it:Kprop.normalization}, see \cite[Thms.~5 \& 8]{MR2206639}.
Property~\eqref{it:Kprop.invariance} is so elementary that it is rarely stated explicitly, but can be easily read off from the construction in \cite{MR2206639}.
Property~\eqref{it:Kprop.colimits} is a combination of \cite[Eq.~(9)]{MR0338129} and \cite[Cor.~5]{MR2206639}.
Property~\eqref{it:Kprop.additivity} follows from Property~\eqref{it:Kprop.exactness} (cf.~also \cite[Prop.~1.3.2]{MR802796}); the latter is proved in \cite[Thm.~2.10]{MR2079996}.
Property~\eqref{it:Kprop.products} is shown in \cite[Thm~1.2]{Kthprod}; for connective $K$-theory, this is originally due to Carlsson \cite{MR1351941}.

\subsection{\texorpdfstring{$X$}{X}-controlled \texorpdfstring{$A$}{A}-objects}\label{ffopjiofr23r2u23oif}

Let $X$ be a bornological coarse space with the bornology $\cB$ and the coarse structure $\cC$,
and let $\bA$ be an additive category with a (strict) $\Gamma$-action. The 
poset $\cB$   is ordered by the subset inclusion and will be regarded as a category.
 {For a functor $A\colon \cB\to\bA$ we define $\gamma A\colon \cB\to \bA$ to be the  functor sending a bounded set $B$ to $\gamma(A(\gamma^{-1}(B)))$.}

\begin{ddd}\label{def:Xcontrolledobject}
An \emph{equivariant $X$-controlled $\bA$-object} is a pair $(A,\rho)$ consisting of a functor
 $A \colon \cB \to \bA$
and a family $\rho=(\rho(\gamma))_{\gamma\in \Gamma}$ of natural isomorphisms 
 $\rho(\gamma)\colon A \to \gamma A$
 satisfying the following conditions:
 \begin{enumerate}
  \item\label{def:Xcontrolledobject:it1} $A(\emptyset) \cong 0$. 
  \item\label{def:Xcontrolledobject:it3} For all $B, B'$ in $\cB$, the commutative square
  \[\xymatrix{
   A(B \cap B')\ar[r]\ar[d] & A(B)\ar[d] \\
   A(B')\ar[r] & A(B \cup B')
  }\]
  is a push-out.
  \item\label{def:Xcontrolledobject:it4} For all $B$ in $\cB$, there exists a finite subset $F$ of $B$ such that the inclusion $F\subseteq B$ induces an isomorphism   $A(F) \xrightarrow{\cong} A(B)$.    \item \label{def:Xcontrolledobject:it5}For all pairs of elements  $\gamma,\gamma^{\prime}$ of $      \Gamma$ we have 
  $\rho( { \gamma} \gamma^{\prime})= \gamma\rho(\gamma^{\prime})\circ \rho(\gamma)$, where $\gamma\rho(\gamma^{\prime})$ is the natural transformation from $\gamma A$ to $ {\gamma \gamma^{\prime}} A$ induced from $\rho(\gamma^\prime)$.
  \qedhere
 \end{enumerate}
\end{ddd}

Let $(A,\rho)$ be an equivariant $X$-controlled $\bA$-object.

\begin{lem}\label{lem:Xcontrolled.props}
 \begin{enumerate}
  \item\label{lem:Xcontrolled.props.1}  \begin{enumerate}\item The canonical morphism
    \[ \bigoplus_{x \in F} A(\{x\}) \xrightarrow{\sum_{x \in F} A(\{x\} \subseteq F)} A(F) \]
    is an isomorphism for every finite subset $F$ of $X$. \item 
    For two finite subsets $F ,F^{\prime}$  of $X$ with $F\subseteq F'$ 
    we have a commuting square \[\xymatrix@C=7em{\bigoplus_{x \in F} A(\{x\})\ar[d]\ar[r]^-{\sum_{x \in F} A(\{x\} \subseteq F)}&A(F)\ar[d]\\ \bigoplus_{x \in F'} A(\{x\})\ar[r]^-{\sum_{x \in F^{\prime}} A(\{x\} \subseteq F^{\prime})} &A(F')}\ .\]
  \end{enumerate}
  \item\label{lem:Xcontrolled.props.2} \begin{enumerate} \item \label{efoiwfewff245} For a   bounded subset $B$ of $X$  there exists a unique minimal finite subset $F_B$ of $B$  such that  $A(F_B) \to A(B)$ is an isomorphism.
    \item If $B$ and $F_{B}$  are as in \ref{efoiwfewff245}, then for any subset $B'$ of $X$ with $F_B \subseteq B' \subseteq B$, the morphisms $A(F_B) \to A(B')$ and $A(B') \to A(B)$ are isomorphisms.  \end{enumerate}
 \end{enumerate}
\end{lem}
\begin{proof}
 \cref{lem:Xcontrolled.props.1} is a direct consequence of \cref{def:Xcontrolledobject}\eqref{def:Xcontrolledobject:it1} 
 and \eqref{def:Xcontrolledobject:it3}.
 
 Suppose now that $F, F'$ are two finite subsets of $B$ as in \cref{def:Xcontrolledobject}\eqref{def:Xcontrolledobject:it4}.
   Then we obtain a push-out square
   \[\xymatrix{
    A(F \cap F')\ar[r]\ar[d] & A(F)\ar[d]^{\cong} \\
    A(F')\ar[r]^{\cong} & A(B)
   }\]
  in which all arrows are inclusions of direct summands. It follows that $A(F \cap F') \to A(B)$ is also an isomorphism. This implies the existence of $F_B$.
  
  For a bounded subset $B'$ with $F_B \subseteq B' \subseteq B$, inspection of a similar push-out square implies that $A(F_B) \to A(B')$ is an isomorphism, and the claim follows.
\end{proof}

Let $(A,\rho)$ be an equivariant $X$-controlled $\bA$-object.

\begin{ddd}
 The function $\sigma$ which sends a bounded subset $B$ of $X$ to the finite subset $F_B$ from \cref{lem:Xcontrolled.props} is called the \emph{support function} of $(A,\rho)$.
\end{ddd}

The support function is an order preserving, equivariant function from $\cB$ to the set of finite subsets of $X$ with the property that $\sigma(\sigma(B)) = \sigma(B)$ for every bounded subset $B$.

Let $(A,\rho), (A',\rho')$ be equivariant $X$-controlled $\bA$-objects and let $U$ be an invariant entourage of $X$.  
\begin{ddd}
 An \emph{equivariant $U$-controlled morphism} $f \colon (A,\rho) \to (A',\rho')$ is a natural transformation
 \[ f \colon A(-) \to A'(U[-])\ ,\]
 such that $\rho'(\gamma)\circ f=(\gamma f)\circ \rho(\gamma)$ for all elements $\gamma$ of $\Gamma$.
\end{ddd}
 We let $\Mor_U((A,\rho),(A',\rho'))$ denote the set of equivariant $U$-controlled morphisms.
Furthermore, we define the set of controlled morphisms from $A$ to $A'$ by
 \[ \Hom_{\bV^{\Gamma}_\bA(X)}((A,\rho),(A',\rho')) := \colim_{U \in  {\cC^\Gamma}} \Mor_U((A,\rho),(A',\rho'))\ .\]
 We denote the resulting category of equivariant $X$-controlled $\bA$-objects and equivariant controlled morphisms by $\bV_\bA^\Gamma(X)$.
 
We observe that the composition of a $U$-controlled and a $U^{\prime}$-controlled morphism is a $U\circ U^{\prime}$-controlled morphism. We conclude that  
composition in $\bV^\Gamma_\bA(X)$ is well-defined.   
\begin{lem}
 The category $\bV^\Gamma_\bA(X)$ is additive.
\end{lem}
\begin{proof}
 Let $(A,\rho),(A',\rho')$ be equivariant $X$-controlled $\bA$-objects. Denote by $A \oplus A'$ their direct sum in $\Fun(\cB,\bA)$.
 Note that $(A \oplus A',\rho\oplus\rho')$ is an $X$-controlled $\bA$-object because finite unions of bounded sets are bounded.
 Since finite unions of coarse entourage are coarse entourages and there are canonical isomorphisms of $\Gamma$-sets   \[ \Nat(A \oplus A', C \circ U[-]) \cong \Nat(A, C \circ U[-]) \times \Nat(A', C \circ U[-]) \]
 and
 \[ \Nat(C, (A \oplus A') \circ U[-]) \cong \Nat(C, A \circ U[-]) \times \Nat(C, A' \circ U[-]) \]
 (we use the symbol $\Nat$ to denote the morphism sets in $\Fun(\cB,\bA)$)
 for all $U$ in  {$\cC^\Gamma$} and equivariant $X$-controlled $\bA$-objects $C$, it follows that $(A \oplus A',\rho\oplus \rho')$ is also a direct sum in $\bV^{\Gamma}_\bA(X)$.
 
Similarly, addition of morphisms in $\Fun(\cB,\bA)$ induces the addition operation of morphisms in $\bV^\Gamma_\bA(X)$.
\end{proof}

Next we discuss the functoriality of $\bV^\Gamma_\bA(X)$ in the variables $X$ and $\bA$.
For a $\Gamma$-equivariant exact functor $\Phi \colon \bA \to \bA'$ of additive categories, there exists an induced exact functor $\bV^\Gamma_\Phi(X) \colon \bV^\Gamma_\bA(X) \to \bV^\Gamma_{\bA'}(X)$ which sends an object $(A,\rho)$ to $(\Phi \circ A,\Phi(\rho))$.
Therefore, we have a functor
\[ \bV^\Gamma_{-}(X) \colon \Fun(B\Gamma,\Add) \to \Add\ .\]
Let $\phi \colon (X,\cB,\cC) \to (X',\cB',\cC')$ be a morphism of $\Gamma$-bornological coarse spaces,
and let $(A,\rho)$ be an equivariant $X$-controlled $\bA$-object.
Since $\phi$ is proper, we can define a functor $\phi_*A \colon \cB' \to \bA$ by
\[ \phi_*A(B) := A(\phi^{-1}(B))\ ,\]
and we define
\[\phi_*\rho(\gamma)(B)=\rho(\gamma)(\phi^{-1}(B))\ .\]
All properties of \cref{def:Xcontrolledobject} except \eqref{def:Xcontrolledobject:it4} are immediate.
To see that \eqref{def:Xcontrolledobject:it4} also holds, we note that $\sigma(\phi^{-1}(B)) \subseteq \phi^{-1}(\phi(\sigma(\phi^{-1}(B))) \subseteq \phi^{-1}(B)$ and apply \cref{lem:Xcontrolled.props} to see that $\phi_*A(\phi(\sigma(\phi^{-1}(B)))) \to \phi_*A(B)$ is an isomorphism.

Let $f \colon (A,\rho) \to (A',\rho')$ be an equivariant $U$-controlled morphism. Then there exists some $V$ in {$\cC^{\prime,\Gamma}$} such that $(\phi \times \phi)(U) \subseteq V$. Then $U[\phi^{-1}(B)] \subseteq \phi^{-1}(V[B])$ for all bounded subsets $B$ of $X$, so we obtain an induced $V$-controlled morphism
\[ \phi_*f = \{ f_{\phi^{-1}(B)} \colon \phi_*A(B) \to \phi_*A(V[B]) \}_{B \in \cB'}\ .\]
This defines a functor
\[ \phi_* \colon \bV^\Gamma_\bA(X) \to \bV^\Gamma_\bA(X')\ .\]
We thus have constructed a functor
\[ \bV^\Gamma_\bA \colon \Gamma\BC \to \Add\ .\]

\subsection{Coarse algebraic \texorpdfstring{$K$}{K}-homology}

Let $\Gamma$ be a group and $\bA$ be an additive category with a $\Gamma$-action.

\begin{ddd}
We define the \emph{coarse algebraic $K$-homology $K\bA\cX^\Gamma$} associated to $\bA$ as
 \[ K\bA\cX^\Gamma := K \circ \bV_\bA^\Gamma \colon \Gamma\BC \to \Sp\ .\qedhere\]
\end{ddd}

This section discusses the homological properties of $K\bA\cX^\Gamma$.
Our first goal is to prove the following theorem.

\begin{theorem}\label{thm:coarseK.homology}
 The functor $K\bA\cX^\Gamma$ is an equivariant coarse homology theory.
\end{theorem}

We divide the proof of \cref{thm:coarseK.homology} into a sequence of lemmas.

\begin{lem}\label{lem:coarseK.ucontinuity}
 The functor $K\bA\cX^\Gamma$ is $u$-continuous.
\end{lem}
\begin{proof}
 Let $X$ be a $\Gamma$-bornological coarse space, and let $U$ be an invariant entourage of $X$.
 The natural map $X_U \to X$ induces a functor $\Phi_U \colon \bV_\bA^\Gamma(X_U) \to \bV_\bA^\Gamma(X)$.
 Since the definition of equivariant $X$-controlled $\bA$-objects is independent of the coarse structure,
 $\Phi_U$ is the identity on objects.
 Additionally, since inclusions of direct summands are monomorphisms, $\Phi_U$ is faithful.
 
 This allows us to view $\bV_\bA^\Gamma(X_U)$ as a subcategory of $\bV_\bA^\Gamma(X)$, and we have
 \[ \bV_\bA^\Gamma(X) = \bigcup_{U \in \cC^\Gamma} \bV_\bA^\Gamma(X_{U}) \]
 since every morphism in $ \bV_\bA^\Gamma(X)$  is $U$-controlled for some $U$ in $\cC^\Gamma$.
 
 Since the algebraic $K$-theory functor is compatible with filtered colimits (Property~\eqref{it:Kprop.colimits}),
 the claim of the lemma follows.
\end{proof}

Let  $\phi,\psi \colon X \to X'$ be morphisms of $\Gamma$-bornological coarse spaces.
\begin{lem}\label{lem:coarseK.closemaps}
If  $\phi$ and $\psi$ are close, then $\phi_*$ and $\psi_*$ are isomorphic.
\end{lem}

\begin{proof}
 Let $U^{\prime}$ be a symmetric entourage of $X^{\prime}$ containing the diagonal such that $(\phi(x),\psi(x))$ lies in $U^{\prime}$ for all $x$ in $X$.
 Note that this implies $\phi^{-1}(B^{\prime}) \subseteq \psi^{-1}(U^{\prime}[B^{\prime}])$ and $\psi^{-1}(B^{\prime}) \subseteq \phi^{-1}(U^{\prime}[B^{\prime}])$ for all bounded subsets $B^{\prime}$  of $X^{\prime}$.
 
Let $(A,\rho)$ be an equivariant $X$-controlled $\bA$-object. The maps
\[A(\phi^{-1}(B^{\prime})) \to A(\psi^{-1}(U^{\prime}[B^{\prime}]))\]
define a natural morphism $f \colon \phi_*A \to \psi_*A$, and similarly we have a natural morphism $g \colon \psi_*A \to \phi_*A$.
 Since the composition $g \circ f$ is given by the natural transformation
 \[ \{ A(\phi^{-1}(B^{\prime}) \subseteq \phi^{-1}((U^{\prime})^2[B^{\prime}])) \colon \phi_*A \to \phi_*A \circ (U^{\prime})^{2}[-] \}_{B^{\prime} \in \cB^{\prime}}\ ,\]
 we have $g \circ f = \id_{\phi_*A}$. Similarly, $f \circ g = \id_{\psi_*A}$.
 It follows that $\phi_* \cong \psi_*$.
\end{proof}

\begin{kor}\label{kor:coarseK.coarseinvariance}
 The functor $K\bA\cX^\Gamma$ is coarsely invariant.
\end{kor}
\begin{proof}
 This is a direct consequence of \cref{lem:coarseK.closemaps} together with the Property~\eqref{it:Kprop.invariance} of the algebraic $K$-theory functor.
\end{proof}

\begin{lem}\label{lem:coarseK.flasqueness}
 The functor $K\bA\cX^\Gamma$ vanishes on flasque $\Gamma$-bornological coarse spaces.
\end{lem}
\begin{proof}
 Let $X$ be a $\Gamma$-bornological coarse space with flasqueness implemented by $\phi \colon X \to X$.
 We claim that the functor
 \[ \Sigma := \bigoplus_{n \in \IN} (\phi^n)_* \colon \bV_\bA^\Gamma(X) \to \bV_\bA^\Gamma(X) \]
 is well-defined (up to canonical isomorphism).
 
 For every bounded subset $B$ of $X$, there exists some $n$ in $\IN$ such that $(\phi^n)^{-1}(B) = \emptyset$,
 so the direct sum $\bigoplus_{n \in \IN} (\phi^n)_*A$ exists for every equivariant $X$-controlled object $(A,\rho)$.
 Let $f \colon (A,\rho) \to (A',\rho')$ be a $U$-controlled morphism.
 Then $\bigoplus_{n \in \IN} (\phi^n)_*f$ is $V$-controlled, where $V:=\bigcup_{n \in \IN} (\phi \times \phi)^n(U)$ is again a coarse  entourage of $X$ by assumption on $\phi$.
 So $\Sigma$ is an exact functor.
 
 Since $\phi$ is close to $\id_X$, we conclude from \cref{lem:coarseK.closemaps} that $\phi_* \circ \Sigma$ and $\Sigma$ are  isomorphic.
 Hence, \[\id_{\bV_\bA^\Gamma(X)} \oplus \Sigma \cong \id_{\bV_\bA^\Gamma(X)} \oplus (\phi_* \circ \Sigma) \cong \Sigma\ ,\]
 and we deduce the lemma from Property~\eqref{it:Kprop.flasqueness} of the algebraic $K$-theory functor.
\end{proof}

Let $X$ be a $\Gamma$-bornological coarse space and $Z$ a $\Gamma$-invariant subset of $X$. For an equivariant $X$-controlled object $(A,\rho)$, we denote by $(A|_Z,\rho|_Z)$ the restriction to $Z$, that is $A|_Z$ is the restriction of $A$ to $\cB\cap Z$ and $\rho|_Z$ the appropriate restriction of $\rho$.

Let $X$ be a $\Gamma$-bornological coarse space and let $\cY = (Y_i)_{i \in I}$ be an equivariant big family in $X$.
For each $i$ in $I$, the canonical exact functor $\bV_\bA^\Gamma(Y_i) \to \bV_\bA^\Gamma(X)$ is injective on objects and fully faithful,
so we can regard $\bV_\bA^\Gamma(Y_i)$ as a full subcategory of $\bV_\bA^\Gamma(X)$. Define
\[ \bV_\bA^\Gamma(\cY) := \bigcup_{i \in I} \bV_\bA^\Gamma(Y_i)\ ,\]
considered as a full subcategory of $\bV_\bA^\Gamma(X)$.

\begin{lem}\label{lem:bigfamily.karoubifiltration}
 The inclusion $\bV_\bA^\Gamma(\cY) \to \bV_\bA^\Gamma(X)$ is a Karoubi filtration.
\end{lem}
\begin{proof}
 Let $(A,\rho), (A',\rho')$ be objects in $\bV_\bA^\Gamma(\cY)$, $(C,\rho_C)$ be an object in $\bV_\bA^\Gamma(X)$, and let $f \colon A \to C$ and $g \colon C \to A'$ be morphisms.
 Choose $i$ in $I$ such that both $A$ and $A'$ are objects in $\bV_\bA^\Gamma(Y_i)$, and pick an invariant and symmetric entourage $U$ which contains the diagonal such that $f$ and $g$ are $U$-controlled.
 Let $j$ in $I$ be such that $U[Y_i] \subseteq Y_j$.
 Since $f$ is a natural transformation $A \to C \circ U[-]$ and $g|_{C|_{X \setminus Y_j}} = 0$, the following diagram commutes:
 \[\xymatrix{
  A\ar[r]^-{f}\ar[d] & C\ar[d]^{\cong}\ar[r]^-{g} & A' \\
  C|_{Y_j}\ar[r]^-{inc} & C|_{Y_j} \oplus C|_{X \setminus Y_j}\ar[r]^-{\pr} & C|_{Y_j}\ar[u]
 }\]
 Hence, the inclusion $\bV_\bA^\Gamma(\cY) \to \bV_\bA^\Gamma(X)$ is a Karoubi filtration.
\end{proof}

\begin{prop}\label{prop:coarseK.excision}
 The functor $K\bA\cX^\Gamma$ is excisive.
\end{prop}
\begin{proof}
 The category of $\emptyset$-controlled $\bA$-objects is the zero category, which has trivial $K$-theory by Property~\eqref{it:Kprop.normalization}.

 Let $X$ be a $\Gamma$-bornological coarse space, and let $(Z,\cY)$ be an equivariant complementary pair on $X$.
 Both inclusions $\bV_\bA^\Gamma(Z \cap \cY) \to \bV_\bA^\Gamma(Z)$ and $\bV_\bA^\Gamma(\cY) \to \bV_\bA^\Gamma(X)$ are Karoubi filtrations by \cref{lem:bigfamily.karoubifiltration}.
 Therefore, we obtain by Property~\eqref{it:Kprop.exactness} of the algebraic $K$-theory functor a map of fiber sequences
 \[\xymatrix{
  K(\bV_\bA^\Gamma(Z \cap \cY))\ar[r]\ar[d] & K(\bV_\bA^\Gamma(Z))\ar[r]\ar[d] & K(\bV_\bA^\Gamma(Z)/\bV_\bA^\Gamma(Z \cap \cY))\ar[d]\ar[r]^-{\partial} & K(\bV_\bA^\Gamma(Z \cap \cY))\ar[d] \\
  K(\bV_\bA^\Gamma(\cY))\ar[r] & K(\bV_\bA^\Gamma(X))\ar[r] & K(\bV_\bA^\Gamma(X)/\bV_\bA^\Gamma(\cY))\ar[r]^-{\partial} & K(\bV_\bA^\Gamma(\cY))
 }\]
 Consider the induced exact functor $\Phi \colon \bV_\bA^\Gamma(Z)/\bV_\bA^\Gamma(Z \cap \cY) \to \bV_\bA^\Gamma(X)/\bV_\bA^\Gamma(\cY)$.
 
 Let $(A,\rho)$ in $\bV_\bA^\Gamma(X)$ and consider the natural morphisms $f \colon (A|_Z,\rho|_Z) \to (A,\rho)$ and $p \colon (A,\rho) \to (A|_Z,\rho|_Z)$. Clearly, $pf = \id_{(A|_Z,\rho|_Z)}$.
 Pick $i$ in $I$ such that $X \setminus Z \subseteq Y_i$.
 Then $\id_{(A,\rho)} - fp$ factors through $(A|_{Y_i},\rho|_{Y_i})$, so $f$ and $p$ define mutually inverse isomorphisms in $\bV_\bA^\Gamma(X)/\bV_\bA^\Gamma(\cY)$.
 We conclude that $\Phi \circ \Psi \cong \id_{\bV_\bA^\Gamma(X)/\bV_\bA^\Gamma(\cY)}$, so $\Phi$ is an equivalence of categories.
 
 It follows from Property~\eqref{it:Kprop.invariance} of the algebraic $K$-theory functor that
 \[\xymatrix{
  K(\bV_\bA^\Gamma(Z \cap \cY))\ar[r]\ar[d] & K(\bV_\bA^\Gamma(Z))\ar[d] \\
  K(\bV_\bA^\Gamma(\cY))\ar[r] & K(\bV_\bA^\Gamma(X))
 }\]
 is a push-out.
 By Property~\eqref{it:Kprop.colimits} of algebraic $K$-theory we have $K(\bV_\bA^\Gamma(\cY)) \simeq K\bA\cX^\Gamma(\cY)$. 
 This proves excision.
\end{proof}

\begin{rem}\label{rem:excision.projection}
Let $(X,\cB,\cC)$ be a $\Gamma$-bornological coarse space and let $Y$ be a $\Gamma$-invariant subspace of $X$ with the property that $U[Y]=Y$ for every $U$ in $\cC$. Then $(Y, X\setminus Y)$ is a coarsely excisive pair. 
Inspecting the proof of \cref{prop:coarseK.excision}, we obtain the following commutative diagram:
\[\xymatrix{
	0\ar[r]\ar[d] & K(\bV_\bA^\Gamma(Y))\ar[r]^-{\simeq}\ar[d] & K(\bV_\bA^\Gamma(Y))\ar[d]^-{\simeq} \\
	K(\bV_\bA^\Gamma(X\setminus Y))\ar[r] & K(\bV_\bA^\Gamma(X))\ar[r] & K(\bV_\bA^\Gamma(X)/\bV_\bA^\Gamma(X\setminus Y))
}\]
In addition, we observe that an inverse to the right vertical equivalence is induced by the functor $\Psi$ which is given by $\Psi(A,\rho) = (A|_{Y},\rho|_{Y})$.
Since this functor is already well-defined as a functor $\Psi \colon \bV_\bA^\Gamma(X) \to \bV_\bA^\Gamma(Y)$, we see that the projection map
\[ K(\bV_\bA^\Gamma(X)) \simeq K(\bV_\bA^\Gamma(X\setminus Y)) \oplus K(\bV_\bA^\Gamma(Y)) \to K(\bV_\bA^\Gamma(Y)) \]
arising from excision coincides with $K(\Psi)$.
\end{rem}

\cref{thm:coarseK.homology} follows now by combining \cref{lem:coarseK.ucontinuity}, \cref{kor:coarseK.coarseinvariance}, \cref{lem:coarseK.flasqueness} and \cref{prop:coarseK.excision}.
In the remainder of this section, we establish some additional properties of the equivariant coarse homology theory $K\bA\cX^\Gamma$.
{For the next two propositions recall} the notions of continuity (\cref{fweoiuiowefwefeve}) and strongness (\cref{foijofifwefwefewfw}).

\begin{prop}\label{prop:coarseK.continuous}
 The equivariant coarse homology theory $K\bA\cX^\Gamma$ is continuous.
\end{prop}
\begin{proof}
 Let $(A,\rho)$ be an equivariant $X$-controlled object. Set $S := \{ x \in X \mid A(\{x\}) \ncong 0 \}$.
 By definition, we have $S \cap B = \sigma(B)$, where $\sigma$ is the support function of $A$.
 Hence, $S$ is a locally finite subset of $X$, so $(A,\rho)$ lies in the full subcategory $\bV_\bA^\Gamma(S)$ of $\bV_\bA^\Gamma(X)$.
 This shows that $\bV_\bA^\Gamma(X) = \bigcup_{S \subseteq X\ \text{locally finite}} \bV_\bA^\Gamma(S)$.
 By Property~\eqref{it:Kprop.colimits} of the algebraic $K$-theory functor, it follows that $K\bA\cX^\Gamma$ is continuous.
\end{proof}

\begin{prop}\label{prop:coarseK.strong}
 The equivariant coarse homology theory $K\bA\cX^\Gamma$ is strong.
\end{prop}
\begin{proof}
 Let $X$ be a $\Gamma$-bornological coarse space with weak flasqueness implemented by $\phi \colon X \to X$.
 As in the proof of \cref{lem:coarseK.flasqueness}, the functor $\Sigma \colon \bV_\bA^\Gamma(X) \to \bV_\bA^\Gamma(X)$ given by
 \[ \Sigma := \bigoplus_{n \in \IN} (\phi^n)_* \]
 is well-defined. By assumption, we have $\id_{K\bA\cX^\Gamma(X)} = K\bA\cX^\Gamma(\phi)$. 
 Now apply Property~\eqref{it:Kprop.additivity} of algebraic $K$-theory to deduce that
 \[ \id_{K\bA\cX^\Gamma(X)} + K(\Sigma) \simeq \id_{K\bA\cX^\Gamma(X)} + K\bA\cX^\Gamma(\phi) \circ K(\Sigma) \simeq K(\id_{\bV_\bA^\Gamma(X)} \oplus \phi_* \circ \Sigma) \simeq K(\Sigma)\ ,\]
 so $\id_{K\bA\cX^\Gamma(X)} \simeq 0$.
\end{proof}

Recall the definition of the free union of a family of $\Gamma$-bornological coarse spaces which was given in \cref{efwifuhwfowefewfewfwef} and \cref{rgfou894ut984t3kfnrekjf} of the notion of strong additivity for an equivariant coarse homology theory.

\begin{prop}\label{prop:coarseK.additive}
 The equivariant coarse homology theory $K\bA\cX^\Gamma$ is strongly additive.
\end{prop}
\begin{proof}
 Let $(X_i)_{i \in I}$ be a family of $\Gamma$-bornological coarse spaces.
 The functors \[\Phi_j \colon \bV_\bA^\Gamma(\bigsqcup_{i \in I}^{\free} X_i) \to \bV_\bA^\Gamma(X_j)\]
 sending a $\bigsqcup_{i \in I}^{\free} X_i$-controlled $\bA$-object $(A,\rho)$ to $(A|_{X_j},\rho|_{X_j})$ for $j$ in $I$ assemble to a functor
 \[ \Phi \colon \bV_\bA^\Gamma(\bigsqcup_{i \in I}^{\free} X_i) \to \prod_{i \in I} \bV_\bA^\Gamma(X_i)\ .\]
 For $j$ in $I$, let $\iota_j \colon X_j \to \bigsqcup_{i \in I}^{\free} X_i$ denote the inclusion. 
 We claim that the functor
 \[ \Psi \colon \prod_{i \in I} \bV_\bA^\Gamma(X_i) \to \bV_\bA^\Gamma(\bigsqcup_{i \in I}^{\free} X_i) \]
 which sends a sequence $(A_i,\rho_i)_i$ to $\bigoplus_{i \in I} (\iota_i)_*(A_i,\rho_i)$ is well-defined (up to canonical isomorphism).
 We only have to check that the direct sum exists. This   follows from the fact that for every bounded subset $B$
  of $\bigsqcup_{i \in I}^{\free} X_i$ the subset $\{i\in I\:|\: B\cap X_{i}\not=\emptyset\}$ is finite, and 
 that $B \cap X_i$ is bounded for all $i$ in $I$.
 
 Clearly, $\Psi \circ \Phi$ is   isomorphic to the identity.
 The composition $\Psi \circ \Phi$ is also   isomorphic to the identity since $(A,\rho) \cong \bigoplus_{i \in I} (A|_{X_i},\rho|_{X_i})$ for all objects $(A,\rho)$.
 
 Using Properties~\eqref{it:Kprop.invariance} and \eqref{it:Kprop.products}, we conclude that
 \[ K(\bV_\bA^\Gamma(\bigsqcup_{i \in I}^{\free} X_i)) \xrightarrow{K(\Phi)} K(\prod_{i \in I} \bV_\bA^\Gamma(X_i)) \xrightarrow{\simeq} \prod_{i \in I} K(\bV_\bA^\Gamma(X_i)) \]
 is an equivalence. Note that the $j$-th component of this equivalence is given by the map $K(\Phi_j)$.
 Now apply \cref{rem:excision.projection} to see that $K(\Phi_j)$ agrees with the projection map coming from excision.
\end{proof}

\subsection{Calculations}
\label{sec092323}
\subsubsection{Examples of the form $(\Gamma/H)_{min,min}$}

Let $H$ be a subgroup of $\Gamma$.  Let $\bA$ be an additive category with trivial $\Gamma$-action.
	
	\begin{lem}\label{kljfioejoifjweiofjwefewfewfw}
		We have an equivalence
		\[K\bA\cX^{\Gamma}((\Gamma/H)_{min,min})\simeq  K(\Fun(BH,\bA)) \ .\]
	\end{lem}
	\begin{proof} 
	In view of the Property \ref{it:Kprop.invariance} of the $K$-theory functor it suffices to 
	construct an equivalence of additive categories
	\[\Phi\colon \bV_{\bA}^{\Gamma}((\Gamma/H)_{min,min})\to \Fun(BH,\bA)\ .\]
	This functor sends an object $(A,\rho)$ of $ \bV^{\Gamma }_{\bA}((\Gamma/H)_{min,min})$ to the functor sending $\gamma$ to $\rho(\gamma)(\{eH\})\colon A(\{eH\})\to A(\{eH\})$. 
		
	Furthermore, the functor $\Phi$ sends a morphism 
	\[f\colon (A,\rho)\to (A',\rho')\]
	in $ \bV^{\Gamma }_{\bA}((\Gamma/H)_{min,min})$ to the transformation $f(\{eH\})\colon  A(\{eH\})\to A'(\{eH\})$.  
	
	In order to define an inverse functor we choose a section $s\colon G/H\to G$ of the projection $G\to G/H$. Then 
	\[\Psi\colon  \Fun(BH,\bA)\to \bV_{\bA}^{\Gamma}((\Gamma/H)_{min,min})\]
	sends a functor $F\colon BH\to \bA$ to the following object $(A,\rho)$ of $ \bV_{\bA}^{\Gamma}((\Gamma/H)_{min,min})$: we choose $A(B)=\bigoplus_{b\in B}F(*)$ and $\rho$ is defined on an element $\gamma$ of $\Gamma$ such that $\rho(\gamma)(B)$   
	is the morphism  $\bigoplus_{b\in B}F(*)\to \bigoplus_{b\in \gamma^{-1}(B)}F(*)$ sending the summand with index $b=gH$ to the summand with index $\gamma^{-1} gH$ via $F(h)$, 
	where $h$ is the element of $H$ which is uniquely determined by the equation $\gamma^{-1} s(gH)=s(\gamma^{-1}g H)h$.  
	
	It is an easy exercise to construct the isomorphisms from the compositions $\Psi\circ \Phi$ and $\Phi\circ \Psi$ to the respective identity functors.
	\end{proof}
\subsubsection{Examples of the form $X_{min,max}\otimes \Gamma_{can,min}$}\label{ekfheqwufuiiu23zr23r2r}
We consider the group $\Gamma$ as a $\Gamma$-bornological coarse space $\Gamma_{can,min}$. In applications of coarse homotopy theory to proofs of the Farrell--Jones conjecture
the coarse algebraic $K$-{homology} $K\bA \cX^{\Gamma}$ twisted by  $\Yo^{s}(\Gamma_{can,min})$ plays an important role. Therefore it is relevant to calculate the spectra 
\[K\bA\cX^{\Gamma}_{\Yo^{s}(\Gamma_{can,min})}((\Gamma/H)_{min,max})\simeq K\bA\cX^{\Gamma}((\Gamma/H)_{min,max}\otimes  \Gamma_{can,min})\ .\] More generally, we will replace $\Gamma/H$ by any $\Gamma$-set $X$.

\begin{ddd}[{\cite[Def.~2.1]{bartels-reich}}]
	Let $\bA$ be an additive category with a $\Gamma$-action and let $X$ be a $\Gamma$-set. We define a new additive category denoted $\bA\ast_\Gamma X$ as follows. An object $A$ in $\bA\ast_\Gamma X$ is a family $A = (A_x)_{x\in X}$ of objects in $\bA$ where we require that $\{x\in X\mid A_x\neq 0\}$ is a finite set. A morphism	$\phi\colon A \to B$ is a collection of morphisms $\phi=(\phi_{x,g})_{(x,g)\in X\times \Gamma}$, where $\phi_{x,g}\colon A_x\to g(B_{g^{-1}x})$ is a morphism in $\bA$. We require that the set of pairs $(x,g)$ in $X\times \Gamma$ with $\phi_{x,g}\neq 0$ is finite. Addition of morphisms is defined componentwise. Composition of morphisms
	is defined as the convolution product.
\end{ddd}
\begin{rem}
	In \cite[Def.~2.1]{bartels-reich} additive categories with right $\Gamma$-action are used. For us it is more convenient to consider left $\Gamma$-actions. 
\end{rem}
Let $H$ be a subgroup of $\Gamma$.
\begin{ddd}
	We will denote $\bA\ast_\Gamma(\Gamma/H)$ by $\bA[H]$.
\end{ddd}
If $\bA$ is the category of finitely generated, free $R$-modules for some ring $R$, then $\bA[H]$ is equivalent to the category of finitely generated, free $R[H]$-modules.

The following calculation closely follows Bartels--Farrell--Jones--Reich \cite[Sec.~6.1 and Proof of Prop.~6.2]{MR2030590}.

\begin{prop}\label{fwefjewfeoewf3453453455} For every $\Gamma$-set $X$ we have an equivalence 
	\[\bV_{\bA}^{\Gamma}(X_{min,max}\otimes \Gamma_{can,min})\simeq \bA\ast_\Gamma X\ .\]
\end{prop}
\begin{proof} 
The desired equivalence is given by an exact functor
\[\Phi\colon \bV_{\bA}^{\Gamma }(X_{min,max}\otimes \Gamma_{can,min})\to \bA\ast_\Gamma X\ .\]

We define $\Phi$ as follows:
\begin{enumerate}
	\item For an object $(A,\rho)$ in $\bV_{\bA}^{\Gamma }(X_{min,max}\otimes \Gamma_{can,min})$, we define $\Phi(A,\rho)_x$ as $A(\{(x,1)\})$.
	\item For a morphism $f\colon (A,\rho)\to (A',\rho')$ we define $\Phi(f)_{x,g}$ as the composition
	\[A(\{x,1\})\xrightarrow{f}A'(\{x\}\times F)\xrightarrow{p_g}A'(\{x,g\})\xrightarrow{\rho'(g)}gA'(\{g^{-1}x,1\})\ ,\]
	where $F$ is a finite subset of $\Gamma$ containing $g$ and $p_g$ is the projection arising from the identification $\bigoplus_{f\in F}A(\{x,f\})\xrightarrow{\cong}A(\{x\}\times F)$.
\end{enumerate}
Note that $\Phi(f)_{x,g}$ is independent of the choice of $F$.

We will first show that $\Phi$ is fully faithful. A morphism $f\colon (A,\rho)\to (A',\rho')$ is determined by its values on $A(\{x,1\})$ by equivariance. Since $X_{min, max}$ has the minimal coarse structure, the family $(\Phi(f)_{x,g})_{(x,g)\in X\times \Gamma}$
determines $f$. Hence $\Phi$ is faithful.

{Since the $\Gamma$-action on $X_{min,max}\otimes \Gamma_{can,min}$ is free, for every finite subset $F$ of $\Gamma$ and every family of morphisms $A(\{x,1\})\to A'(\{x\}\times F)$ indexed by points $x$ in $X$ there exists a unique equivariant extension to a morphism $f\colon (A,\rho)\to (A,\rho')$. Let $\phi\colon \Phi(A,\rho)\to \Phi(A',\rho')$ be any morphism in $A\ast_\Gamma X$. Let $F:=\{g\in \Gamma\mid \exists x\in X : \phi_{g,x}\neq 0\}$, then $F$ is a finite subset of $\Gamma$. The family of morphisms
\[\left((A(\{x,1\})\xrightarrow{\bigoplus\limits_{g\in F}\rho'_{g^{-1}}\circ \phi_{x,g}}\bigoplus_{g\in F}A'(\{x,g\})\xrightarrow{\cong}A'(\{x\}\times F)\right)_{x\in X}\]
extends to a morphism $f\colon (A,\rho)\to (A',\rho')$ with $\Phi(f)=\phi$. This shows that $\Phi$ is fully faithful.}

We now show that $\Phi$ is essentially surjective. Every finitely supported family $(A_x)_{x\in X}$ of objects of $A$ extends essentially uniquely to an equivariant object $(A,\rho)$ with $A(\{x,1\})=A_x$ for all $x$ in $X$. This uses the choice of finite sums of the objects $A_x$.  \end{proof}

Let $X$ be a $\Gamma$-set.
\begin{kor}
	We have an equivalence
	\[K\bA\cX^\Gamma(X_{min,max}\otimes\Gamma_{can,min})\simeq K(\bA\ast_\Gamma X)\ .\]
\end{kor}
 
\subsubsection{Examples of the form $(\Gamma/H)_{min,min}\otimes\Gamma_{?,max}$.}
Let $H$ be a subgroup of $\Gamma$ and let $\bA$ be an additive category with $\Gamma$-action.
\begin{lem}\label{rgoergug9gu43g9}
	If $H$ is finite, then
	\[K\bA\cX^{\Gamma}((\Gamma/H)_{min,min}\otimes \Gamma_{max,max})\simeq  K\bA\cX^{\Gamma}((\Gamma/H)_{min,min}\otimes \Gamma_{can,max})\simeq  K(\bA[H])\ .\]
	Otherwise
	\[K\bA\cX^{\Gamma}((\Gamma/H)_{min,min}\otimes \Gamma_{max,max})\simeq  K\bA\cX^{\Gamma}((\Gamma/H)_{min,min}\otimes \Gamma_{can,max})\simeq 0\ .\]
\end{lem}

\begin{proof}
We argue similarly as in the proof of  \cref{fwefjewfeoewf3453453455} for $X=(\Gamma/H)_{min,min}$.
If $H$ is finite, then the set $F$ appearing in the proof is still finite,  but for a different reason.

If $H$ is infinite, then there are no non-trivial $\Gamma$-invariant  $(\Gamma/H)_{min,min}\otimes\Gamma_{?,max}  $  controlled modules (this does not depend on the coarse structures).
\end{proof}

\subsection{Change of groups}

Let $H$ be a subgroup of $\Gamma$.
\begin{theorem}
There is an equivalence of $H$-equivariant coarse homology theories
\[\indd_H^\Gamma\colon K\bA\cX^H \xrightarrow{\simeq} K\bA\cX^\Gamma\circ \Ind_H^\Gamma\ .\]
\end{theorem}
In the proof we will construct a equivalence in the other direction. We state the theorem in this form since this is the more common direction.
\begin{proof}
	Let $(X,\cB,\cC)$ be a bornological coarse space. Recall from the \cref{rem:jtjbzuzgkvrgkieilfu} that the bornological coarse space $\Ind_H^\Gamma X$ is given by the set $\Gamma\times_H X$ with bornology generated by the subsets $\{g\}\times B$ for $B$ in $\cB$ and coarse structure generated by the entourages $\diag_\Gamma\times U$ for $U$ in $\cC$.
	
	Note that $H\times_H X\simeq X$ is an $H$-invariant coarse component of $\Ind_H^\Gamma X$. Hence, restricting an object $(A,\rho)$ of $\bV_\bA^\Gamma(\Ind_H^\Gamma X)$ to $(A|_X,\rho_X)$ yields a functor
	\[\bV_\bA^\Gamma(\Ind_H^\Gamma X)\to \bV_\bA^H(X)\ .\]
	Similarly to the proof of \cref{kljfioejoifjweiofjwefewfewfw}, one checks that this functor is an equivalence.
\end{proof}

Let $H$ be a subgroup of $\Gamma$. Sending a $\Gamma$-equivariant $X$-controlled object $(A,\rho)$ to the object $(A,\{\rho(h)\}_{h\in H})$ yields a natural transformation
\begin{equation}
\res_H^\Gamma\colon K\bA\cX^\Gamma\to K\bA\cX^H\circ \Res_H^\Gamma\ .
\end{equation}

\subsection{Variations on the definition}

Let again $\bA$ be an additive category with a $\Gamma$-action and $(X,\cB,\cC)$ be a $\Gamma$-bornological coarse space.
Intuitively, an equivariant $X$-controlled $\bA$-object $(A,\rho)$ is some (infinite) sum of objects in $\bA$ parametrized by points in $X$ together with an action of $\Gamma$.
One may want to keep track of the ``global'' object associated to an $X$-controlled object explicitly.
The purpose of this section is to give an alternative definition of $\bV_\bA^\Gamma(X)$ which accomplishes precisely this,
and discuss a variation of this definition which leads to an example of a non-continuous coarse homology theory.

Since an equivariant $X$-controlled $\bA$-object usually involves an infinite number of objects in $\bA$, we need to enlarge our coefficient category appropriately.
Therefore, let $\bA \to \widehat{\bA}$ be a fully faithful and $\Gamma$-equivariant embedding of $\bA$ into an additive category $\widehat{\bA}$ with $\Gamma$-action which admits sufficiently large direct sums.
The sum completion of $\bA$ is a canonical choice for $\widehat{\bA}$.

For an object $A$ in $\widehat{\bA}$, let $\Pi(A)$ denote the set of idempotents on $A$ whose image splits off as a direct sum and is isomorphic to an object in $\bA$.

\begin{ddd}\label{ddd:global.controlled.object}
 A \emph{global equivariant $X$-controlled $\bA$-object} is a triple $(A,\phi,\rho)$ consisting of
 \begin{enumerate}
  \item an object $A$ in $\widehat{\bA}$;
  \item a function $\phi \colon \cB \to \Pi(A)$;
  \item a morphism $\rho(\gamma) \colon A \to \gamma A$ for every element $\gamma$ in $\Gamma$;
 \end{enumerate}
 such that the following conditions are satisfied:
 \begin{enumerate}
  \item\label{it:global.controlled.object1} The function $\phi$ satisfies the following relations:
   \begin{enumerate}
    \item $\phi(\emptyset) = 0$;
    \item $\phi(B_1 \cup B_2) = \phi(B_1) + \phi(B_2) - \phi(B_1 \cap B_2)$;
    \item $\phi(B_1 \cap B_2) = \phi(B_1) \circ \phi(B_2)$.
   \end{enumerate}
  \item\label{it:global.controlled.object2} For every bounded subset $B$ of $X$, there exists some finite subset $F$ of $B$ such that $\phi(B) = \phi(F)$.
  \item For all pairs of elements $\gamma$, $\gamma'$ in $\Gamma$, we have $\rho(\gamma'\gamma) = \gamma\rho(\gamma') \circ \rho(\gamma)$.
  \item For all elements $\gamma$ of $\Gamma$ and bounded subsets $B$ of $X$, we have the equality
  \[ \phi(\gamma^{-1} B) = \rho(\gamma)^{-1} \circ \gamma \phi(B) \circ \rho(\gamma)\ .\qedhere\]
 \end{enumerate}
\end{ddd}

In contrast to the data specifying an equivariant $X$-controlled object, this definition does not fix a chosen image for each of the idempotents $\phi(B)$.

\begin{ddd}
 A morphism $f \colon (A,\phi,\rho) \to (A',\phi',\rho')$ of global equivariant $X$-controlled $\bA$-objects is a morphism $f \colon A \to A'$ in $\widehat{\bA}$ satisfying the following conditions:
 \begin{enumerate}
  \item $f$ is equivariant in the sense that $(\gamma f) \circ \rho(\gamma) = \rho'(\gamma) \circ f$;
  \item $f$ is controlled in the sense that the set
  \[ \bigcap \{ U \subseteq X \times X \mid \forall\ B,B' \in \cB \colon U[B] \cap B' = \emptyset \Rightarrow \phi'(B')f\phi(B) = 0 \} \]
  is an entourage of $X$.\qedhere
 \end{enumerate}
\end{ddd}

Morphisms of global equivariant $X$-controlled $\bA$-objects can be composed.
We denote the resulting category by $\bV_{\bA \subseteq \widehat{\bA}}^\Gamma(X)$.
Similar to the discussion in \cref{ffopjiofr23r2u23oif}, one shows that $\bV_{\bA \subseteq \widehat{\bA}}^\Gamma(X)$ is additive. If $f\colon X\to X^{\prime}$ is a morphism of $\Gamma$-bornological coarse spaces, then we define a functor
\[f_{*}\colon\bV_{\bA \subseteq \widehat{\bA}}^\Gamma(X)\to \bV_{\bA \subseteq \widehat{\bA}}^\Gamma(X^{\prime})\]
which sends $(A,\phi,\rho)$ to  $(A,f_{*}\phi,\rho)$, were $f_{*}\phi(B^{\prime}):=\phi(f^{-1}(B^{\prime}))$ for all $B^{\prime}$ in $\cB^{\prime}$. Furthermore, the construction is functorial
 with respect to commutative squares of additive functors
\[\xymatrix{
 \bA\ar[r]\ar[d]_{\Phi} & \widehat{\bA}\ar[d]^{\widehat{\Phi}} \\
 \bA'\ar[r] & \widehat{\bA}'
}\]
in which the horizontal arrows are fully faithful and $\Gamma$-equivariant embeddings.

We have already indicated how objects in $\bV_\bA^\Gamma(X)$ correspond to objects in $\bV_{\bA \subseteq \widehat{\bA}}^\Gamma(X)$ and vice versa,
namely by summing up all values of an equivariant $X$-controlled object and choosing images of idempotents, respectively.
In fact, it is not difficult to show the following.

\begin{prop} 
There is a zig-zag of equivalence between the functors $\bV_\bA^\Gamma$ and $\bV_{\bA \subseteq \widehat{\bA}}^\Gamma$ from $\Gamma\BC$ to $\Add$.
\end{prop}

\begin{ex}\label{gfoihwjgoij32rtgergergerge}
We now provide a modification of the definition of $\bV_{\bA \subseteq \widehat{\bA}}^\Gamma$
which leads to a non-continuous coarse homology theory. 
 
 Let $(A,\phi,\rho)$ be a global equivariant $X$-controlled $\bA$-object.
 Condition \eqref{it:global.controlled.object2} implies, in the presence of Condition \eqref{it:global.controlled.object1},
 that the object $(A,\phi)$ is essentially determined by its restriction to the poset of finite subsets of $X$, and that the support of $(A,\phi)$, i.e.\ the set $\{ x \in X \mid \phi(\{x\} \neq 0 \}$,
 is a locally finite subset of $X$.
 
 Dropping Condition \eqref{it:global.controlled.object2}, we obtain another additive category $\bV_{\bA,\Psi}^\Gamma(X)$ which is also functorial in $\Gamma$-bornological coarse spaces.
 Taking non-connective algebraic $K$-theory of this category also gives rise to an equivariant coarse homology theory $K\bA\cX^\Gamma_\Psi$ since the proofs of \cref{lem:coarseK.ucontinuity}, \cref{lem:coarseK.closemaps}, \cref{lem:coarseK.flasqueness} and \cref{prop:coarseK.excision} go through without change.
 However, the proof of continuity (\cref{prop:coarseK.continuous}) does not apply to $K\bA\cX^\Gamma_\Psi$ since the condition ensuring that the support of an object is locally finite has been omitted.
 
 In fact, the following example shows that the coarse homology theory $K\bA\cX_\Psi$ is not continuous:
 Suppose that there exists an object $A$ in $\bA$ whose class in $K_0$ is non-trivial.
 Consider the bornological coarse space $\IN_{min,max}$.
 Choose an ultrafilter $\cF$ on $\IN$ and define the function $\phi \colon \cP(\IN) \to \Pi(A)$ by
 \[ \phi(B) := \begin{cases} \id_A & B \in \cF, \\ 0 & B \notin \cF. \end{cases} \]
 Then $(A,\phi)$ is a global $\IN_{min,max}$-controlled $\bA$-object since $\cF$ is an ultrafilter.
 Moreover, the morphism of bornological coarse spaces $\IN_{min,max} \to *$ induces a homomorphism
 \[ \pi_0 K\bA\cX_\Psi(\IN_{min,max}) \to \pi_0 K\bA\cX_\Psi(*) \cong K_0(\bA) \]
 which maps the class $[(A,\phi)]$ to $[A] \neq 0$.
 
 On the other hand, locally finite subsets of $\IN_{min,max}$ are precisely the finite subsets of $\IN$.
 Since each finite subset $F$ is a union of coarse components of $\IN_{min,max}$, the pair of subsets $(F,\IN\setminus F)$ is a complementary pair.
 From excision, we obtain a direct sum decomposition
 \[ K\bA\cX_\Psi(\IN_{min,max}) \simeq K\bA\cX_\Psi(F_{min,max}) \oplus K\bA\cX_\Psi((\IN\setminus F)_{min,max})\ ,\]
 in which the projection $K\bA\cX_\Psi(\IN_{min,max}) \to K\bA\cX_\Psi(F_{min,max})$ is induced by the functor restricting $\IN_{min,max}$-controlled objects to $F$ (cf.~\cref{rem:excision.projection}).
 Since the restriction of $(A,\phi)$ to any finite subset of $\IN$ is the zero object,
 we conclude that the class $[(A,\phi)]$ does not lie in the image of the comparison map
 \[ \colim_{F \subseteq \IN\ \text{finite}} \pi_0 K\bA\cX_\Psi(F_{min,max}) \to \pi_0 K\bA\cX_\Psi(\IN_{min,max})\ .\]
 In particular, $K\bA\cX_\Psi$ is not continuous.
\end{ex}

\part{Cones and assembly maps}

\section{Cones}\label{rgergiujoi345346456456}

{In this section we will introduce the cone and the `cone at infinity' functors, and discuss their properties. The cone at infinity is later used in \cref{secjks9990rfwe} to define the universal assembly map, and in \cref{ifjwoifwefewfewff} to transform coarse homology theories into topological homology theories. In conjunction with the Rips complex construction the cone at infinity will be used in \cref{secjkfwe0002323} to construct the universal coarse assembly map.}

{The first three Sections~\ref{secjsdfgnk23w}--\ref{secjn2309sdmnb} are technical preparations. The cone functor is then  defined and discussed in \cref{seccone} and \cref{secjknds2sd23}.}

\subsection{\texorpdfstring{$\Gamma$}{Gamma}-uniform bornological coarse spaces}
\label{secjsdfgnk23w}
 
In this section we introduce the category $\Gamma\UBC$ of $\Gamma$-uniform {bornological} coarse spaces. Its objects are $\Gamma$-bornological {coarse} spaces with an additional $\Gamma$-uniform structure. {The additional datum of a $\Gamma$-uniform structure is needed in order to define hybrid structures, see \cref{sechybrid}.}

We start with recalling some basics on uniform spaces. Let $X$ be a set and $\cT$ be a subset of $\cP(X\times X)$, the power set of $X\times X$.
\begin{ddd}
The set $\cT$ is a {uniform structure} if it is {non-empty}, closed under composition, {inversion}, supersets, finite intersection, {every element of $\cT$ contains the diagonal of $X$,} and if for every $U$ in $\cT$ there exists  $V$ in $\cT$ such that $V\circ V\subseteq U$.
\end{ddd}

\begin{rem}
 {Note that any subset $S$ of $\cP(X \times X)$ with the property that for every $U$ in $S$ there exists $V$ in $X$ such that $V \circ V \subseteq U$ generates a uniform structure on $X$ by taking the closure of $S$ under composition, inversion, supersets and finite intersection.}
\end{rem}

 The elements of $\cT$ are called \emph{uniform entourages}. We will consider $\cT$ as a filtered partially ordered set whose order relation is the opposite of the inclusion relation.
\begin{ddd}
A \emph{uniform space} is a pair $(X,\cT)$ of a set $X$ together with a uniform structure $\cT$.
\end{ddd}

Let $(X,\cT)$ and $(X^{\prime},\cT^{\prime})$ be uniform  spaces and $f\colon X\to X^{\prime}$ be a map of sets.
\begin{ddd}
The map $f$ is called a \emph{uniform map} if for every uniform entourage $U^{\prime}$  of $X^{\prime}$  we have   $(f\times f)^{-1}(U^{\prime})\in \cT$.
\end{ddd}

If a group $\Gamma$ acts on a uniform space $(X,\cT)$, then it acts on the set of uniform entourages~$\cT$. We let $\cT^{\Gamma}$  denote  the subset of  $\cT$ of $\Gamma$-invariant uniform entourages.

\begin{ddd}
A  \emph{$\Gamma$-uniform space} is a uniform space $(X,\cT)$ with an action of $\Gamma$ by automorphisms such that $\cT^{\Gamma}$ is cofinal in $\cT$.
\end{ddd}

A uniform structure $\cT$ on a $\Gamma$-set  $X$  such that  $X$ is a  $\Gamma$-uniform  
space will be called a \emph{$\Gamma$-uniform structure}.

\begin{ex}
The uniform structure $\cT_{d}$ of a metric space $(X,d)$ is generated by the  uniform entourages
\begin{equation}\label{fhbeufhiu24r}
U_{r}:=\{(x,y)\in X\times X\:|\: d(x,y) \le r\}
\end{equation}
for all $r$ in $(0,\infty)$.
 
If $\Gamma$ acts isometrically on a metric space $(X,d)$, then the {associated uniform space} $X_{u,d}:=(X,\cT_{d})$ 
is a $\Gamma$-uniform space. 

If the metric is implicitly clear, then we will also write $X_{u}$ instead of $X_{u,d}$.

{A uniformly continuous map between metric spaces induces a uniform map between the
{associated} uniform spaces.}

The standard metric turns $\R$ into a uniform space $\R_{u}$.
  The action of $\Z$ on $\R$ by dilatations $(n,x)\mapsto 2^{n}x$ is an action on $\R_{u}$ by automorphisms of uniform spaces, but $\R_{u}$ is not a  $\Z$-uniform space.
\end{ex}

Let $\Gamma\bU$ be the category of $\Gamma$-uniform spaces and  uniform equivariant maps.

\begin{ex}\label{rliheiogherigergregreg}
Let $K$ be a simplicial complex with a simplicial $\Gamma$-action. On $K$ we consider the path quasi-metric induced by the spherical metric on the simplices. This quasi-metric is preserved by $\Gamma$ and  the associated uniform structure  is a $\Gamma$-uniform structure. Therefore a simplicial complex with a simplicial $\Gamma$-action gives rise to a $\Gamma$-uniform space $K_{u}$.

If $K\to K^{\prime}$ is an equivariant simplicial map, then it is a  uniform map  $K_{u}\to K^{\prime}_{u}$  between the $\Gamma$-uniform spaces.
\end{ex}

We now consider the combination of uniform and bornological coarse structures. 
We consider a $\Gamma$-set $X$ with a $\Gamma$-coarse structure $\cC$ and a $\Gamma$-uniform structure $\cT$. 
\begin{ddd} We say that $\cC$ and $\cT$ are \emph{compatible} if $\cC^{\Gamma}\cap \cT^{\Gamma}\not=\emptyset$.
\end{ddd}
In words, the coarse and the uniform structures are compatible if there exists an invariant entourage which is both a coarse entourage and a uniform entourage.

\begin{ddd}
 We define the category $\Gamma\UBC$ of $\Gamma$-uniform bornological coarse spaces as follows:
\begin{enumerate}
\item The objects of $\Gamma\UBC$ are $\Gamma$-bornological coarse spaces with an additional compatible $\Gamma$-uniform structure. 
\item The   morphisms of $\Gamma\UBC$ are morphisms of $\Gamma$-bornological coarse spaces which are in addition uniform.\qedhere
\end{enumerate}
\end{ddd}

\begin{ex}
Let $(X,d)$ be a quasi-metric space with an action of $\Gamma$  by isometries and a $\Gamma$-invariant bornology $\cB$. Assume that the metric and the bornology are compatible in the sense  that for every $r$ in $(0,\infty)$ and $B$ in $\cB$ we have $U_{r}[B]\in \cB$, where $U_{r}$ is as in \eqref{fhbeufhiu24r}. Then we get a $\Gamma$-uniform bornological coarse space $X_{du}$ with the following structures:
\begin{enumerate}
\item The coarse structure is generated by the {coarse} entourages $U_{r}$ for all $r$ in $(0,\infty)$.
\item The uniform structure is generated by the {uniform entourages} $U_{r}$ for all $r$ in $(0,\infty)$.
\item The bornology is $\cB$.
\end{enumerate}
If $(X^{\prime},d^{\prime})$ is a second quasi-metric space with isometric $\Gamma$-action and $f\colon X\to X^{\prime}$ is a proper (this refers to the bornologies), $\Gamma$-equivariant contraction, then
$f\colon X_{du}\to X^{\prime}_{du}$ is a morphism of $\Gamma$-uniform bornological coarse spaces.
\end{ex}

\begin{ex}\label{ifjwoifwoefwefewfewfw}
Let $K$ be a $\Gamma$-simplicial complex equipped with the spherical quasi-metric.
Then we can equip $K$ with the bornology of metrically bounded subsets and obtain a $\Gamma$-uniform bornological coarse space $K_{du} $.
 Alternatively we can equip it with the maximal bornology and get the $\Gamma$-uniform bornological coarse space $K_{du,max}$.

A morphism of $\Gamma$-simplicial complexes $f \colon K\to K^{\prime}$ always induces the two morphisms $f\colon K_{du,max}\to K^{\prime}_{du,max}$ and $K_{du,max}\to K^{\prime}_{du}$
of $\Gamma$-uniform bornological coarse spaces. If $f$ is proper (in the sense that preimages of simplices are finite complexes), then it also induces a morphism $f \colon K_{du}\to K_{du}^{\prime}$.
\end{ex}

\begin{ex}\label{fhwiofweewff}
For a $\Gamma$-uniform bornological coarse space $X$ let $F_{\cT}(X)$ denote the underlying $\Gamma$-bornological coarse space obtained by forgetting the datum of the uniform structure.
Let $X$ and $Y$ be $\Gamma$-uniform bornological coarse spaces. Then we define the $\Gamma$-uniform bornological coarse space $X\otimes Y$  such that  $\cF_{\cT}(X\otimes Y)=F_{\cT}(X)\otimes F_{\cT}(Y)$  and the $\Gamma$-uniform structure on $X\otimes Y$ is generated by the  products $U\times V$ for all pairs of uniform entourages $U$ of $X$ and $V$ of $Y$. \end{ex}

\subsection{Hybrid structures}
\label{sechybrid}

{In this section we will define   hybrid coarse structures, which will feature in the definition of cones in \cref{seccone}.}

We consider a $\Gamma$-uniform bornological coarse space~$X$ with coarse structure $\cC$, bornology~$\cB$, and uniform structure $\cT$. Let furthermore a $\Gamma$-invariant big family $\cY=(Y_{i})_{i\in I}$ be given. 
\begin{ddd}
The pair $(X,\cY)$ is called \emph{hybrid data}.
\end{ddd}
In this situation we can define the hybrid coarse structure $\cC_{h}$ {as follows.}

Note that $\cP(X\times X)^{\Gamma}$ is a filtered poset with the opposite of the inclusion relation. We consider a 
 function $\phi\colon I\to \cP(X\times X)^{\Gamma}$, i.e., an order-preserving map. 
 \begin{ddd} \label{efokwpfwefwefwefewffew}The function $\phi$
  is called $\cT^{\Gamma}$-admissible, if for every $U$ in $\cT^{\Gamma}$ there exists $i$ in $I$ such that $\phi(i)\subseteq U$.\end{ddd}

Given a $\cT^{\Gamma}$-admissible function $\phi\colon I\to \cP(X\times X)^{\Gamma}$ we define the entourage
\[U_{\phi}:=\left\{(x,y)\in X\times X\:|\: {\big(\forall i\in I\:|\: (x,y) \in Y_{i}\times Y_{i} \text{ or } (x,y) \in \phi(i)\big)}\right\}\ .\]
Note that $U_{\phi}$ is $\Gamma$-invariant.
\begin{ddd}\label{efioofwefwefwefewf}
The \emph{hybrid coarse structure} $\cC_{h}$  is the coarse structure generated by the entourages $U\cap U_{\phi}$ for all $U$ in ${\cC^\Gamma}$ and $\cT^{\Gamma}$-admissible functions $\phi\colon I\to \cP(X\times X)^{\Gamma}$.
\end{ddd}

\begin{ddd}
The \emph{hybrid space} $X_{h}$  is defined to be the $\Gamma$-bornological coarse space with underlying set $X$,  the hybrid coarse structure $\cC_{h}$ and the bornological structure $\cB$.
\end{ddd}

A morphism of hybrid data $f\colon (X,\cY)\to (X^{\prime},\cY^{\prime})$ is a morphism of $\Gamma$-uniform bornological coarse spaces which is compatible with the big families $\cY=(Y_{i})_{i\in I}$ and $\cY^{\prime}=(Y^{\prime}_{i^{\prime}})_{i^{\prime}\in I^{\prime}}$ in the sense that for every $i$ in $I$ there exists $i^{\prime}$ in $I^{\prime}$ such that $f(Y_{i})\subseteq Y^{\prime}_{i^{\prime}}$.

\begin{lem}
If $f$ is a morphism of hybrid data, then the underlying map of sets is a morphism $f\colon X_{h}\to X^{\prime}_{h}$ of $\Gamma$-bornological coarse spaces.
\end{lem}

\begin{proof}
\cite[Lem.~5.15]{buen}.
\end{proof}

\begin{rem}
One could set up a category of hybrid data and understand the construction of the hybrid structure as a functor from hybrid data to $\Gamma$-bornological coarse spaces. 
\end{rem}

\subsection{Decomposition Theorem and Homotopy Theorem}
\label{secjn2309sdmnb}

{In this section we discuss the Decomposition Theorem and the Homotopy Theorem for hybrid coarse structures. These two theorems constitute important technical results which are needed to prove crucial properties of the cone functor in \cref{secjknds2sd23}.}

Let $A,B$ be $\Gamma$-invariant subsets of a $\Gamma$-uniform space $Y$ with uniform structure $\cT$. For an entourage  $U$  of $Y$  we set \[\cT^{\Gamma}_{\subseteq U}:=\{V\in \cT^{\Gamma}\:|\: V\subseteq U\}\ .\]
\begin{ddd}\label{fewoifweifoew23242342341}
The pair $(A,B)$ is an \emph{equivariant uniform decomposition} if
\begin{enumerate}
\item  $Y=A\cup B$, and
\item there is  an invariant  uniform entourage $U$ of $Y$   and a function
\[s\colon \cP(Y\times Y)_{\subseteq U}^{\Gamma} \to \cP(Y\times Y)^{\Gamma}\]
such that for every $W$ in $\cP(Y\times Y)_{\subseteq U}^{\Gamma}$ we have  the inclusion
\[W[A]\cap W[B]\subseteq s(W)[A\cap B]\]
and the restriction
$s_{|\cT^{\Gamma}_{\subseteq U}}$ is $\cT^{\Gamma}$-admissible.\qedhere
\end{enumerate}
\end{ddd}
According to \cref{efokwpfwefwefwefewffew},  the function $s_{|\cT^{\Gamma}_{\subseteq U}}$ is  $\cT^{\Gamma}$-admissible  if for every entourage $V$ in $\cT^{\Gamma}$ there is an entourage $W$ in $\cT^{\Gamma}_{\subseteq U}$ such that $ {s}(W)\subseteq V$.

\begin{ex}
Assume that  $K$ is a $\Gamma$-simplicial complex, and $A$ and $B$ are $\Gamma$-invariant subcomplexes such that $K=A\cup B$. Then $(A,B)$ is an equivariant uniform decomposition
of $K_{u}$. This follows from \cite[Ex.~5.19]{buen}.
\end{ex}

Let $Y$ be a $\Gamma$-uniform bornological coarse space with   an invariant  big family $\cY=(Y_{i})_{i\in I} $. We further assume that   $(A,B)$ is an   equivariant  uniform decomposition of $Y$.  We let $Y_{h}$ denote the associated bornological coarse space with the hybrid structure. 

We write $A_{h}$ for the $\Gamma$-bornological coarse space obtained from  the $\Gamma$-uniform bornological coarse structure on $A$ induced from $Y$ by first restricting the hybrid data to $A$ and then forming the hybrid structure. By 
$A_{Y_{h}}$ we denote the $\Gamma$-bornological coarse space 
obtained by restricting the structures of $Y_{h}$ to the subset $A$.
It was shown in \cite[Lem.~5.17]{buen} that $A_h=A_{Y_h}$.

\begin{ddd}
	A $\Gamma$-uniform space $(Y,\cT)$ is called Hausdorff, if $\bigcap_{U\in\cT}U=\diag_Y$.
\end{ddd}

\begin{theorem}(Decomposition Theorem)\label{sdf8923njn}
If $I=\nat$ and $Y$ is Hausdorff, then the following square in $\Gamma\Sp\cX$ is cocartesian:
\begin{equation}\label{dqwdqd3kjnr23kjrn23r32r32}
\xymatrix{\Yo^{s}((A\cap B)_{h},  A\cap B\cap \cY)\ar[r]\ar[d]&\Yo^{s}(A_{h},A\cap\cY )\ar[d]\\\Yo^{s}(B_{h},B\cap \cY )\ar[r]&\Yo^{s}(Y_{h},\cY)}
\end{equation}
\end{theorem}

\begin{proof}
The proof of  \cite[Thm.~5.20]{buen}
goes through word-for-word. One just works with invariant entourages or invariant uniform neighbourhoods everywhere.
\end{proof}

We consider a $\Gamma$-uniform {bornological coarse} space $Y$ with an invariant  big family $\cY=(Y_{n})_{ {n\in \IN}} $. We consider the unit interval $[0,1]$ with the trivial $\Gamma$-action
as a $\Gamma$-uniform bornological coarse space $[0,1]_{du}$ with the structures induced from the metric. 
On the tensor product $[0,1]_{du}\otimes Y$ (\cref{fhwiofweewff}) we consider the big family   {$([0,1]\times Y_{n})_{n\in\IN}$}. Let $\cB$ denote the bornology of $Y$.
  
\begin{theorem}[Homotopy Theorem]\label{fiewfewifieowifowioioio23r}
 {Assume that for every $B$ in $\cB$ there exists $n$ in $\IN$ such that $B\subseteq Y_{n}$.}
Then the projection induces an equivalence 
 \[\Yo^{s}(([0,1]_{du}\otimes Y)_{h})\to \Yo^{s}(Y_{h})\ .\]
\end{theorem}

\begin{proof}
The proof of  \cite[Thm.~5.25]{buen} goes through in the present equivariant case.
\end{proof}

\subsection{The cone functor}
\label{seccone}

In this section we define the cone functor
\begin{equation}\label{ecew2r4r43r}
\cO\colon \Gamma\UBC\to \Gamma\BC
\end{equation}
{and prove that decompositions of the spaces which are simultaneously uniformly and coarsely excisive  lead to corresponding coarsely excisive decompositions of the cones.}
  
We consider the metric space $[0,\infty)$ with the metric induced from the inclusion into $\R$ and the trivial $\Gamma$-action.  We get a 
 $\Gamma$-uniform bornological coarse space  $[0,\infty)_{du}$. For a $\Gamma$-uniform bornological coarse space $Y$  we  form the $\Gamma$-uniform bornological coarse space \[  [0,\infty)_{du}\otimes Y \] (\cref{fhwiofweewff}).
 This $\Gamma$-uniform bornological coarse   space has a canonical
big family given by
\begin{equation}\label{fewklfjewfejewfewf}
\cY(Y):=([0,n]\times Y)_{n\in \nat}\ .
\end{equation}

A morphism $f\colon Y\to Y^{\prime}$ of $\Gamma$-uniform bornological coarse spaces induces a morphism of hybrid data
\[(  [0,\infty)_{du}\otimes Y,\cY(Y))\to  ( [0,\infty)_{du}\otimes Y^{\prime},\cY(Y^{\prime}))\]
 given by the map $\id_{[0,\infty)}\times f$ on the underlying sets.

\begin{ddd}\label{wefjweoifwe345345}
The \emph{cone functor} \eqref{ecew2r4r43r}  is defined such that it sends the $\Gamma$-uniform bornological coarse space $Y$ to the $\Gamma$-bornological coarse space
\[\cO(Y):= ([0,\infty)_{du}\otimes Y)_{h} \] and a morphism $f\colon Y\to Y^{\prime}$ of $\Gamma$-uniform bornological coarse spaces to the morphism \[\cO(f)\colon \cO(Y)\to \cO(Y^{\prime})\] of $\Gamma$-bornological coarse spaces given by the map $\id_{[0,\infty)}\times f$ of the underlying sets. \end{ddd}
 
\begin{ex}\label{ropjreopg0iu09t34t34t34t3}
If the  $\Gamma$-uniform bornological coarse space $Y$  is discrete as a uniform and as a coarse space, then $\cO(Y)$ is flasque: Flasqueness of $\cO(Y)$ can be implemented by the map $(t,y) \mapsto (t + \frac{1}{t+1}, y)$.
 
  If the uniform structure of $Y$  is   discrete, but the coarse structure of $Y$ is strictly larger than the discrete one, then the above morphism $f$ does not implement flasqueness of $\cO(Y)$, because Condition \ref{fjweoifjewoijeoiejfiwjeiofjewfewfewf} in \cref{hckjhckhwhuiechiu1} is  violated.
 \end{ex} 

Let $X$ and $Y$ be $\Gamma$-uniform bornological coarse spaces. Recall that $F_{\cT}$ is the functor from $\Gamma$-uniform bornological coarse spaces to $\Gamma$-bornological coarse spaces which forgets the uniform structure.
\begin{lem}\label{reopgjreogjoregregregeg}
If $Y$ is discrete as a coarse space, then
\[\cO(X)\otimes F_{\cT}(Y) \cong \cO(X\otimes Y)\ .\]
\end{lem}

\begin{proof}
Immediate from the definitions.
\end{proof}
 
Let  $Y$ be a $\Gamma$-uniform bornological coarse space with coarse and uniform structures $\cC$, $\cT$, and let $A,B$ be $\Gamma$-invariant subsets of $Y$.
\begin{lem}\label{ifeioewfueoiwfwfewfwf}
If $(A,B)$ is an {equivariant} uniform {(\cref{fewoifweifoew23242342341})} and coarsely excisive decomposition of $Y$, then 
\[( [0,\infty)\times A  ,    [0,\infty)\times B )\] is a coarsely excisive  pair on $\cO(Y)$.
\end{lem}

\begin{proof}
We have  {$([0,\infty)\times A)\cup ([0,\infty)\times B)=[0,\infty)\times Y$.}
 
Let $s$ and $U$ be as in \cref{fewoifweifoew23242342341}, let $\phi\colon  {\IN} \to \cP(Y\times Y)^{\Gamma}$ be a $\cT^{\Gamma}$-admissible function and let
$\kappa\colon [0,\infty)\to [0,\infty)$ be monotoneously decreasing  such that $\lim_{u\to \infty} \kappa(u)=0$.  
The pair $\psi:=(\phi,\kappa)$ determines the invariant entourage
\[U_{\psi}:=\{((a,x),(b,y))\in  ([0,\infty)\times Y)^{\times 2} \:|\: |a-b|\le \kappa(\max\{a,b\}) \:\&\: (x,y)\in \phi(\lceil a \rceil )\cap \phi(\lceil b \rceil)   \}\ .\]

For $W $ in $\cC^{\Gamma}$ and $r$ in $(0,\infty)$ we consider the entourage  $W_{r}:=U_{r}\times W$ of $[0,\infty)_{d}\otimes Y$.
The entourages of the form $U_{\psi}\cap W_{r}$ for all $\psi$ as above, $r$ in $(0,\infty)$ and  $W $ in $\cC^{\Gamma}$ are cofinal     in the  hybrid coarse structure of $ \cO(Y)$. 

We now fix $\psi$, $W $   {and $r$} as above. We must show that there exist $r^{\prime}$ in $(0,\infty)$, $W^{\prime}$ in $\cC^{\Gamma}$, and $\psi^{\prime}$ such that
\begin{equation}\label{jkcjekhduihui2hr}(U_{\psi}\cap W_{r} )[[0,\infty)\times A]\cap (U_{\psi}\cap W_{r})[[0,\infty)\times B]\subseteq (U_{\psi^{\prime}}\cap W^{\prime}_{r^{\prime}})[[0,\infty)\times (A\cap B)]\ .\end{equation}
 Using coarse excisiveness of $(A,B)$  we can choose an invariant entourage $W^{\prime}$ of $X$   such that $W[A]\cap W[B]\subseteq W^{\prime}[A\cap B]$. We further set
  $r^{\prime} := r$.   
Then
\begin{equation}\label{rigjrwiogiojio4j4t} W_{r} [[0,\infty)\times A]\cap W_{r}[[0,\infty)\times B]\subseteq W^{\prime}_{r}[[0,\infty)\times (A\cap B)] .\end{equation}
By $\cT^{\Gamma}$-admissibility of $\phi$
there is $u_{0}$ in ${\IN}$ such that $\phi(u)\subseteq U$ for all $u$ in ${\IN}$ with $u\ge u_{0}$.
We define \[\phi^{\prime}\colon  {\IN} \to  {\cP(Y\times Y)^{\Gamma}}\ , \quad \phi^{\prime}(u):=\left\{\begin{array}{cc} W^{\prime} &  {u < u_0} \\ s(\phi(u))&u\ge u_{0}\end{array}\right.
\]
Then $\phi^{\prime}$ is $\cT^{\Gamma}$-admissible. We further set $\psi^{\prime}:=(\phi^{\prime},\kappa)$.
We claim that
\begin{equation}\label{kleklfefwefewf}
(U_{\psi}\cap W_{r})[[0,\infty)\times A]\cap  (U_{\psi}\cap W_{r}) [[0,\infty)\times B]
\subseteq U_{\psi^{\prime}}[ [0,\infty)\times (A\cap B)]\ .
\end{equation}
Consider a point $(u,z)$ in  $(U_{{\psi}}\cap W_{r})[[0,\infty)\times A]\cap   (U_{{\psi}}\cap W_{r}) [[0,\infty)\times B]
$. Then there exists
$(a,x)$ in $[0,\infty)\times A$ such that we have
\[|a-u|\le \kappa(\max\{a,u\}) \quad \text{and} \quad (z,x)\in \phi ({\lceil u \rceil})\cap \phi({\lceil a \rceil})\ ,\]
and there exist $(b,y)$ in  $[0,\infty)\times B$ such that
\[|b-u|\le  \kappa(\max\{b,u\}) \quad \text{and} \quad (z,y)\in \phi({\lceil u \rceil})\cap \phi({\lceil b \rceil})\ .\]
In particular, we have $z\in (\phi({\lceil u \rceil}) \cap W)[A]\cap (\phi({\lceil u \rceil}) \cap W)[B]$. If $u\ge u_{0}$, then we have  $z\in \phi^{\prime}({\lceil u \rceil}) [A\cap B]$ {by the corresponding property of $s$ (see \cref{fewoifweifoew23242342341})}.
Let $w $ in $A\cap B$ be such that we have $(z,w)\in \phi^{\prime}(\lceil u \rceil)$. Then
$((u,z),(u,w))\in U_{{\psi}^{\prime}}[[0,\infty)\times (A\cap B)]$. 
If $u<u_{0}$, then again $((u,z),(u,w))\in U_{{\psi}^{\prime}}[[0,\infty)\times (A\cap B)]$ because of the choice of $W^{\prime}$.

The relations \eqref{rigjrwiogiojio4j4t} and \eqref{kleklfefwefewf} together imply \eqref{jkcjekhduihui2hr}.
\end{proof}

\begin{rem}\label{vvwoivwivweoivjweiov}
In the above proof, in contrast to the general Decomposition Theorem~\ref{sdf8923njn} for hybrid structures, we do not use that $Y$ is Hausdorff.
\end{rem}
 
\subsection{The cone at infinity}
\label{secjknds2sd23}
  
In this section we will define the ``cone at infinity'' functor.
It fits into the cone fiber sequence.
We  discuss invariance under coarsenings and  calculate it for  discrete spaces.  Furthermore, we show that  the ``cone at infinity'' is  excisive and homotopy invariant.
 
 If $Y$ is a  $\Gamma$-uniform bornological coarse space $Y$, then 
  $\cO(Y)$ has a canonical big family $\cY(Y)$ given by   \eqref{fewklfjewfejewfewf}. 
  Recall the notation \eqref{fkhwiufuiz823zr824234242424234234234234}.
 \begin{ddd}
We define the functor
\[\cO^{\infty}\colon \Gamma\UBC\to \Gamma\Sp\cX\] by
\[\cO^{\infty}(Y):=\Yo^{s}(\cO(Y) ,\cY(Y) )\ .\qedhere\]
\end{ddd}

{Recall that $F_{\cT}$ is the functor from $\Gamma$-uniform bornological coarse spaces to $\Gamma$-bornological coarse spaces which forgets the uniform structure.}
For $n$ in $\nat$ let $([0,n]\times Y)_{\cO(Y)}$ denote the $\Gamma$-bornological coarse space given by the subset $[0,n]\times Y$  of $ \cO(Y)$  with the  induced structures. The inclusion
\[F_{\cT}(Y)\to ([0,n]\times Y)_{\cO(Y)}\ , \quad y\mapsto (0,y)\]
is an equivalence of $\Gamma$-bornological coarse spaces for every integer $n$. Hence we have an equivalence
\[\Yo^{s}(Y)\simeq \Yo^{s}(\cY(Y)_{\cO(Y)})\ .\]

\begin{kor}\label{wefoiuweiofuewwefwwefwf}
For every $\Gamma$-uniform bornological coarse space $Y$ we have a natural fiber sequence
\begin{equation}\label{gerpogikrepoergerg}
 \Yo^{s}(F_{\cT}(Y))\to \Yo^{s}(\cO(Y))\to \cO^{\infty}(Y) \xrightarrow{\partial} \Sigma \Yo^{s}(F_{\cT}(Y)) 
\end{equation}
in $\Gamma\Sp\cX$.
\end{kor}

\begin{proof}
The fiber sequence is associated to the pair $(\cO(Y),\cY(Y))$, see \cref{kjeflwfjewofewuf98ewuf98u798798234234324324343}.\ref{ifjweifjewiojwefw231}.
\end{proof}

Let $Y$ be a $\Gamma$-uniform bornological coarse space.
Then we consider the   $\Gamma$-bornological coarse space
$\cO(Y)_{-}$  obtained from the $\Gamma$-uniform bornological coarse space
$\R\otimes Y$ by taking the hybrid coarse structure \ref{efioofwefwefwefewf} associated to the big family
$ ((-\infty,n]\times Y)_{n\in \nat}$. 
Note that the subset $[0,\infty)\times Y$ of $\cO(Y)_{-}$ with the induced structures is the cone $\cO(Y)$. 
We then have maps of $\Gamma$-bornological coarse spaces
\[F_{\cT}(Y)\xrightarrow{i} \cO(Y)\xrightarrow{j} \cO(Y)_{-}\xrightarrow{d} F_{\cT}(\R_{du}\otimes Y)\ .\]
The first two maps $i$ and $j$ are the inclusions, and the last map $d$  is given by the identity of the underlying sets.

\begin{prop}\label{ofwpefwefewfewf}
We have a commutative diagram in $\Gamma\Sp\cX$
\begin{equation}\label{voihoi3f3f4f}
\xymatrix{\Yo^{s}(\cO(Y))\ar@{=}[d]\ar[r]^{j}&\Yo^{s}(\cO(Y)_{-})\ar[d]^{\simeq}\ar[r]^-{d}&\Yo^{s}(F_{\cT}(\R_{du}\otimes Y))\ar[d]^{\simeq}\\
\Yo^{s}(\cO(Y))\ar[r]&\cO^{\infty}(Y)\ar[r]^-{\partial}&\Sigma \Yo^{s}(F_{\cT}(Y))}
\end{equation}
\end{prop}

\begin{rem}
This proposition identifies a segment of the cone sequence \eqref{gerpogikrepoergerg} with a sequence represented by 
maps between   $\Gamma$-bornological coarse spaces. It in particular shows that the cone $\cO^{\infty}(Y)$ is represented by the $\Gamma$-bornological coarse space $\cO(Y)_{-}$.
\end{rem}

\begin{proof}[Proof of \cref{ofwpefwefewfewf}]
We  consider the diagram of motivic coarse spectra
\begin{equation}\label{vhiuf3f34}
\xymatrix{\Yo^{s}(F_{\cT}(Y))\ar[r]^{i}\ar[d]&\Yo^{s}(\cO(Y))\ar[r]\ar[d]^{j}&\Yo^{s}(F_{\cT}([0,\infty)\otimes Y))\ar[d]\\\Yo^{s}(F_{\cT}((-\infty,0]\otimes Y))\ar[r]&\Yo^{s}(\cO(Y)_{-})\ar[r]^-{d}&\Yo^{s}(F_{\cT}(\R_{du}\otimes Y))}
\end{equation}
The left   and right vertical and the lower left horizontal map are given by the canonical inclusions.
The upper right horizontal map is the identity map of the underlying sets.
This diagram commutes since it is obtained by applying $\Yo^{s}$ to a commuting diagram of bornological coarse spaces.

The left square in \eqref{vhiuf3f34} is cocartesian since the pair
$((-\infty,0]\times Y, \cO(Y))$ in $\cO(Y)_{-}$ is coarsely excisive.
Furthermore, since
$((-\infty,0]\times Y,[0,\infty)\times Y)$ is coarsely excisive in $F_{\cT}(\R_{du}\otimes Y)$ the 
outer square is cocartesian. It follows that the right square is cocartesian.

Since the upper right and the lower left corners in \eqref{vhiuf3f34}  are trivial by flasqueness of
the rays the diagram is equivalent to the composition
\[\xymatrix{\Yo^{s}(F_{\cT}(Y))\ar[r]^{i}\ar[d]&\Yo^{s}(\cO(Y))\ar[r]\ar[d]^{j}&0\ar[d]\\0\ar[r]&\Yo^{s}(\cO(Y)_{-})\ar[r]^-{d}&  \Yo^{s}( F_{\cT}(\R_{du}\otimes  Y))}\]
of cocartesian squares.
Note that $\cO^{\infty}(Y)$ is defined as the cofiber of the left upper horizontal map $i$. Hence the 
  left square yields  the middle vertical  equivalence in \eqref{voihoi3f3f4f}.

 The outer square yields the equivalence
$ \Yo^{s}( F_{\cT}(\R_{du}\otimes  Y))\simeq \Sigma \Yo^{s}(Y)$. The
 right square then identifies  $d$ 
  with the boundary map $\partial$ of the cone sequence.
\end{proof}

Next we will observe that $\cO^{\infty}(Y)$ is essentially independent of the coarse structure on $Y$.
Let $Y$ be a $\Gamma$-uniform bornological coarse space  with coarse structure $\cC$, bornology $\cB$ and uniform structure $\cT$. Let $\cC^{\prime}$   be a $\Gamma$-coarse structure on $Y$ such that $\cC\subseteq \cC^{\prime}$ and $\cC^{\prime}$  is still compatible with the bornology. We write $Y^{\prime}$ for the $\Gamma$-uniform bornological coarse space obtained from $Y$ by replacing the coarse structure $\cC$ by $\cC^{\prime} $.   Then the identity  map of the underlying sets is a morphism $Y\to Y^{\prime}$ of {$\Gamma$-uniform bornological coarse spaces}.
We will call such a morphism a \emph{coarsening}. 

Let  $Y$ be a $\Gamma$-uniform bornological coarse space.

\begin{prop}\label{fiiwofwufw9e8fuew9fwefwef}
If $Y\to Y^{\prime}$ is a coarsening, then the induced map
\[\cO^{\infty}(Y)\to \cO^{\infty}(Y^{\prime})\]
is an equivalence.
\end{prop}

\begin{proof}
 By definition of $\cO^{\infty}(Y)$ we have 
\begin{equation}\label{vrevkervkeio43uoir34r4r34t4tr4rr}
\cO^{\infty}(Y)\simeq \colim_{n\in \nat }  \Yo^{s}(([0,\infty)_{du}\otimes Y)_{h} ,[0,n] \times  Y   )\ ,
\end{equation} 
where the subsets    $[0,n] \times  Y   $ of $ ([0,\infty)_{du}\otimes Y)_{h}$ have the induced bornological coarse structure.   By $u$-continuity of $\Yo^{s}$ we have
\begin{equation}\label{vrevkervkeio43uoir34r4r34}
\cO^{\infty}(Y )\simeq \colim_{n\in \nat }\colim_{U} \Yo^{s}(([0,\infty)_{du} \otimes  Y  )_{U} ,([0,n] \times  Y )_{U}   )\ ,
\end{equation}
 where $U$ runs over the $\Gamma$-invariant entourages of $([0,\infty)_{du}\otimes Y)_{h}$. Here  for a subset $X^{\prime}$ of $  X$ the notation $X^{\prime}_{U}$ denotes the set $X^{\prime}$ with the structures induced from $X_{U}$, i.e., $X^{\prime}_{U}$ is a short-hand notation for $X^{\prime}_{X_{U}}$.
For every integer $n $, there is a cofinal set of entourages $U$ such that the pair
\[([0,n] \times Y  ,[n,\infty) \times Y  )\]
is coarsely excisive {on} $([0,\infty) \times Y)_{U}$. In fact, this is  a coarsely excisive pair for any entourage $U$ {that allows propagation from $\{n\} \times Y$} in the direction of the ray.
Since the Yoneda functor $\Yo^{s}$ is excisive we get the equivalence
\begin{equation}\label{vrevkervkeio43uoir34r4r341}
\cO^{\infty}(Y )\simeq \colim_{n\in \nat }\colim_{U} \Yo^{s}(([n,\infty) \times  Y  )_{U} ,(\{n\} \times  Y )_{U} )\ .
\end{equation}
 
In general, for a $\Gamma$-bornological coarse space $X$  with coarse structure $\cC$ and an invariant subset $Z$ we have an equivalence
\begin{equation}\label{hckjhwkjehcewchiuhuwcwcw}
\colim_{U\in \cC^{\Gamma}}\Yo^{s}(Z_{U})\simeq   \colim_{U\in \cC^{\Gamma}}
\Yo^{s}(Z_{(Z\times Z)\cap U }) \end{equation}
(here we  must not omit the colimit).
We insert this into  \eqref{vrevkervkeio43uoir34r4r341} and get 
\begin{equation}\label{vrevkervkeio43uoir34r4r343}
\cO^{\infty}(Y)\simeq \colim_{n\in \nat }\colim_{U} \Yo^{s}(([n,\infty) \times  Y)_{U_{n}  } ,(\{n\}\times  Y)_{U_{n} } )\ ,
\end{equation}
where we use the abbreviation $U_{n}:= (([n,\infty)\times  Y)\times ([n,\infty)\times  Y))\cap U$.
We can now interchange the order of the colimits and get
\begin{equation}\label{vrevkervkeio43uoir34r4r34122}
\cO^{\infty}(Y)\simeq  \colim_{U}\colim_{n\in \nat } \Yo^{s}(([n,\infty) \times Y )_{U_{n}  } ,(\{n\}\times  Y )_{U_{n} } )\ .
\end{equation}

We argue now that in this formula we can replace the colimit over the invariant entourages  $U$
of  $([0,\infty)_{du}\otimes Y)_{h}$
by the colimit over all invariant entourages $U^{\prime}$ of $([0,\infty)_{du}\otimes Y^{\prime})_{h}$.
To this end we consider the generating  entourages $U_{\psi}\cap W^{\prime}$ of $\cO(Y^{\prime})$ (see the proof of \cref{ifeioewfueoiwfwfewfwf} for notation).  Since $\cT$ and $\cC$ are compatible,   there exists an integer  $n_{0} $  sufficiently large such that $\phi(n_{0})\in \cC$. But then, since $\phi$ is monotoneous, we have
$\phi(x)\in \cC$ for every $x$ in $[n_{0},\infty)$. We conclude that for every integer $n$ with $n\ge n_{0}$ we have
\[(U_{\psi}\cap W^{\prime})\cap \big( ([n,\infty)\times Y)\times ([n,\infty)\times Y) \big) \subseteq U_{\psi}\cap (W^{\prime}\cap \phi(n))\] and $W^{\prime}\cap \phi(n)\in \cC$.

This gives 
 \begin{equation}\label{vrevkervkeio43uoir34r4r34122r}
\cO^{\infty}(Y)\simeq  \colim_{U^{\prime}}\colim_{n\in \nat } \Yo^{s}(([n,\infty) \times  Y)_{U^{\prime}_{n}  } ,(\{n\}\times  Y)_{U^{\prime}_{n} } )\ ,
\end{equation}
 where now $U^{\prime}$ runs over the invariant entourages of
 $([0,\infty)_{du}\otimes Y^{\prime})_{h}$.  Going the argument above backwards with $Y$ replaced by $Y^{\prime}$ we end up  with 
  \begin{equation}\label{vrevkervkeio43uoir34r4r34rr334343434}
\cO^{\infty}(Y)\simeq \colim_{n\in \nat }  \Yo^{s}(([0,\infty)_{du}\otimes  Y^{\prime})_{h}   ,[0,n] \times Y    )\simeq \cO^{\infty}(Y^{\prime})
\end{equation}
and this completes the proof.
\end{proof}

In the following proposition we use the invariance under coarsening in order to calculate the value of the $\cO^{\infty}$-functor on $\Gamma$-uniform bornological coarse spaces whose underlying uniform structure is discrete.

For a $\Gamma$-uniform bornological coarse space  $X$ which is discrete as a uniform space let $X_{disc}$ denote the $\Gamma$-uniform bornological coarse space obtained by replacing the coarse structure by the discrete coarse structure.

\begin{rem}
If $X$ is not discrete as a uniform space, then the discrete coarse structure is not compatible with the uniform structure.
\end{rem}

Let $X$ be a $\Gamma$-uniform bornological coarse space. 
\begin{prop}\label{oiefjewoifo23r4243455435345}
If $X$ is discrete as a uniform space, then we have an equivalence
\[\cO^{\infty}(X)\simeq \Sigma \Yo^{s}({F_\cT(X_{disc})})\]
in $\Gamma\Sp\cX$.
\end{prop}
\begin{proof}
{Since  $X_{disc} \to X$  is a coarsening, by \cref{fiiwofwufw9e8fuew9fwefwef}  we have  an equivalence}
\[\cO^{\infty}({X_{disc}})\xrightarrow{\simeq} \cO^{\infty}(X)\ .\]
By \cref{ropjreopg0iu09t34t34t34t3} we know that $\cO({X_{disc}})$ is flasque  and hence $\Yo^{s}(\cO({X_{disc}}))\simeq 0$. The fiber sequence obtained in \cref{wefoiuweiofuewwefwwefwf} yields an equivalence $\cO^{\infty}({X_{disc}})\simeq \Sigma \Yo^{s}({F_\cT(X_{disc})})$ as desired.
\end{proof}

Next we discuss excision and homotopy invariance for $\cO^{\infty}$.

Let $Y$ be a $\Gamma$-uniform bornological coarse space  and $A,B$ be $\Gamma$-invariant subsets of $Y$.
 
\begin{kor}\label{kldjedjoiewufowe23435335}
If $(A,B)$ is an equivariant uniformly and coarsely excisive decomposition, then the following square in $\Gamma\Sp\cX$ is cocartesian:
\begin{equation}\label{dqwd234fe23r32r32}
\xymatrix{ \cO^{\infty}(A\cap B)\ar[r]\ar[d]&\cO^{\infty}(B)\ar[d]\\\cO^{\infty}(A)\ar[r]&\cO^{\infty}(Y)}
\end{equation}
\end{kor}

\begin{proof}
Since $(A,B)$ is coarsely excisive
 the square
\[\xymatrix{ \Yo^{s}(\cF_{\cT}(A)\cap \cF_{\cT}(B))\ar[r]\ar[d]&\Yo^{s}(\cF_{\cT}(B))\ar[d]\\\Yo^{s}(\cF_{\cT}(A))\ar[r]&\Yo^{s}(\cF_{\cT}(Y))}\] is cocartesian.
Furthermore, by \cref{ifeioewfueoiwfwfewfwf} the square
\[\xymatrix{ \cO(A\cap B)\ar[r]\ar[d]&\cO(B)\ar[d]\\\cO(A)\ar[r]&\cO(Y)}\]
is  cocartesian. Now it just remains to use the cone sequence \eqref{gerpogikrepoergerg} 
in order to conclude that the square \eqref{dqwd234fe23r32r32} is cocartesian.
\end{proof}

\begin{rem}\label{oerjgroigrgregegergeg}
If we assume that  the underlying uniform space of  $Y$ is Hausdorff, then we could drop the assumption that $(A,B)$  is coarsely excisive. In this case it will follow from the Decomposition Theorem \ref{sdf8923njn} applied to   the equivariant  uniform  decomposition
$([0,\infty) \times A,[0,\infty) \times B)$ of $[0,\infty)_{du}\otimes Y$
that \eqref{dqwd234fe23r32r32} is cocartesian.
\end{rem}

\begin{kor}\label{fijofiwewefewf}
The functor $\cO^{\infty}\colon \Gamma\UBC\to \Gamma\Sp\cX$ is homotopy invariant.
\end{kor}

\begin{proof}
Let $Y$ be a $\Gamma$-uniform bornological coarse space. Let $h\colon [0,1]_{du}\otimes Y\to Y$ be the projection. By functoriality of the cone
  we  get  the morphism
\[\cO(h)\colon \cO([0,1]_{du}\otimes Y)\to \cO(Y)\ .\]
We now observe that
\[\cO([0,1]_{du}\otimes Y)\cong([0,\infty)_{du}\otimes [0,1]_{du}\otimes Y)_{h}\cong  ([0,1]_{du}\otimes[0,\infty)_{du}\otimes Y)_{h}\ .\]
By the Homotopy Theorem \ref{fiewfewifieowifowioioio23r} we get an equivalence
\[\Yo^{s}(\cO([0,1]_{du}\otimes Y))\simeq \Yo^{s}(( [0,1]_{du}\otimes [0,\infty)_{d}\otimes Y)_{h})\simeq \Yo^{s}(([0,\infty)_{du}\otimes Y)_{h})\simeq \Yo^{s}(\cO(Y))\ .\]
Since the projections $[0,1]_{du}\otimes [0,n]_{du}\otimes Y\to [0,n]_{du}\otimes Y$  induce equivalences of underlying $\Gamma$-bornological coarse spaces we conclude that the projection $h$ induces an equivalence
\[\cO^{\infty}([0,1]_{du}\otimes Y)\to \cO^{\infty}(Y)\] in $\Gamma\Sp\cX$.
\end{proof}

\section{Topological assembly maps}\label{feijwifoewfffwfewfewf}

The overall theme of this section is the interplay between equariant coarse homology theories and equivariant homology theories.

We start in \cref{wfifjweoifjwoifewfwfewf} with the general  discussion of equivariant homology theories, and then in \cref{giuheriugregege443435}  we modify the  ``cone at infinity'' functor $\cO^\infty$ to get an equivariant $\Gamma\Sp\cX$-valued homology theory \[\cO_{{\homolg}}^{\infty}\colon \Gamma\Top\to \Gamma\Sp\cX\ .\] We introduce classifying spaces  $E_{\cF}\Gamma$ for families $ \cF$ of subgroups  and define
the motivic  assembly map \[\alpha_{\cF}\colon \cO_{{\homolg}}^{\infty}(E_{\cF}\Gamma)\to  \cO_{{\homolg}}^{\infty}(*)\] in \cref{secjks9990rfwe}. Using the cone sequence we define further versions $\alpha_{X,Q}$  of the motivic  assembly map with {twist $Q$}
and discuss some instances where it is an equivalence.  Finally, in \cref{ifjwoifwefewfewff} we use the functor $\cO_{{\homolg}}^{\infty}$ in order to derive equivariant homology theories from coarse homology theories.

\subsection{Equivariant homology theories}\label{wfifjweoifjwoifewfwfewf}

In this  section we will  {recall} the notion of a (strong) equivariant homology theory on $\Gamma$-topological spaces.

Our basic category of topological spaces is the convenient   category $\Top$    of compactly generated weakly Hausdorff spaces. A map between topological spaces  
is  a weak equivalence if it induces an isomorphism between the sets of connected components and isomorphisms of homotopy groups in all positive degrees and for all choices of base points.
The  $\infty$-category obtained from (the nerve of) $\Top$ by inverting these weak equivalences is a model for the presentable $\infty$-category $\Spc$ of spaces. In particular,  we have the localization functor  \begin{equation}\label{fjiiioioio}\kappa\colon \Top\to \Spc\ .
\end{equation}

A $\Gamma$-topological space is a topological space with an action of the group $\Gamma$ by automorphisms. We  denote the category of $\Gamma$-topological spaces and equivariant continuous maps by  $\Gamma\Top$. A weak equivalence between $\Gamma$-topological spaces is a $\Gamma$-equivariant map which induces weak equivalences on fixed-point spaces for all subgroups of $\Gamma$. We will model this homotopy theory by presheaves on the orbit category of $\Gamma$.

The orbit category $\Orb(\Gamma)$ of $\Gamma$ is the category of transitive  $\Gamma$-sets and equivariant maps. A $\Gamma$-set can naturally be considered as a discrete $\Gamma$-topological space. In this way we get   a fully faithful functor
$\Orb(\Gamma)\to \Gamma\Top$. For a transitive  $\Gamma$-set $S$  and  $\Gamma$-topological space $X$ we consider the topological space   $\Map_{\Gamma\Top}(S,X)$  of equivariant maps from $S$ to $X$.  
  
\begin{rem}
We consider a transitive  $\Gamma$-set $S$.  If we fix a base point $s $ in $S$ and denote the stabilizer of $s$ by $\Gamma_{s}$, then we get an identification   $\Map(S,X)\simeq X^{\Gamma_{s}}$, where $X^{\Gamma_{s}}$ is the subspace of $\Gamma_{s}$-fixed points. 
\end{rem}
We define a functor
\begin{equation}\label{oiiuhf4u9824f23ff}
\ell\colon \Gamma\Top\to \PSh(\Orb(\Gamma))\ \, \text{by} \ \, \ell(X) (S):=\kappa( \Map_{\Gamma\Top} (S,X))\ \, \text{for} \ \,  S\in \Orb(\Gamma)\ . 
\end{equation}

\begin{rem}\label{wrfiowfewfewfewfew}
A map between topological spaces is a weak equivalence if and only if its image under $\kappa$ is an equivalence. Consequently,
a map between $\Gamma$-topological spaces is a weak equivalence if and only if its image under $\ell$ is an equivalence. 
 By Elmendorf's theorem \cite[Thm.~VI.6.3]{MR1413302} (which boils down to the assertion that $\ell$ is essentially surjective) the functor $\ell$ induces   an equivalence
 \begin{equation}\label{gijoirfjoervververv}
\Gamma\Top[W^{-1}]\xrightarrow{\simeq} \PSh(\Orb(\Gamma))\ , 
\end{equation}
where $\Gamma\Top[W^{-1}] $ denotes   the $\infty$-category  obtained from $\Gamma\Top$ by inverting the weak equivalences. Occasionally we will use the fact that a weak equivalence in $\Gamma\Top$ between $\Gamma$-CW-complexes is actually a  homotopy equivalence in $\Gamma\Top$.
\end{rem}

Let $\bC$ be a   cocomplete $\infty$-category. By the universal property of the presheaf category we have an equivalence of $\infty$-categories
\begin{equation}\label{chbwehjfbwejf87zr84r343t3t}
\Fun^{\mathrm{colim}}(\PSh(\Orb(\Gamma) ),\bC)\simeq \Fun(\Orb(\Gamma) ,\bC)\ ,
\end{equation}
where the superscript $\mathrm{colim}$ stands for colimit-preserving. The localization functor \eqref{oiiuhf4u9824f23ff}
 induces a faithful restriction functor
\begin{equation}\label{fwefewiuiu425435}
\Fun^{\mathrm{colim}}( \PSh(\Orb(\Gamma),\bC )\to \Fun(\Gamma\Top,\bC)\ .
\end{equation}
 
From now on we assume that $\bC$ is cocomplete and stable.

Our preferred definition of the notion of an equivariant $\bC$-valued  homology theory would be the following.

Let $E\colon \Gamma\Top\to \bC$ be a functor.
\begin{ddd}\label{dddjk4r}
$E $ is called a \emph{strong equivariant $\bC$-valued homology theory} if it is in the essential image of \eqref{fwefewiuiu425435}.
\end{ddd}
Assume that we are given a functor $E $ as above. If it sends weak equivalences to equivalences, then, using the equivalence \eqref{gijoirfjoervververv}, it extends essentially uniquely to a functor $\PSh(\Orb(\Gamma)) \to \bC$. The functor $E$ is an equivariant $\bC$-valued  homology theory if this extension preserves colimits. In general it seems to be complicated to check these conditions if $E$ is given by some geometric construction. 
For this reason we add the adjective \emph{strong} in order to distinguish this notion from the \cref{fweopfkjpwef345345} of an equivariant $\bC$-valued homology theory that we actually work with.

Let $E\colon \Gamma\Top\to \bC$ be a functor. We extend $E$ to pairs $(X,A)$ of $\Gamma$-topological spaces and subspaces  by setting 
\[E(X,A):=\Cofib \big( E(*)\to E(X\cup_{A}\Cone(A)) \big)\ ,\]
where $\Cone(A)$ denotes the cone over $A$ and $*$ is the base point of the cone.
\begin{ddd}\label{fweopfkjpwef345345}
The functor $E$ is called an \emph{equivariant $\bC$-valued homology theory} if it has the following properties:
\begin{enumerate}
\item(Homotopy invariance) For every  $\Gamma$-topological space $X$ the projection induces an equivalence
\[E([0,1]\times X)\to E(X)\ .\]
\item (Excision) If $(X,A)$ is a pair of $\Gamma$-topological spaces and $U$ is an invariant open subset of $A$ such that $\overline U$ is contained in the interior of $A$, then the inclusion $(X\setminus U,A\setminus U)\to (X,A)$ induces an equivalence \[E(X\setminus U,A\setminus U)\to E(X,A)\ .\]
\item (Wedge axiom) For every family $(X_{i})_{i\in I}$ of $\Gamma$-topological spaces the canonical map
\[\bigoplus_{i\in I} E(X_{i})\xrightarrow{\simeq}E \Big( \coprod_{i\in I} X_{i} \Big)\]
is an equivalence.\qedhere
\end{enumerate}
\end{ddd}

\begin{rem}\label{kefhwefiuziu24234324}
In order to verify that $E$ satisfies excision one must show that $E$ sends the square
\[\xymatrix{\Cone(A\setminus U)\ar[r]\ar[d]&(X\setminus U) \cup_{A\setminus U}\Cone(A\setminus U)\ar[d]\\\Cone(A)\ar[r]&X\cup_{A}\Cone(A)}\] to a push-out square. For homotopy invariant functors $E$ this is equivalent to the property that $E$ sends the right vertical map in the square above to an equivalence.

For homotopy invariant functors $E$  excision  follows from the stronger condition  of closed excision, i.e., that for every decomposition $(A,B)$ of $X$ into closed invariant subsets     the diagram
\[\xymatrix{E(A\cap B)\ar[r]\ar[d]&E(A)\ar[d]\\E(B)\ar[r]&E(X)}\]
is a push-out square. This can be seen as follows. Assume that $E$ is homotopy invariant and satisfies closed excision. Let   $X$, $ A$, and $ U$ be as above. Then we have a closed decomposition
\[(X\setminus U\cup_{A\setminus U} \Cone(A\setminus U)\: ,\:  \Cone(\overline U))\]
of $X\cup_{A} \Cone(A)$ with intersection
$\Cone(\overline U\setminus U)$. By closed excision we get the cocartesian square
\[\xymatrix{E(\Cone(\overline U\setminus U))\ar[r]\ar[d]&E(X\setminus U\cup_{A\setminus U} \Cone(A\setminus U))\ar[d]\\ E(\Cone(\overline U))  \ar[r]&E(X\cup_{A} \Cone(A))}\]
Since $E$ is homotopy invariant it sends the left vertical map to an equivalence since  cones are contractible. Consequently, the right vertical map is an equivalence, too.
\end{rem}

\begin{rem}
Let $E$ be an equivariant $\bC$-valued homology theory. In general we can not expect that it factorizes over the localization \eqref{oiiuhf4u9824f23ff}.

Using the equivalence \eqref{chbwehjfbwejf87zr84r343t3t} the restriction of $E$ to the orbit category gives rise to a strong equivariant $\bC$-valued homology theory $E^{\%}$ which comes with a natural transformation $E^{\%}\to E$. Using the theory developed {by Davis--L\"{u}ck} \cite[Sec.~3]{davis_lueck} one can check that
\[E^{\%}(X)\xrightarrow{\simeq} E(X)\]
for all $\Gamma$-CW-complexes $X$.
\end{rem}

\subsection{The cone as an equivariant homology theory}\label{giuheriugregege443435}

In \cref{secjknds2sd23} we have seen that the ``cone at infinity'' functor $\cO^\infty$ is a homotopy invariant and excisive functor from   $\Gamma\UBC$ to $\Sp\cX$. In this section we modify this functor in order  to get an equivariant homology theory
$\cO_{{\homolg}}^{\infty}\colon \Gamma\Top\to \Gamma\Sp\cX$.

If $X$ is a $\Gamma$-uniform space, then we can consider $X$ as a $\Gamma$-uniform bornological coarse space $X_{max,max}$ by equipping the uniform space $X$ in addition with the maximal coarse structure and the maximal bornology. 
In this way we get a functor
\begin{equation}\label{verivjiojioru34f34f}
\cM \colon \Gamma\bU\to \Gamma\UBC\ , \quad X\mapsto \cM(X):=X_{max,max}\ .
\end{equation}

Let $Y$ be a $\Gamma$-set and $Q $ be a $\Gamma$-bornological coarse space. The projection $Y_{max,max}\to *$ induces a morphism
\begin{equation}\label{jhqehuidhwdqwdqwd}
\Yo^{s}( {Y_{max,max}})\otimes \Yo^{s}(Q)\to \Yo^{s}(Q)
\end{equation} 
in $\Gamma\Sp\cX$.

\begin{lem}\label{roigirhgreiughreige}
If the underlying set of $Q$ is a free $\Gamma$-set, then
\eqref{jhqehuidhwdqwdqwd} is an equivalence.
\end{lem}

\begin{proof}
Since $Q$ is a free $\Gamma$-set we can choose a $\Gamma$-equivariant map of sets $\kappa\colon Q\to Y$. The map
$(\kappa,\id)\colon Q\to  {Y_{max,max}}\otimes  Q$ is then a morphism of $\Gamma$-bornological coarse spaces. It is an inverse to the projection up to equivalence. Hence the projection is an equivalence of $\Gamma$-bornological coarse spaces and therefore
$\Yo^{s}( {Y_{max,max}}\otimes Q)\to \Yo^{s}(Q)$ is an equivalence.
We now use that
$\Yo^{s}$ is a symmetric monoidal functor (see \cref{secni2039}) in order to rewrite the domain of the morphism as in \eqref{jhqehuidhwdqwdqwd}.
\end{proof}

The functor
\begin{equation}
\label{eqnj77we23}
\cO^{\infty}\circ \cM\colon \Gamma\bU\to \Gamma\Sp\cX
\end{equation}
behaves very much like an equivariant homology theory.
Of course, it is not defined on $\Gamma\Top$,  but on  $\Gamma\bU$. On the other hand   it is homotopy invariant by \cref{fijofiwewefewf}, and  sends equivariant uniform decompositions to push-outs by \cref{kldjedjoiewufowe23435335}. The drawback is that it in general only preserves finite coproducts.

In order  to improve  these points  we define a new functor 
\[\cO_{{\homolg}}^{\infty}\colon \Gamma\Top\to \Gamma\Sp\cX\]
by first restricting $ \cO^{\infty}\circ \cM $ to the subcategory of $\Gamma$-compact $\Gamma$-metrizable spaces and then left-Kan extending the result  to $\Gamma\Top$. In the following we describe the details.

{Let $X$ be a $\Gamma$-toplogical space. Recall that $X$ is \emph{$\Gamma$-compact} if there exists a compact subset $K$ of $X$ such that $\Gamma K = X$,
and that $X$ is \emph{$\Gamma$-metrizable} if there exists a $\Gamma$-invariant metric on $X$ which induces the topology of $X$.}

We denote by $\Gamma\Top^{cm}$ the full sub-category of $\Gamma\Top$ spanned by all the $\Gamma$-compact and $\Gamma$-metrizable $\Gamma$-topological spaces. Associated to any $X$ in $\Gamma\Top^{cm}$ we define 
\[ \cN(X) := \{ N \subseteq X \times X \mid N\text{ contains a $\Gamma$-invariant neighborhood of the diagonal}\}\ .\] Furthermore we set $\cU(X) := (X, \cN(X))$.
For a $\Gamma$-invariant metric on $X$ which is compatible with the topology we let $\cT_d$ denote the associated metric uniform structure on $X$.
\begin{lem}\label{lem:nbhdvsmetricentourage}Assume that $X$ is in $\Gamma\Top^{cm}$.
\begin{enumerate} 
	\item $\cN(X)$ is a $\Gamma$-uniform structure. 
	\item If $d$ is a $\Gamma$-invariant metric compatible with the topology, then $\cN(X)=\cT_{d}$. 
	\item The assignment $X \mapsto \cU(X)$ defines a functor $\cU \colon \Gamma\Top^{cm} \to \Gamma\bU$. 
\end{enumerate}
\end{lem}
\begin{proof} By assumption we can choose a $\Gamma$-invariant metric $d$ which is compatible with the topology.
 We claim that for every $\Gamma$-invariant neighborhood $N$ of the diagonal there exists some $T$ in $\cT_d$ such that $T \subseteq N$. The case $N=X\times X$ is trivial and we assume that $N$ is a proper subset. 
 We define a function
 \[ b \colon X \to (0,\infty), \quad x \mapsto \sup \{ \epsilon \in (0,\infty)\mid B_\epsilon(x) \times B_\epsilon(x) \subseteq N \}\ , \]
 where $B_\epsilon(x)$ denotes the open ball of radius $\epsilon$ around $x$ (with respect to $d$).
Since $d$ is compatible with the topology the argument of the $\sup$ is non-empty for every $x$ in $X$ and the value $b(x)$ is indeed positive. Furthermore, the  supremum is attained. Moreover, the supremun can not be infinite since we assume that $N$ is a proper subset.
 Finally, we observe that $b$ is $\Gamma$-equivariant with respect to the trivial $\Gamma$-action on $(0,\infty)$.
 
 We now show that the function $b$ is 1-Lipschitz. Let $x,y$ be two points in $X$ and let $\delta$ be their distance. If both $b(x)$ and  $b(y)$ are less than $\delta$, then so is their distance. Therefore, we can assume that $b(x)\geq b(y)$ and $b(x)\geq \delta$. By the triangle inequality we have
 	\[B_{b(x)-\delta}(y)\times B_{b(x)-\delta}(y)\subseteq B_{b(x)}(x)\times B_{b(x)}(x)\subseteq N\]
 This implies $b(x)\geq b(y)\geq b(x)-\delta$. In particular, $|b(x)-b(y)|\leq \delta$ and $b$ is indeed 1-Lipschitz and thus continuous.
 
By assumption we can choose a compact subset $K$ of $ X$ such that $\Gamma K = X$. Since $b$ is continuous, the restriction  $b|_K \colon K \to (0,\infty)$ attains a positive minimal value, which we denote by  $\epsilon_0$.
 By $\Gamma$-equivariance of $b$, we then have $b(x) \geq \epsilon_0$ for all $x$ in $X$.
 We conclude that the metric uniform entourage $\{ (x,y)\in X\times X \mid d(x,y) < \epsilon_0 \}$ is contained in $N$.
 
 Note that every metric uniform entourage is a neighbourhood of the diagonal since $d$ is compatible with the topology.
 Since both $\cN(X)$ and $\cT_d$ are closed under taking supersets, we have shown that $\cN(X) = \cT_d$. This shows the first two assertions of the Lemma.
 
  We now show the third assertion. Let $f\colon X\to X^{\prime}$ be an equivariant  continuous map between two $\Gamma$-compact and $\Gamma$-metrizable $\Gamma$-topological spaces.  We must show that   $f\colon (X,\cN(X))\to (X^{\prime},\cN(X^{\prime}))$ is uniformly continuous.
  Let $V^{\prime}$ belong to $\cN(X^{\prime})$. Then $V^{\prime}$ contains a $\Gamma$-invariant neighbourhood of the diagonal $U^{\prime}$.  Since $f$ is equivariant and continuous, $(f^{-1}\times f^{-1})(U)$ is a $\Gamma$-invariant neighbourhood of the diagonal of $X$ contained in 
  $(f^{-1}\times f^{-1})(V^{\prime})$. Consequently, $(f^{-1}\times f^{-1})(V^{\prime})\in \cN(X)$. This implies that
  $f\colon (X,\cN(X))\to (X^{\prime},\cN(X^{\prime}))$ is uniformly continuous.
\end{proof}

Let $X$ be a $\Gamma$-topological space, and let $A$ and $B$  be {closed} $\Gamma$-invariant subsets of $X$ such that $A \cup B = X$.
\begin{lem}\label{feifhweifewffewfwf}
 If $X$ is $\Gamma$-compact and $\Gamma$-metrizable, then $(A,B)$ is an equivariant uniform decomposition of $\cU(X)$ (\cref{fewoifweifoew23242342341}).
\end{lem}
\begin{proof}
 We choose a $\Gamma$-invariant metric $d$ on $X$ which is compatible with the topology.
 By \cref{lem:nbhdvsmetricentourage},  {$(X,\cT_d)=\cU(X)$}.
 
  For a subset $Z$ of $Y$ and $e$ in $ (0,\infty)$ we consider the $e$-thickening
\[U_{e}[Z]:=\{y\in Y\:|\: d(y,Z)\le e\}\] (note that $U_{e}$ is defined in \eqref{fhbeufhiu24r}   with a $\le$-relation, too) of $Z$. If $Z$ is invariant, then $U_{e}[Z]$ is again invariant.  

If $A\cap B=\emptyset$, then  by $\Gamma$-compactness of $Y$ the subsets $A$ and $B$   are uniformly separated and  
$(A,B)$ is an equivariant uniform decomposition.

Assume now that $A\cap B\not=\emptyset$.
It suffices to  define a monotoneous   function
$s\colon (0,\infty)\to (0,\infty)$ such that:
\begin{enumerate}
\item \label{iewiouffwewefewfwefwef1} $\lim_{e\to 0} s(e)=0$.
\item \label{iewiouffwewefewfwefwef}  For all $e$  in $(0,\infty)$ we have
\[U_{e}[A]\cap U_{e}[B]\subseteq U_{s(e)}[A\cap B]\ .\]
\end{enumerate}

We define a function $s\colon (0,\infty)\to [0,\infty]$ by 
\[s(e):=\inf\{e^{\prime}\in (0,\infty)\:|\: U_{e}[A]\cap U_{e}[B] \subseteq U_{e^{\prime}}[A\cap B]\}\ .\]
By construction, this function is monotoneous.
Since $A\cap B\not=0$, and since by $\Gamma$-compactness of $Y$ there exists $R$ in $(0,\infty)$ such that $U_{R}[A\cap B]=Y$, the function $s$ is finite. Condition~\ref{iewiouffwewefewfwefwef} follows from this observation.

We  claim that
\begin{equation}
\label{eqjnk092332}
A\cap B=\bigcap_{e>0} U_{e}[A]\cap U_{e}[B]\ .
\end{equation}
It is clear that
\[A\cap B\subseteq \bigcap_{e>0} U_{e}[A]\cap U_{e}[B] \ .\] On the other hand, assume that $y\in Y\setminus  (A\cap B)$. Without loss of generality (interchange the roles of $A$ and $B$, if necessary) we can assume that $y\in Y\setminus A$. Since $A$ is closed  there exists $e$ in $(0,\infty)$  such that $y\not\in  U_{e}[A]$. Hence
$y\in Y\setminus  \bigcap_{e>0} U_{e}[A]\cap U_{e}[B]$. This  shows  the opposite inclusion
\[\bigcap_{e>0} U_{e}[A]\cap U_{e}[B]\subseteq A\cap B .\]
 
We now show Condition \ref{iewiouffwewefewfwefwef1}. Assume the contrary. Then there exists $\epsilon$ in $(0,\infty)$ such that 
$s(e)\ge \epsilon>0$ for all $e$ in $(0,\infty)$. For every integer $n$  there  exists a point $y_{n} $ in $U_{1/n}[A]\cap U_{1/n}[B]$
such that $y_{n}\not\in U_{\epsilon }[A\cap B]$.
We can assume (after replacing $y_{n}$ by $\gamma_{n}y_{n}$ for suitable elements $\gamma_{n}$ of $\Gamma$) by $\Gamma$-compactness that $y_{n}\to y$ for $n\to \infty$. Then
$y\not\in  U_{\epsilon/2 }[A\cap B]$ but
$y\in  A\cap B$ by \eqref{eqjnk092332}. This is a contradiction, and so we have verified Condition~\ref{iewiouffwewefewfwefwef1}.
\end{proof}

\begin{ddd}
 We define the functor
 \[ \cO^\infty_\homolg \colon \Gamma\Top \to \Gamma\Sp\cX \]
 as the left Kan extension
 \[\xymatrix@C=4em{
  \Gamma\Top^{cm}\ar[r]^-{\cO^\infty \circ \cM \circ \cU}\ar[d] & \Gamma\Sp\cX \\
  \Gamma\Top\ar@{.>}[ur]_{\cO^\infty_\homolg} & \\
 }\]
of $\cO^\infty \circ \cM \circ \cU$ along the fully faithful functor $\Gamma\Top^{cm} \hookrightarrow \Gamma\Top$.
\end{ddd}
  
\begin{prop}\label{groejeroigreogregregreg}
$\cO_{{\homolg}}^{\infty}$ is an equivariant $\Gamma\Sp\cX$-valued homology theory.
\end{prop}
\begin{proof}
We first note that the objectwise formula for the left Kan-extension gives
\begin{equation}\label{vervkj4io3rjoi3r3r43r43r3r}
 \cO_{{\homolg}}^{\infty}(X) \simeq  {\colim_{(Y\to X) \in \Gamma\Top^{cm}/X} \cO^{\infty}(\cM(\cU(Y)))}\ .
\end{equation}
Let $Y$ be a {$\Gamma$-compact and $\Gamma$-metrizable $\Gamma$-topological space and $(\phi,f) \colon Y \to [0,1] \times X$ be an equivariant map}.
Then  {we have the factorization
\[ (\phi,f) \colon Y \xrightarrow{(\phi,\id_{Y})}  [0,1] \times Y \xrightarrow{(\id_{ [0,1] },f)}  [0,1] \times X \ .\]
Note that $[0,1] \times Y$ is again a $\Gamma$-compact and $\Gamma$-metrizable $\Gamma$-topological space.}
This shows that the 
 category of maps of the form
 $(\id_{[0,1]},f) \colon  {[0,1] \times Y} \to [0,1]\times X$ for morphisms of $\Gamma$-topological spaces $f \colon Y\to X$  {with $Y$ in $\Gamma\Top^{cm}$} is cofinal in  {$\Gamma\Top^{cm}/ ([0,1]\times X)$.}  {By \cref{lem:nbhdvsmetricentourage}, we have an isomorphism of $\Gamma$-uniform spaces}
 \[ {\cU([0,1]\times Y) \cong [0,1]_{u} \otimes \cU(Y)\ .} \]
By the homotopy invariance of $\cO^{\infty}$ (\cref{fijofiwewefewf}) the projection
 $ {[0,1]_{u} \otimes \cU(Y) \to \cU(Y)}$ induces an equivalence
 $ {\cO^{\infty}(\cM(\cU([0,1]\times Y))) \simeq \cO^{\infty}(\cM(\cU(Y)))}$. Hence we conclude that
 $\cO^{\infty}_{\homolg}([0,1]\times X)\simeq \cO^{\infty}_{\homolg}(X)$, i.e., that $\cO^{\infty}_{\homolg}$ is homotopy invariant.
  
Assume now that $(X_{i})_{i\in I}$ is a filtered family of $\Gamma$-topological spaces and set
$X:=\colim_{i\in I}X_{i}$. Then every morphism
$Y\to X$ for a  {$Y$ in $\Gamma\Top^{cm}$} factorizes as $Y\to X_{i}\to X$ for some $i$ in $I$ since $Y$ is $\Gamma$-compact.
It follows that
\begin{align*}
\cO_{{\homolg}}^{\infty}(X) & \simeq  {\colim_{(Y\to X)\in  \Gamma\Top^{cm} / X} \cO^{\infty}(\cM(\cU(Y)))} \\
& \simeq  {\colim_{i\in I} \colim_{(Y\to X_{i}) \in  \Gamma\Top^{cm} / X_{i}}\cO^{\infty}(\cM(\cU(Y)))}\\
& \simeq  \colim_{i\in I}\cO_{{\homolg}}^{\infty}(X_{i})\ .
\end{align*}
This in particular implies the wedge axiom for $\cO_{{\homolg}}^{\infty}$.

 In order to show that $\cO_{{\homolg}}^{\infty}$ satisfies excision, by \cref{kefhwefiuziu24234324} it suffices to show the stronger result that for every decomposition $(A,B)$ of  a $\Gamma$-topological space $X$  into two closed invariant subsets we have a push-out square \begin{equation}\label{dqwdqkwudhuiu34324}
\xymatrix{
\cO_{{\homolg}}^{\infty}(A\cap B)\ar[r]\ar[d]&\cO_{{\homolg}}^{\infty}(A)\ar[d]\\
\cO_{{\homolg}}^{\infty}(B)\ar[r]&\cO_{{\homolg}}^{\infty}(X)}\ .
\end{equation}
Let  {$r\colon Y\to X$ be an object of $\Gamma\Top^{cm}/ X$}. Then
 $(r^{-1}(A),r^{-1}(B))$ is a closed decomposition of $Y$ into  {$\Gamma$-compact and $\Gamma$-metrizable subspaces}.
 Note that the objects of the form  {$r^{-1}(A)\to A$ of $\Gamma\Top^{cm} / A$}  for all $r\colon Y\to X$ are cofinal in  {$\Gamma\Top^{cm} / A$}.

We conclude that \eqref{dqwdqkwudhuiu34324} is a colimit of commuting squares 
\begin{equation}\label{fwefewf3244}
 {\xymatrix{
\cO^{\infty}(\cM(\cU(r^{-1}(A)\cap r^{-1}(B))))\ar[r]\ar[d]&\cO^{\infty}(\cM(\cU(r^{-1}(A))))\ar[d]\\
\cO^{\infty}(\cM(\cU(r^{-1}(B))))\ar[r]&\cO^{\infty}(\cM(\cU(Y)))}}
\end{equation}
Since $Y$ is a $\Gamma$-compact and $\Gamma$-metrizable $\Gamma$-topological space, by \cref{feifhweifewffewfwf} the decomposition 
 {$(r^{-1}(A),r^{-1}(B))$} is an equivariant uniform decomposition of $\cU(Y)$. Therefore the square \eqref{fwefewf3244} is a push-out square by \cref{kldjedjoiewufowe23435335}. Being a colimit of push-out squares the square \eqref{dqwdqkwudhuiu34324} is therefore also a push-out square.
\end{proof}

\begin{rem}\label{gioerjgioerg893ut9834t4}
Instead of $\cO^{\infty}\circ \cM$ we could consider any functor \[A\colon \Gamma\bU\to \bC\] for some cocomplete stable $\infty$-category {$\bC$}, which is excisive for closed decompositions and homotopy invariant. The construction above produces an equivariant $\bC$-valued homology theory $A_{{\homolg}}\colon \Gamma\Top\to \bC$ as a left Kan-extension
\[ { \xymatrix@C=4em{
\Gamma\Top^{cm}\ar[r]^-{A \circ \cU}\ar[d]&\bC\\
\Gamma\Top\ar@{..>}[ur]_-{\ A_{{\homolg}}}&
}} \]
The proof that this is indeed an equivariant $\bC$-valued homology theory is word-for-word the same as the proof of \cref{groejeroigreogregregreg}.

Note that the proof also shows that for a filtered family of $\Gamma$-topological spaces $(X_{i})_{i\in I}$ the following natural morphism is an equivalence
\begin{equation}\label{gpo4tjgopgpog54g4}
\colim_{i\in I} A_{\homolg}(X_{i})\xrightarrow{\simeq}A_{\homolg}(\colim_{i\in I}X_{i}) \ .
\end{equation}
We will use this equivalence later in this paper.
\end{rem}

A $\Gamma$-uniform space $X$ has an underlying $\Gamma$-topological space
which we will denote by $\tau(X)$. The topology  of $\tau(X)$ is generated by the sets
$V[x]$ for all points $x$ and uniform entourages $V$ of $X$. We actually get a functor
\begin{equation}
\label{eqerg45rert4re}
\tau\colon \Gamma\bU\to \Gamma\Top\ .
\end{equation}

\begin{rem}\label{fwejfkwjefiowejfoiwef}
 Let $Y$ be a $\Gamma$-compact and $\Gamma$-metrizable $\Gamma$-topological space, and let $X$ be a $\Gamma$-uniform space. Then every continuous map $Y\to \tau(X)$ induces a uniform map $\cU(Y)\to X$.
 Therefore, we obtain a natural morphism
 \begin{equation}
 \label{eq:morphism} A_{\homolg}(\tau(X)) \to A(X)\ .
 \end{equation} 
 Since the inclusion functor $\Gamma\Top^{cm} \hookrightarrow \Gamma\Top$ is fully faithful, the morphism \eqref{eq:morphism} is an equivalence for $X:=\cU(Y)$ if $Y$ is a $\Gamma$-compact and $\Gamma$-metrizable $\Gamma$-topological space.
\end{rem}

\subsection{Families of subgroups and the universal assembly map}
\label{secjks9990rfwe}

In this section we will define the classifying space $E_{\cF}\Gamma$ for a family $\cF$ of subgroups of $\Gamma$, and we will define the corresponding universal assembly map \[\alpha_{\cF}\colon \cO_{{\homolg}}^{\infty}(E_{\cF}\Gamma)\to  \cO_{{\homolg}}^{\infty}(*)\ .\] We will also introduce the motivic assembly map \[\alpha_{X,Q} \colon \cO_{{\homolg}}^{\infty}(X)\otimes \Yo^{s}(Q) \to \Sigma \Yo^{s}(Q)\] for a $\Gamma$-topological space $X$ twisted by a $\Gamma$-bornological coarse space $Q$.

\begin{ddd} A \emph{family of subgroups} $\cF$ of $\Gamma$ is a {non-empty} subset of the set of subgroups of $\Gamma$ which is closed under conjugation and taking subgroups.
\end{ddd}

\begin{ex}
Examples of families of subgroups are:
\begin{enumerate}
\item $\mathbf{\{1\}}$ -- the family containing only the trivial subgroup
\item $\Fin$ -- the family of all finite subgroups
\item $\mathbf{Vcyc}$ -- the family of all virtually cyclic subgroups
\item $\All$ -- the family of all subgroups
\end{enumerate}
In \cref{secjnk232332} we will mostly work with the family $\Fin$ since one can model the corresponding classifying space (see \cref{dddjk23wd}) by the Rips complex.
\end{ex}

For a family of subgroups $\cF$ we let $\Orb_{\cF}(\Gamma)\subseteq \Orb(\Gamma)$ be the full subcategory of transitive $\Gamma$-sets whose stabilizers belong to $\cF$. 

\begin{ex}
Note that $\Orb_{\All}(\Gamma)= \Orb(\Gamma)$.

There is furthermore an equivalence
$B\Gamma\to \Orb_{\{1\}}(\Gamma)^{op}$ which sends the unique object $*$ of $B\Gamma$ to the $\Gamma$-set $\Gamma$ (with the left action). On morphisms this equivalence sends the morphism
$\gamma$ in $\Hom_{B\Gamma}(*,*)=\Gamma$ to the  automorphism of the $\Gamma$-set $\Gamma$ given by right-multiplication with $\gamma$.
\end{ex}

For two families of subgroups $\cF,\cF^{\prime}$ satisfying  $\cF\subseteq \cF^{\prime}$ we have an inclusion \[\Orb_{\cF}(\Gamma)\hookrightarrow \Orb_{\cF^{\prime}}(\Gamma)\] of orbit categories. 
On presheaves this inclusion of orbit categories defines a restriction functor which   is the right-adjoint of  an adjunction
\begin{equation}\label{hfrjkfhiuiu34tr34t}
\Ind_{\cF}^{\cF^{\prime}}\colon\PSh(\Orb_{\cF }(\Gamma))\leftrightarrows \PSh(\Orb_{\cF^{\prime}}(\Gamma)):\Res^{\cF^{\prime}}_{\cF  }\ .
\end{equation} 

We let $*_{\cF}$ denote the final object in $\Orb_{\cF}(\Gamma)$. 
\begin{ddd}\label{dddjk23wd}
The object \[E_{\cF}\Gamma:= \Ind_{\cF}^{\All}(*_{\cF})\] of $ \PSh(\Orb(\Gamma))$
is called  the classifying space  of $\Gamma$ for the family $\cF$.
\end{ddd}

\begin{rem}\label{feiowofewfewfewfewf}
Assume that $X$ is a $\Gamma$-CW-complex with an equivalence $\ell(X)\simeq E_{\cF}\Gamma$. Then the  following \cref{ewiuhwefewfewwf} shows that for every subgroup $H$ of $\Gamma$  the fixed point set $X^{H}$ is contractible or empty depending on whether $H$ belongs to $\cF$ or not.  This property is the usual characterization of 
a classifying space of $\Gamma$ for the family $\cF$.  We say that $X$ is a model for $E_{\cF}\Gamma$. A model for $E_{\cF}\Gamma$ is unique up to contractible choice.
\end{rem}
 
Let $y\colon \Orb(\Gamma)\to \PSh(\Orb(\Gamma))$ denote the Yoneda embedding.
\begin{lem}\label{ewiuhwefewfewwf}
For  $T\in \Orb(\Gamma)$ we have
\[\Map (y(T),E_{\cF}\Gamma)\simeq
\begin{cases}
\emptyset & \text{if } T\not\in \Orb_{\cF}(\Gamma)\ ,\\
* & \text{if } T \in \Orb_{\cF}(\Gamma)\ .
\end{cases}\]
\end{lem}

\begin{proof}
For a presheaf $E$ on $ \Orb_{\cF}(\Gamma)$ and object $T$ in $\Orb(\Gamma)$ we have (we  interpret the   $\Hom$-sets as discrete spaces)
\[\Ind_{\cF}^{\All}(E)(T) \simeq  \colim_{(S\to T)\in \Orb_{\cF}(\Gamma)^{op}/T}  E(S)\ .\]
Consequently, we have
\[\Map (y(T),E_{\cF}\Gamma)\simeq \colim_{(S\to T)\in \Orb_{\cF}(\Gamma)^{op}/T} *_\cF(S)\ .\]
Note that $*_{\cF}(S)\simeq *$ and $\Orb_{\cF}(\Gamma)^{op}/T$ is empty if $T$ is not in $\Orb_\cF(\Gamma)$ and otherwise has the identity on $T$ as final object.
\end{proof}

\begin{rem}\label{ergoiergergreg}
We must consider models for $E_{\cF}\Gamma$ since $\cO^{\infty}_{\homolg}$ or the equivariant homology theory $A_{\homolg}$ constructed in \cref{gioerjgioerg893ut9834t4} are not expected to be strong equivariant homology theories {(\cref{dddjk4r})} and therefore can not be applied to the object $E_{\cF}\Gamma$ of the $\infty$-category $\PSh(\Orb(\Gamma))$. But on the other hand, since these functors are homotopy invariant, the evaluations $\cO^{\infty}_{\homolg}(E_{\cF}\Gamma^{cw})$ or $A_{\homolg}(E_{\cF}\Gamma^{cw})$ on a model $E_{\cF}\Gamma^{cw}$ for 
$E_{\cF}\Gamma$  are well-defined up to equivalence.

From now on $E_{\cF}\Gamma^{cw}$ will denote some choice of a model.
\end{rem}

The projection $E_{\cF}\Gamma^{cw}\to *$ is a morphism in  $\Gamma\Top$.

If $F\colon \Gamma\Top\to \bC$ is an equivariant $\bC$-valued homology theory, then the Farrell--Jones and Baum--Connes type question is  
for which family $\cF$ this projection
induces an equivalence 
\[F(E_{\cF}\Gamma^{cw})\to F(*)\ .\]

The coarse geometric approach to the question uses that this projection 
induces a morphism in $\Gamma \Sp\cX$
\[\alpha_{E_{\cF}\Gamma^{cw}}\colon \cO_{{\homolg}}^{\infty}(E_{\cF}\Gamma^{cw})\to  \cO_{{\homolg}}^{\infty}(*)\ .\]
\begin{ddd}\label{welfjwefjoiuio23ui23u4324243}
The morphism $\alpha_{E_{\cF}\Gamma^{cw}}$ is called the \emph{universal assembly map} for the family of subgroups $\cF$.
\end{ddd} 
 
One could ask if there is an interesting family $\cF$ (i.e., not the family of all subgroups of $\Gamma$) for which the universal assembly map is an equivalence.

Most of the study of the assembly map is based on an identification of this map with a forget-control map, or equivalently, a boundary operator of a cone sequence. We develop this point of view below.

The following diagram is one of the starting points  of \cite{blr}. We consider a $\Gamma$-topological space $X$  and the projection $X\to *$. It induces the vertical maps in the diagram in $\Gamma\Sp\cX$ whose horizontal parts are segments of the cone sequence \cref{wefoiuweiofuewwefwwefwf}:
\begin{equation*}
\mathclap{\xymatrix{
\colim\limits_{(Y\to X)\in  \Gamma\Top^{cm} / X}
\Yo^{s}(Y_{max,max})\ar[r]\ar[d] &\colim\limits_{(Y\to X)\in  \Gamma\Top^{cm} / X}
\Yo^{s}(\cO( \cM(\cU(Y))))\ar[d]  \ar[r]& \cO_{{\homolg}}^{\infty}(X)\ar[d]^{\alpha_{X}} \\ \Yo^{s}(*)\ar[r]&\Yo^{s}(\cO(*))\ar[r]&\cO_{{\homolg}}^{\infty}(*)
}}\ .
\end{equation*}
 
Note that $\Yo^{s}(\cO(*))\simeq 0$ since $\cO(*)$ is flasque. Furthermore,
 if we tensor-multiply (with respect to $- \otimes -$) the diagram with $\Yo^{s}(Q)$ for a $\Gamma$-bornological coarse space $Q$ whose underlying $\Gamma$-set is free, then by \cref{roigirhgreiughreige} the left vertical map becomes an equivalence. 
\begin{ddd}\label{gioegregregeg}
We define the \emph{obstruction motive} to be the object
\[M(X):=\colim_{(Y\to X)\in  \Gamma\Top^{cm} / X }\Yo^{s}(\cO(\cM(\cU(Y))))\]
of the category $\Gamma\Sp\cX$.
\end{ddd}

The discussion above has the following corollary.
Let $Q$ be a $\Gamma$-bornological coarse space.
\begin{kor}\label{ropgerog948ut934ut93t}
If the underlying $\Gamma$-set of   $Q $ is free, then we have a fiber sequence
\[\mathclap{
M(X)\otimes\Yo^{s}(Q)\to \cO_{{\homolg}}^{\infty}(X)\otimes\Yo^{s}(Q) \xrightarrow{\alpha_{X,Q}} \Sigma \Yo^{s}(Q) \to \Sigma (M(X)\otimes  \Yo^{s}(Q))\ .}\]
\end{kor}
\begin{ddd}\label{rjiojfiofjiejoijo2i3roi2r23ori}
The morphism $\alpha_{X,Q}:=\alpha_{X}\otimes \id_{\Yo^{s}(Q)}$ is called the \emph{motivic assembly map} with twist $Q$.
\end{ddd}

Let $Q$ be a $\Gamma$-bornological coarse space.
\begin{kor}\label{gorjerigoiregreijgioj}
Let the underlying $\Gamma$-set of $Q $ be free.
Then the motivic assembly map with $Q$-twist $\alpha_{X,Q}$ is an equivalence if and only  if $M(X)\otimes  \Yo^{s}(Q)\simeq 0$.
\end{kor}

A typical example for $Q$ is $\Gamma_{can,min}$, and variants like $\Gamma_{can,max}$, $\Gamma_{max,max}$, $\Gamma_{min,min}$, etc.

 In general we expect the assembly map $\alpha_{E_{\cF}\Gamma,Q}$ to become an equivalence only after applying suitable equivariant coarse homology theories. But the following is an example where the assembly map is an equivalence already motivically.

Let $\cF$ be a family of subgroups of $\Gamma$ and consider  a $\Gamma$-bornological coarse space $Q$.

\begin{lem}\label{flkjweoifewfewfewf}
Assume that:
\begin{enumerate}
\item \label{efojwefowfewfewf}$E_{\cF}\Gamma^{cw}$ is   $\Gamma$-compact and $\Gamma$-metrizable.
\item $Q$ is discrete as a coarse space.
\item $Q$   has stabilizers in $\cF$.
\item $Q$ is $\Gamma$-finite.
\end{enumerate}
Then we have
\[M (E_{\cF}\Gamma^{cw})\otimes \Yo^{s}(Q)\simeq 0\ .\]
\end{lem}

\begin{rem}
 Note that if the family $\cF$ is \textbf{VCyc}, i.e., the family of all virtually cyclic subgroups, then the condition of $E_{\cF}\Gamma^{cw}$ being $\Gamma$-compact is very restrictive: by a conjecture of Juan-Pineda--Leary \cite{vcyc_finite_conj} this implies that $\Gamma$ is virtually cyclic itself.

This conjecture is proven for hyperbolic groups (Juan-Pineda--Leary \cite{vcyc_finite_conj}), elementary amenable groups (Kochloukova--Martinez-Perez--Nucinkis \cite{kmpn_vcyc} and Groves--Wilson \cite{gw_vcyc}) and for one-relator groups, acylindrically hyperbolic groups, $3$-manifold groups, $\mathrm{CAT}(0)$ cube groups, linear groups (von Puttkamer--Wu \cite{puttkamer_wu_vcyc,puttkamer_wu_vcyc_linear}).
\end{rem}

\begin{proof}
Since $E_{\cF}\Gamma^{cw}$ is $\Gamma$-compact and $\Gamma$-metrizable the colimit in the \cref{gioegregregeg} of $M(E_{\cF}\Gamma)^{cw}$ stabilizes.  We consider $Q$ as a $\Gamma$-uniform bornological coarse space $Q_{disc}$ with the discrete uniform structure.
Using \cref{reopgjreogjoregregregeg}, we get the equivalence 
\[M(E_{\cF}\Gamma^{cw})\otimes \Yo^{s}(Q)\simeq {\Yo^{s}} (\cO(\cM(\cU(E_{\cF}\Gamma^{cw})))\otimes Q_{disc})\ .\]
The space $E_{\cF}\Gamma^{cw}$ has the following universal property: for every $\Gamma$-space $X$ with stabilizers in $\cF$  which is homotopy equivalent to a $CW$-complex  the space
$\Map_\Gamma(X,E_{\cF}\Gamma^{cw})$ is contractible.
It follows from the universal property of $E_{\cF}\Gamma^{cw}$ that there is a unique homotopy class of maps $\iota\colon Q \to E_{\cF}\Gamma^{cw}$. 
Furthermore, the maps
$\iota\circ \pr_{Q},\pr_{E_{\cF}\Gamma^{cw}} \colon E_{\cF}\Gamma^{cw}\times Q \to E_{\cF}\Gamma^{cw}$ are homotopic. 
Since $Q$ is $\Gamma$-finite and $E_{\cF}\Gamma^{cw}$ is $\Gamma$-compact $\Gamma$-metrizable these maps and homotopies are automatically uniform maps {when we consider $E_\cF\Gamma^{cw}$ as a uniform space by applying $\cU$.}
 Hence $\id_{\cM(\cU(E_{\cF}\Gamma^{cw}))\otimes Q_{disc}}$ is homotopic to the composition 
\[\cM(\cU(E_{\cF}\Gamma^{cw}))\otimes Q_{disc} \xrightarrow{\pr_{Q}} Q_{disc} \xrightarrow{{(\iota, \id_{Q})}} \cM(\cU(E_{\cF}\Gamma^{cw}))\otimes Q_{disc}\]
of morphisms between $\Gamma$-uniform bornological coarse spaces.
Using that $\cO$ is homotopy invariant 
we conclude that $\id_{M (E_{\cF}\Gamma^{cw})\otimes \Yo^{s}(Q)}$ factorizes over
$\Yo^{s}(\cO(Q_{disc}))$. We now use 
\cref{ropjreopg0iu09t34t34t34t3} in order to conclude that
$\Yo^{s}(\cO(Q_{disc}))\simeq 0$. This  finishes this proof.
\end{proof}

For the following we just combine \cref{flkjweoifewfewfewf} and \cref{gorjerigoiregreijgioj}.
\begin{kor}\label{kor239jksd}
If $E_{\cF}\Gamma^{cw}$ is $\Gamma$-compact {and $\Gamma$-metrizable}, then  the motivic assembly map $\alpha_{E_{\cF}\Gamma^{cw},\Gamma_{min,?}}$ is an equivalence for  $?\in \{\min,\max\}$.
\end{kor}

\subsection{Homology theories from coarse homology theories}
\label{ifjwoifwefewfewff}

Let $E$ be an equivariant $\bC$-valued coarse homology theory, i.e., a colimit preserving  functor $E\colon \Gamma\Sp\cX\to \bC$. By \cref{groejeroigreogregregreg} we can use the cone $\cO_{{\homolg}}^{\infty}$ in order to pull-back $E$ to an equivariant homology theory:
 
\begin{ddd}\label{ioefuwo345345345}
Let us define
\[E\cO^{\infty}_{\homolg} := E\circ \cO_{{\homolg}}^{\infty}\colon \Gamma\Top\to \bC\ ,\]
which is an equivariant $\bC$-valued homology theory.
\end{ddd}

Recall that for an equivariant  coarse motivic spectrum $L$  the functor
\[-\otimes L\colon \Gamma\Sp\cX\to \Gamma\Sp\cX\] preserves colimits.
If $E$ is a $\bC$-valued equivariant coarse homology theory, then we can define a new equivariant coarse homology theory
\begin{equation}\label{fwefgzu2382t78t476t476t746t764t76t}
E_{L}(-):=E(-\otimes L)\colon \Gamma\Sp\cX \to \bC\ .
\end{equation}
We will refer to $L$ as a \emph{twist}.  

The Farrell--Jones/Baum--Connes type question for $E_{\Yo^s(Q)} \cO^{\infty}_\homolg$ is now the question for which family $\cF$ of subgroups  the morphism
$E_{\Yo^s(Q)} (\alpha_{E_{\cF}\Gamma})$ is an equivalence of spectra. The next corollary is a consequence of \cref{gorjerigoiregreijgioj}.

Let $Q$ be a $\Gamma$-bornological coarse space.
\begin{kor}
Assume that the underlying $\Gamma$-set of $Q $ is free.  Then
the assembly map $E_{\Yo^{s}(Q)} (\alpha_{E_{\cF}\Gamma})$ is an equivalence if and only if
\[E_{\Yo^{s}(Q)} (M (E_{\cF}\Gamma))\simeq 0\ .\]
\end{kor}
 
Given the coarse homology theory $E$ and a twist $L$ it is  of  particular interest to characterize the homology theory
$  E_{L}\cO^{\infty}_\homolg$. As explained in \cref{wfifjweoifjwoifewfwfewf}, the restriction of  an equivariant homology theory from $\Gamma\Top$ to the full subcategory $\Gamma\mathbf{CW}$ of $\Gamma$-$CW$ complexes is determined by its restriction to the orbit category $\Orb(\Gamma)$.
So we must   understand the evaluations $E_{L}\cO_{{\homolg}}^{\infty}(S)$   for all transitive $\Gamma$-sets $S$. The main tool is   \cref{oiefjewoifo23r4243455435345}. It gives
\[E_{L}\cO_{{\homolg}}^{\infty}(S)\simeq \Sigma E_{L}({\Yo^{s}}(S_{min,max}))\ .\]
We have calculated these evaluations for equivariant coarse ordinary homology and for equivariant coarse algebraic $K$-{homology} explicitly:
\begin{enumerate}
\item see \cref{wifhuifhiwefhieuwhufewfewfwf} for $H\cX^{\Gamma}_{\Yo^{s}(\Gamma_{{can,min}})} $, and
\item see \cref{sec092323} for $ K\bA\cX^{\Gamma}_{\Yo^{s}(\Gamma_{?,?})}$.
\end{enumerate}

\section{Forget-control and assembly maps}
\label{secjnk232332}

\subsection{The forget-control map}

Let $X$ be a coarse space and $U$ an entourage of $X$.

Let $\mu\colon \cP(X)\to [0,1]$ be a probability measure on the measurable space $(X,\cP(X))$.  
\begin{ddd}
The measure $\mu$ is called finite $U$-bounded if there is a finite $U$-bounded subset $F$ of $X$ with $\mu(F)=1$. We then define the support $\supp(\mu)$ to be the smallest subset of $X$ with measure one.
	
We let $P_{U}(X)$ denote the topological space of finite $U$-bounded probability measures on $X$ equipped with the topology induced by the evaluation against finitely supported functions.
\end{ddd}
Every point $x$ in $X$ gives rise to a Dirac measure $\delta_{x}$ at $x$. If $U$ contains the diagonal of $X$, then we get a map $X \to P_U(X)$,  $x\mapsto \delta_{x}$, of sets.  A probability measure $\mu$ which is finite $U$-bounded can be
written as a finite convex combination of Dirac measures. More concretely,
\[\mu=\sum_{x\in \supp(\mu)} \mu(\{x\})\delta_{x}\ .\]
The set $P_{U}(X)$ has a natural structure of a simplicial complex. In other words, $P_{U}(X)$ is the simplicial complex with vertex set $X$ such that a subset $\{x_1,\ldots,x_m\}$ spans a simplex if and only if $(x_i,x_j)\in U$ for all $i,j$ in $\{1,\dots,m\}$. Let $P_{U}(X)_{u}$ denote the uniform space with the uniform structure induced by the canonical path quasi-metric on the simplicial complex, see \cref{rliheiogherigergregreg}. Note that the topology induced by this uniform structure is the  topology on $P_{U}(X)$  induced by the evaluation against finitely supported functions, i.e., we have
\[\tau(P_{U}(X)_{u}) = P_U(X)\ ,\]
where $\tau$ is the functor associating the underlying topological space to a uniform space, cf.~\eqref{eqerg45rert4re}. Eilenberg--Steenrod \cite[p.~75]{eilenberg_steenrod_foundations} called this the metric topology. In general this topology differs from  the weak topology on the simplicial complex in the sense of Eilenberg--Steenrod \cite[p.~75]{eilenberg_steenrod_foundations}.

If $X^{\prime}$ is a second coarse space with entourage    $U^{\prime}$   and $f\colon X\to X^{\prime}$ is a map such that $(f\times f)(U)\subseteq U^{\prime}$, then we get a  map of simplicial complexes
\[f_{*}\colon P_{U}(X)\to P_{U^{\prime}}(X^{\prime})\ ,\quad  \mu\mapsto f_{*}\mu\ ,\]
where the measure  $f_{*}\mu$ is the push-forward of the measure $\mu$.
 
If $X$ is a $\Gamma$-coarse space and the entourage $U$ is $\Gamma$-invariant, then $\Gamma$ acts on the simplicial complex $P_{U}(X)$ such that the latter becomes a $\Gamma$-simplicial complex, so in particular it becomes a $\Gamma$-topological space.  Furthermore we obtain a $\Gamma$-uniform space $P_{U}(X)_{u}$. If $X^{\prime}$ is a second $\Gamma$-coarse space, $U^{\prime}$ is a $\Gamma$-invariant entourage, and $f\colon X\to X^{\prime}$ is equivariant with $(f\times f)(U)\subseteq U^{\prime}$, then
$f_{*}\colon P_{U}(X)\to P_{U^{\prime}}(X^{\prime})$ is equivariant.

  Let $X$ be a $\Gamma$-coarse space. 
   
\begin{ddd}\label{fojwekopfwef233r4}
The $\Gamma$-topological space
\[\Rips(X):=\colim_{U\in \cC^{\Gamma}} P_{U}(X)\]
is called the \emph{Rips complex} of $X$. Note that the colimit is taken in $\Gamma\Top$.
\end{ddd}

\begin{rem}
Note that in general the simplicial complex  $P_U(X)$ is not locally finite. In this case its topology 
 does not exhibit it as a CW-complex (this happens if and only if the simplicial complex is locally finite).

But by a result of Dowker \cite[Thm.~1 on P.~575]{dowker} $P_U(X)$ has the   homotopy type of a CW-complex (see also Milnor \cite[Thm.~2]{milnor_CW} for a short proof of this). Concretely, the underlying identity map of the set induces a   homotopy equivalence $\operatorname{CW}(P_U(X)) \to P_U(X)$, where $\operatorname{CW}(P_U(X))$ denotes $P_U(X)$ retopologized as a CW-complex. 
\end{rem}

Let $\Gamma_{can}$ be the  group $\Gamma$ considered as a $\Gamma$-coarse space with its canonical coarse structure. Recall the localization \eqref{oiiuhf4u9824f23ff}.
\begin{lem}\label{wfoi8765ewfewf}
We have an equivalence $\ell(\Rips(\Gamma_{can}))\simeq  E_{\Fin}\Gamma$. 
\end{lem}
\begin{proof}

By definition, $\ell(\Rips(\Gamma))$ is the presheaf on $\Orb(\Gamma)$ given by
\[S\mapsto \kappa(\Map_{\Gamma\Top}(S,\Rips(\Gamma_{can})))\ ,\]  where $\kappa$ is as in \eqref{fjiiioioio}. We now observe that the stabilizers of the points of $\Rips(\Gamma_{can})$ are finite subgroups of $\Gamma$.
Consequently, the presheaf belongs to the subcategory of presheaves supported 
on the orbits with finite stabilizers.
Since, by definition,
$E_{\Fin}\Gamma$ corresponds to the final object in this subcategory there exists a morphism
$\ell(\Rips(\Gamma_{can}))\to E_{\Fin}\Gamma$.
In order to show that it is an equivalence it suffices to show that the 
spaces $\Map_{\Gamma\Top}(\Gamma/H,\Rips(X))$ for all finite subgroups $H$ of $\Gamma$ are equivalent to $*$, i.e., that these spaces are connected and that their homotopy groups are trivial.

We now study these homotopy groups.
Since the spheres $S^{n}$ are compact for all $n$ in $\nat$ and the structure maps
\[\Map_{\Gamma\Top}(\Gamma/H,P_{U}(\Gamma_{can}))\to \Map_{\Gamma\Top}(\Gamma/H,P_{U^{\prime}}(\Gamma_{can}))\] for $U\subseteq U^{\prime}$
are inclusions of CW-complexes,
 we have 
\[\pi_{*}(\Map_{\Gamma\Top}(\Gamma/H,\Rips(\Gamma_{can})))\cong \colim_{U\in \cC^{\Gamma}}\pi_{*}(\Map_{\Gamma\Top}(\Gamma/H,P_{U}(\Gamma_{can})))\ .\]
We have a homeomorphism  
\[\Map_{\Gamma\Top}(\Gamma/H,P_{U}(\Gamma_{can}))\cong P_{U}(\Gamma_{can})^{H}\ .\]
We fix an integer $n$ and consider a map
\[f\colon S^{n}\to P_{U}(\Gamma_{can})^{H}\ .\]
  The image $f(S^{n})$ is  a compact subset of $P_{U}(\Gamma_{can})^{H}$.  
In the case $n=0$, since $\Gamma_{can}$ is coarsely connected, we can increase the  entourage $U $ such that the image $f(S^{0})$ belongs to a connected component of $P_{U}(\Gamma_{can})$.
  Let now $n$ be arbitrary. We can now assume that   $ f(S^{n})$  belongs to a connected component of $P_{U}(\Gamma_{can})$.
It is hence bounded in diameter by some integer $N$. We conclude that under the map $P_{U}(\Gamma_{can})\to P_{U^{N}}(\Gamma_{can})$ the image 
$f(S^{n})$  is mapped to a subset of a single simplex of $P_{U^{N}}(\Gamma_{can})$. 
The intersection of the $H$-fixed points with this simplex 
is a convex subset and hence itself contractible. We conclude that $f$ is homotopic to a constant map.   
\end{proof}

\begin{ddd}   
\label{243trfd34}
If $X$ is a $\Gamma$-bornological coarse space and $U$ an invariant entourage of $X$, then we equip the $\Gamma$-simplicial complex $P_{U}(X)$ with the bornology generated by the subsets $P_{U}(B)$ for all bounded subsets $B$ of $X$. Equipped with this bornology and the coarse structure induced by the metric we obtain a $\Gamma$-bornological coarse space which we denote by $P_{U}(X)_{bd}$. We furthermore write $P_{U}(X)_{bdu}$ for the corresponding $\Gamma$-uniform bornological coarse space. Note in contrast that the notation $P_{U}(X)_{d}$ (or  $P_{U}(X)_{du}$, respectively) would mean the $\Gamma$-bornological coarse space (or  $\Gamma$-uniform bornological coarse space, respectively)  whose  bornological and coarse (and uniform, respectively) structures are induced from the metric, see \cref{ifjwoifwoefwefewfewfw}.
\end{ddd}

Let $X$ be a $\Gamma$-bornological coarse space and $U$ be an invariant entourage of $X$. The Dirac measures provide a morphism of $\Gamma$-bornological coarse spaces
\begin{equation}\label{fwkfhik23r234r23r23r2r}
\delta\colon X_{U}\to P_{U}(X)_{bd}\ .
\end{equation}

\begin{lem}\label{weflkjweiofewfewfewf}
Assume:
\begin{enumerate}
\item $\Gamma$ is torsion-free.
\item The underlying $\Gamma$-set of $X$ is free.
\end{enumerate}
Then the morphism \eqref{fwkfhik23r234r23r23r2r} is an equivalence of $\Gamma$-bornological coarse spaces.
\end{lem}

\begin{proof}
We first observe that $\Gamma$ acts freely on $P_{U}(X)$. Indeed, for $\gamma$ in $\Gamma$ and $\mu$ in $P_{U}(X)$  satisfying  $\gamma \mu=\mu$  the subgroup of $\Gamma$ generated by $\gamma$ has a finite orbit contained in $\supp(\mu)$. Since $\Gamma$ is torsion-free and $X$ is a free $\Gamma$-set this can only happen if $\gamma=1$.

To define an inverse morphism $g\colon P_{U}(X)_{bd}\to X_{U}$ we first choose representatives for the orbits $P_{U}(X)/\Gamma$. Then we define $g(\mu)$ for every chosen representative $\mu$  to be a point in $\supp(\mu)$, and extend equivariantly.  Then $g\circ \delta=\id_{X_{U}}$ and $\delta\circ g$ is close to $\id_{P_{U}(X) }$.
\end{proof}

The following corollary follows immediately from the above lemma.
\begin{kor}
If $\Gamma$ is a finitely generated torsion-free group, then $\Gamma_{can,min}\to P_{U}(\Gamma)_{bd}$ is an equivalence for every invariant generating entourage $U$ of $\Gamma$.
\end{kor}

Assume that $Q$ is a $\Gamma$-bornological coarse space. Let $X$ be a $\Gamma$-bornological coarse space and $U$ be an invariant entourage of $X$.
\begin{lem}\label{iofiowefhewiofewfewfew}
If the underlying $\Gamma$-set of $Q$ is free, then
\[\delta\times \id_{Q}\colon X_{U}\otimes Q\to P_{U}(X)_{bd}\otimes Q\]     is an equivalence of $\Gamma$-bornological coarse spaces. 
\end{lem}
Note that we do not assume that $X$ is a free $\Gamma$-set.
\begin{proof}
First we note that $\Gamma$ acts freely on the set $P_{U}(X) \times Q$. We choose representatives for the orbits $(P_{U}(X)\times Q)/\Gamma$. Then we choose  for every
representative $(\mu,q) \in P_{U}(X)\times Q$ a point $x\in \supp(\mu)$ and set   $g(\mu,q):=(x,q)$. Then we extend  this to an equivariant  map  $g\colon P_{U}(X)\times Q\to X\times Q$.  This map of sets is a morphism 
\[g\colon P_{U}(X)_{bd}\otimes Q\to X_{U}\otimes Q\ .\]
By construction
$g\circ (\delta\times \id_{Q})=\id_{X\times Q}$ and $g\circ (\delta\times \id_{Q})$ is close to the identity.
\end{proof}

 For the following recall the cone functor $\cO$  from \cref{seccone} and the ``cone at infinity'' functor $\cO^\infty $ from \cref{secjknds2sd23}.

Let $X$ be a $\Gamma$-bornological coarse space.
\begin{ddd}\label{gkirjoiergegergergerg}
We define the following equivariant coarse motivic spectra:
\begin{align*}
F(X) & := \colim_{U\in \cC^{\Gamma}}\Yo^{s}(\cO(P_{U}(X)_{bdu}))\ ,\\
F^{\infty}(X) & := \colim_{U\in \cC^{\Gamma}} \cO^{\infty} (P_{U}(X)_{bdu})\ ,\\
F^{0}(X) & :=  \colim_{U\in \cC^{\Gamma}} \Yo^{s}(P_{U}(X)_{bd})\ .
\end{align*}\end{ddd}
Using standard Kan-extension techniques one can refine the above description to functors
\[F,F^{\infty},F^{0}\colon \Gamma\BC\to \Gamma\Sp\cX\ , \] see Remark \ref{ropewfwfwefwef} for details.
 The fiber sequence from \cref{wefoiuweiofuewwefwwefwf} provides a natural fiber sequence of functors 
\begin{equation}\label{vwlkvewlkvwevewvewve}
F^{0}(X)  \to F(X) \to F^{\infty} (X) \xrightarrow{\beta_{X}} \Sigma F^{0}(X)\ .
\end{equation}

\begin{ddd}\label{rlkjflewf89233iurz2i3urkjf2ff}
We call $\beta_{X}$ the \emph{forget-control map}.
\end{ddd}

\begin{rem}\label{ropewfwfwefwef}
We let $\Gamma\BC^{\cC}$ denote the category of pairs $(X,U)$, where $X$ is a $\Gamma$-bornological coarse space and
$U$ is an invariant entourage of $X$ containing the diagonal. A morphism $(X,U)\to (X^{\prime},U^{\prime})$ is a morphism $f\colon X\to X^{\prime}$ in $\Gamma\BC$ such that $(f\times f)(U)\subseteq U^{\prime}$.
We have a forgetful functor
\begin{equation}\label{vreoi34fg3vefv}
\Gamma\BC^{\cC}\to \Gamma\BC\ ,\quad (X,U)\mapsto X\ .
\end{equation}

Let \[\tilde E\colon 
\Gamma\BC^{\cC}\to \bC\] be a functor to some cocomplete target $\bC$ and
let $E$ be the left Kan extension of $\tilde E$ along \eqref{vreoi34fg3vefv}.  The evaluation of $E$ on 
  a $\Gamma$-bornological coarse space $X$ is then given as follows:
\begin{lem} We have an equivalence
\[
E(X)\simeq \colim_{U\in \cC^{\Gamma}(X)} \tilde E(X,U)\ .
\]
\end{lem}
\begin{proof}
By the pointwise formula for the left Kan extension we have an  equivalence
\[E(X)\simeq \colim_{((X^{\prime},U^{\prime}),f\colon X^{\prime}\to X) \in \Gamma\BC^{\cC}/X} \tilde E(X^{\prime},U^{\prime})\ .\]
 If $((X^{\prime},U^{\prime}),f\colon X^{\prime}\to X)$ belongs to $\Gamma\BC^{\cC}/X$, then we have a morphism \[(X^{\prime},U^{\prime})\to (X,f(U^{\prime})\cup \diag(X))\] in $\Gamma\BC^{\cC}/X$. 
This easily implies that the full subcategory of objects of the form $((X,U),\id_{X})$ of $\Gamma\BC^{\cC}/X$ with $U$ in $\cC^{\Gamma}(X)$ is cofinal in $\Gamma\BC^{\cC}/X$.
 \end{proof}
 We have a functor \[P\colon \Gamma\BC^{\cC}\to \Gamma\UBC\ , \quad (X,U)\mapsto P_{U}(X)_{bdu}\ .\]
 We construct the fibre sequence \eqref{vwlkvewlkvwevewvewve}  by  applying the left Kan extension to the fibre sequence of functors
 $\Gamma\BC^{\cC}\to \Gamma\Sp\cX$
 \[\Yo^{s}\circ \cF_{\cT}\circ P\to \Yo^{s}\circ \cO\circ  P\to \cO^{\infty}\circ P\to \Sigma \Yo^{s}\circ \cF_{\cT}\circ P\]
 obtained by  precomposing  the sequence from \cref{wefoiuweiofuewwefwwefwf} with $P$.
 \end{rem}

Let $X$ be a $\Gamma$-bornological space. In the following two corollaries  we  identify the  $\Gamma$-coarse motivic spectrum $F^0(X)$.

\begin{kor}\label{wefewfefewfwfwfw1}
If $\Gamma$ is torsion-free and  the underlying $\Gamma$-set of $X$ is  free, then
\[F^{0}(X)\simeq  \Yo^{s}(X)\ .\]
\end{kor}

\begin{proof}
We have equivalences
\[F^{0}(X)= \colim_{U\in \cC^{\Gamma}} \Yo^{s}(P_{U}(X)_{bd})\stackrel{\cref{weflkjweiofewfewfewf}}{\simeq} \colim_{U\in \cC^{\Gamma}} \Yo^{s}( X_{U})\stackrel{\cref{kjeflwfjewofewuf98ewuf98u798798234234324324343}.\ref{efwijfiewfoiefe3u40934332r}}{\simeq}\Yo^{s}(X)\ ,\]
which proves the claim.
\end{proof}

Let $Q$ and $X$ be $\Gamma$-bornological coarse spaces.
\begin{kor}\label{wefewfefewfwfwfw}
If the underlying $\Gamma$-set of  $Q$   is free, then
\[F^{0}(X)\otimes \Yo^{s}(Q)\simeq  \Yo^{s}(X)\otimes \Yo^{s}(Q)\ .\]
\end{kor}

\begin{proof}
Here we use \cref{iofiowefhewiofewfewfew}, that $\Yo^{s}$ is symmetric monoidal, and that the functor
\[-\otimes \Yo^{s}(Q)\colon \Gamma\Sp\cX\to \Gamma\Sp\cX\] preserves colimits. Therefore we can write down an analogous sequence of equivalences as in the above proof of \cref{wefewfefewfwfwfw1}.
\end{proof}

\subsection{Comparison of the assembly and the forget-control map}\label{reoirgioerigueog3453455}

In this section we will compare the assembly map for the family of finite subgroups with the forget-control map. 

For every two $\Gamma$-bornological coarse spaces $X$ and $L$ we have the forget-control morphism (\cref{rlkjflewf89233iurz2i3urkjf2ff})
\[\beta_{X,L}\colon \colim_{U\in\cC^\Gamma}\cO^\infty(P_U(X)_{bdu})\otimes \Yo^s(L)\to \colim_{U\in\cC^\Gamma}\Sigma \Yo^s(P_U(X)_{bd}\otimes L)\ .\]
We have furthermore the assembly map (see \cref{rjiojfiofjiejoijo2i3roi2r23ori})
\[\alpha_{\Rips(X),L}\colon \cO^\infty_\homolg(\Rips (X))\otimes \Yo^s(L)\to  \cO^\infty_\homolg(*)\otimes \Yo^s(L)\simeq\Sigma \Yo^s(L)\]
induced by the morphism $\Rips(X)\to *$ of $\Gamma$-topological spaces.

{Recall \cite[Def.~2.28]{buen} that a coarse space $(X,\cC)$ is called \emph{coarsely connected} if for any two points $x,y$ in $X$ there exists an entourage $U$ in $\cC$ such that $(x,y) \in U$.}

\begin{ddd}
A $\Gamma$-bornological coarse space $X$ is \emph{eventually coarsely connected} if there exists a coarse entourage $U$ such that $X_U$ is coarsely connected.

Note that an eventually coarsely connected space is in particular coarsely connected.
\end{ddd}

While one is interested in $\alpha_{\Rips(X),\Gamma_{can,min}}$, using descent methods (like in \cite{desc}) we will only be able to derive split-injectivity of $\beta_{X,\Gamma_{max,max}}$ in several cases. The following theorem allows us to compare both maps. This comparison does not hold directly but only after forcing continuity.

Let $X$ be a $\Gamma$-bornological coarse space. Recall \cite[Def.~6.100]{buen} that $X$ has \emph{strongly bounded geometry} if it is equipped with the minimal compatible bornology and if for every entourage $U$ of $X$ there exists a uniform finite upper bound on the cardinalities of $U$-bounded subsets of~$X$. Furthermore recall the functor $C^{s}$ from \eqref{oihiuhfiuwheiuewf23}.

\begin{theorem}\label{oreorpgergregeg}
	Assume:
	\begin{enumerate}
		\item\label{54tref3f} $X$ has strongly bounded geometry.
		\item $X$ is $\Gamma$-finite.
		\item The action of $\Gamma$ on $X$ is proper (\cref{fewkfjwklefjewlkfewfewe}). 
		\item $X$ is eventually coarsely connected.
	\end{enumerate}
	Then the morphisms $C^s(\alpha_{\Rips(X),\Gamma_{can,min}})$ and $C^s(\beta_{X,\Gamma_{max,max}})$ are equivalent. 
\end{theorem}

\begin{rem}
Note that $X$ being $\Gamma$-finite and the action of $\Gamma$ on $X$ being proper implies that $X$ has the minimal bornology.

Furthermore, if $X$ is $\Gamma$-finite and $U$ is a $\Gamma$-invariant entourage of $X$, then the assumption that every $U$-bounded subset is finite, already implies a uniform upper bound on the cardinality of $U$-bounded subsets.
Hence Assumption~\ref{54tref3f} in \cref{oreorpgergregeg} above could be equivalently replaced by the seemingly weaker assumption  that every $U$-bounded subset is finite. \end{rem}

The rest of this section is devoted to the proof of \cref{oreorpgergregeg}.

We consider a $\Gamma$-simplicial complex $K$. Recall the notation introduced in \cref{ifjwoifwoefwefewfewfw}:
\begin{enumerate}
\item $K_{u}$ denotes the $\Gamma$-uniform space associated to $K$.
\item $K_{u,max,max}$ denotes the $\Gamma$-uniform bornological coarse space which has the uniform structure of $K_{u}$, but the maximal coarse and bornological structures.
\item $K_{u,d,max}$ denotes the $\Gamma$-uniform bornological coarse space with the uniform structure of $K_{u}$, the metric coarse structure and the maximal bornological structure.
\item $K_{d,max}$ denotes the $\Gamma$-bornological coarse space underlying $K_{u,d,max}$.
\end{enumerate}

We denote by
\[\beta_{X,L}^{max,max}\colon \colim_{U\in\cC^\Gamma}\cO^\infty(P_U(X)_{u,max,max})\otimes \Yo^s(L)\to \colim_{U\in\cC^\Gamma}\Sigma \Yo^s(P_U(X)_{max,max}\otimes L)\]
and
\[\beta_{X,L}^{d,max}\colon \colim_{U\in\cC^\Gamma}\cO^\infty(P_U(X)_{u,d,max})\otimes \Yo^s(L)\to \colim_{U\in\cC^\Gamma}\Sigma \Yo^s(P_U(X)_{d,max}\otimes L)\]
the forget-control maps.

The proof of \cref{oreorpgergregeg} consists of a sequence of lemmas.

Let $X$ and $L$ be $\Gamma$-bornological coarse spaces.
\begin{lem}
	\label{lem:comparison1}
	Assume:
	\begin{enumerate}
		\item $X$ has strongly bounded geometry.
		\item $X$ is $\Gamma$-finite.
		\item The underlying set of $L$ is a free $\Gamma$-set.
	\end{enumerate}
	Then $\alpha_{\Rips(X),L}$ and $\beta^{max,max}_{X,L}$ are equivalent.
\end{lem}
\begin{proof}
	The assumptions on $X$ imply that $P_{U}(X)$ is $\Gamma$-compact and locally finite, hence $\Gamma$-metrizable, for every invariant coarse entourage $U$ of $X$. Therefore, by \eqref{eq:morphism}, we have an equivalence 
	\[\cO^\infty_\homolg(P_U(X))\simeq \cO^\infty(P_U(X)_{u,max,max})\ .\]
	Using \cref{fojwekopfwef233r4} of the Rips complex as a filtered colimit of $\Gamma$-topological spaces and the relation \eqref{gpo4tjgopgpog54g4}  we get
	\begin{align*}
	\cO^\infty_\homolg(\Rips(X)) & \simeq \colim_{U\in \cC^{\Gamma}}\cO^\infty_\homolg(P_U(X))\\
	& \simeq \colim_{U\in \cC^{\Gamma}}\cO^\infty(P_U(X)_{u,max,max}) \ .
	\end{align*}
	Hence we get the following commutative diagram:
	\[\xymatrix@C=4em{
		\colim_{U\in \cC^{\Gamma}}\cO^\infty(P_U(X)_{u,max,max})\otimes \Yo^s(L) \ar[r]^-{\beta^{max,max}_{X,L}}\ar[d]^\simeq& \colim_{U\in\cC^\Gamma}\Sigma\Yo^s(P_U(X)_{max,max}\otimes L)\ar[dd]^\simeq\\
		\cO^\infty_\homolg(\Rips(X))\otimes \Yo^s(L)\ar[d]^{\alpha_{\Rips(X),L}}&\\
		\cO^\infty_\homolg(*)\otimes\Yo^s(L)\ar[r]^\simeq&\Sigma\Yo^s(L)		
	}\]
	The vertical arrow on the right is an equivalence by \cref{roigirhgreiughreige}.
\end{proof}

{Recall the functor $C^s$ from \eqref{oihiuhfiuwheiuewf23}.} Let $X$ be a $\Gamma$-bornological coarse space.
\begin{lem}
	\label{lem:comparison2}
	Assume:
	\begin{enumerate}
		\item $X$ has strongly bounded geometry.
		\item $X$ is eventually coarsely connected.
	\end{enumerate}
	Then $C^s(\beta_{X,\Gamma_{can,min}}^{max,max})$ and $C^s(\beta_{X,\Gamma_{can,min}}^{d,max})$ are equivalent.
\end{lem}
\begin{proof}
	The morphism $P_U(X)_{u,d,max}\to P_U(X)_{u,max,max}$ of $\Gamma$-uniform bornological coarse spaces induces a diagram
	\[\xymatrix@C=4em{
		C^s(\cO^\infty(P_U(X)_{u,d,max})\otimes\Yo^s(\Gamma_{can,min}))\ar[r]^-{\beta_{X,\Gamma_{can,min}}^{d,max}}\ar[d]&C^s(\Sigma\Yo^s(P_U(X)_{d,max}\otimes \Gamma_{can,min}))\ar[d]\\
		C^s(\cO^\infty(P_U(X)_{u,max,max})\otimes\Yo^s(\Gamma_{can,min}))\ar[r]^-{\beta_{X,\Gamma_{can,min}}^{max,max}}&C^s(\Sigma\Yo^s(P_U(X)_{max,max}\otimes \Gamma_{can,min}))
	}\]
	Since $P_U(X)_{u,d,max}\to P_U(X)_{u,max,max}$ is a coarsening, the left vertical arrow is an equivalence by \cref{fiiwofwufw9e8fuew9fwefwef}. To show that the right vertical arrow is an equivalence for large entourages $U$ we will need continuity.
	
	For an invariant entourage $U$ of $X$ we denote the set of finite subcomplexes of $P_U(X)$ by $\cF(P_U(X))$. It is a filtered partially ordered set with respect to the inclusion relation.
	For every finite subcomplex $F$ we define a $\Gamma$-invariant subcomplex $D_{F}:=\Gamma(F\times \{1\})$ of $P_U(X) \times \Gamma$.
	We consider the family of $\Gamma$-invariant subsets
	\[\cD:=(D_{F})_{F\in \cF(P_U(X))}\]
	of $P_U(X)\times \Gamma$. By \cref{fwoiejoweijfoiwefewfewfew} the family $\cD$ is a co-$\Gamma$-bounded exhaustion of both spaces $P_U(X)_{max,max}\otimes \Gamma_{can,min}$ and $P_U(X)_{d,max}\otimes \Gamma_{can,min}$.
	By continuity, it suffices to show that the bornological coarse structures on $\cD$ induced from $P_U(X)_{max,max}\otimes \Gamma_{can,min}$ and from $P_U(X)_{d,max}\otimes \Gamma_{can,min}$, respectively, agree.
	
	Since the bornologies of $P_U(X)_{max,max}\otimes \Gamma_{can,min}$ and $P_U(X)_{d,max}\otimes \Gamma_{can,min}$ agree, we only have to care about the coarse structures. Let $U$ be large enough, such that $P_U(X)$ is connected.
	
	The coarse structure on $D_F$ induced by $P_U(X)_{d,max}\otimes \Gamma_{can,min}$ is generated by the entourages
	\[(D_{F}\times D_{F})\cap (U_{r}\times V_{B})\ ,\] where $U_{r}$ is a metric entourage of $P_U(X)$ of size $r$ in $(0,\infty)$ and $V_{B}:=\Gamma(B\times B)$ for a finite subset $B$ is one of the generating entourages of the canonical structure of $\Gamma$.
	
	The coarse structure on $D_F$ induced by $P_U(X)_{max,max}\otimes \Gamma_{can,min}$ is generated by the entourages
	\[(D_{F}\times D_{F})\cap ((P_U(X)\times P_U(X))\times V_{B})\] for finite subsets $B$  of $\Gamma$. 
	It is clear that the coarse structure of the latter is larger than the one of the first, and it remains to show to other inclusion.
	
	We have 
	\[(D_{F}\times D_{F})\cap ((P_U(X)\times P_U(X))\times V_{B})\cong \bigcup_{(\gamma,\gamma^{\prime})\in V_B} (\gamma F\times \gamma^{\prime}F)\times \{(\gamma,\gamma^{\prime})\}\ .\]
	Since $F$ is a finite subcomplex and $P_U(X)$ is connected, there exists an $r $ in $(0,\infty)$ such that
	$BF\times BF\subseteq U_{r}$. 
	This implies that
	\[(\gamma F \times \gamma^{\prime}F)\times \{(\gamma,\gamma^{\prime})\}\subseteq U_{r}\times V_{B}\] for all  pairs $(\gamma,\gamma^{\prime})$ in $V_B$. We conclude that
	\[(D_{F}\times D_{F})\cap ((P_U(X)\times P_U(X))\times V_{B})\subseteq (D_{F}\times D_{F})\cap(U_{r}\times V_{B})\ .\]
	This finishes the proof.		
\end{proof}

Let $X$ be a $\Gamma$-bornological coarse space.
\begin{lem}
	\label{lem:comparison3}
	Assume:
	\begin{enumerate}
		\item $X$ is $\Gamma$-finite.
		\item $X$ has strongly bounded geometry.
		\item the $\Gamma$-action on $X$ is proper.
	\end{enumerate}
	Then the maps $C^s(\beta_{X,\Gamma_{can,min}}^{d,max})$ and $C^s(\beta_{X,\Gamma_{can,max}})$ are equivalent.
\end{lem}

\begin{proof}
{Recall from \cref{243trfd34} that $P_{U}(X)_{bd}$ denotes the $\Gamma$-bornological coarse space whose coarse structure is induced from the metric and whose bornology is generated by the subsets $P_{U}(B)$ for all bounded subsets $B$ of $X$. Recall furthermore that we write $P_{U}(X)_{bdu}$ for the corresponding $\Gamma$-uniform bornological coarse space.}

{Recall from \eqref{eq3t4rger435} that $\Yo_{c}^{s} \simeq C^s \circ \Yo^{s}$.} It suffices to produce diagrams
	\[	\xymatrix{
		\Yo_{c}^{s}([0,k]\otimes  P_U(X)_{bd}\otimes \Gamma_{can,max})\ar[r]\ar[d]^{\simeq} & C^s(\cO(P_U(X)_{bdu})\otimes \Yo^s(\Gamma_{can,max}))\ar[d]^{\simeq}\\
		\Yo_{c}^{s}([0,k]\otimes  P_U(X)_{d,max}\otimes \Gamma_{can,min})\ar[r]&C^s(\cO(P_U(X)_{u,d,max})\otimes \Yo^s(\Gamma_{can,min}))
	}\]
	for every natural number $k$ which are compatible with increasing $k$ and $U$. 
	
	To produce the diagram we will use continuity with the exhaustion
	$\cY=(Y_{\kappa})_{\kappa\in \cF(P_U(X))^{\nat}}$ from \cref{fkwehfueuiewiuu423}, where $\cF(P_U(X))$ denotes the set of all finite subcomplexes of $P_U(X)$.
	Recall that for $\kappa$ in  $\cF(P_U(X))^{\nat}$ we set
	\begin{equation}
	Y_{\kappa}:=\bigcup_{n\in \nat}[n-1,n]\times D_{\kappa(n)}\ .
	\end{equation} 
	Since $P_U(X)_{bd}$ and $P_U(X)_{d,max}$ are $\Gamma$-bounded the exhaustion is trapping by \cref{fkwehfueuiewiuu423} for both spaces
	\[\cO(P_U(X)_{u,d,max }) \otimes \Yo^s(\Gamma_{can,min}) \quad \text{and} \quad \cO(P_U(X)_{bdu})  \otimes \Yo^s(\Gamma_{can,max})\ .\]
	Note that the hybrid coarse structure does not play a role here since trapping exhaustions are a bornological concept.
	
	Since the definition of the exhaustion is independent of $k$ and compatible with increasing $U$, it remains for us to show that the bornological coarse structures on $Y_\kappa$ induced from $\cO(P_U(X)_{bdu})\otimes \Yo^s(\Gamma_{can,max})$ and $\cO(P_U(X)_{u,d,max})\otimes \Yo^s(\Gamma_{can,min})$, respectively, agree {in order to obtain (by continuity) the} equivalences in the above diagram. Since the coarse structures of $\cO(P_U(X)_{bdu})\otimes \Yo^s(\Gamma_{can,max})$ and $\cO(P_U(X)_{u,d,max})\otimes \Yo^s(\Gamma_{can,min})$ agree, we only have to consider the bornologies.
	
	Every bounded subset of 
	$\cO(P_U(X)_{u,d,max })\otimes \Yo^s(\Gamma_{can,min})$ or $\cO(P_U(X)_{bdu})\otimes \Yo^s(\Gamma_{can,max})$ is contained in $[0,n]\times P_U(X)\times \Gamma$ for some $n$. It therefore suffices to see that the induced bornologies on $([0,n]\times P_U(X)\times \Gamma)\cap Y_{\kappa}$ coincide. We can now  further finitely decompose
	\[([0,n]\times P_U(X)\times \Gamma)\cap Y_{\kappa}\subseteq\bigcup_{i=1}^{n+1}[i-1,i]\times D_{\kappa(i)}\ .\]
	It suffices to show that the induced bornologies on $[i-1,i]\times D_{\kappa(i)}$ coincide. For this we have to show that for every $F\in \cF(P_U(X))$ the bornologies on $D_F$ induced from $P_U(X)_{bd}$ and $P_U(X)_{d,max}$, respectively, agree.  
	
	Since $X$ is $\Gamma$-finite and the $\Gamma$-action on $X$ is proper, $X$ carries the minimal bornology. Consequently, every bounded subset of $P_U(X)_{bd}$ is contained in a finite subcomplex. Hence, the bornology on $D_F$ induced by $P_U(X)_{bd}$ is generated by the sets $D_{F}\cap(F^{\prime}\times \Gamma)$ for all $F^{\prime}$ in $\cF({P_U(X)})$. 
	This set is equal to
	\[\big(\bigcup_{\gamma\in \Gamma} \gamma F \times \{\gamma\}\big)\cap (F^{\prime}\times \Gamma)=\bigcup_{\{\gamma\in \Gamma\:|\: \gamma F\cap F^{\prime}\not=\emptyset\}} (\gamma F\cap F^{\prime})\times \{\gamma\}\ . \]
	Note that the index set of the union on the right hand side is finite since the $\Gamma$-action is proper.
	
	The bornology induced by $P_U(X)_{d,max}$ is generated by the sets $D_{F}\cap ({P_U(X)}\times B)$ for all finite subsets $B$ of $\Gamma$. This set can be written in the form
	\[\bigcup_{\gamma\in B} \gamma F\times \{\gamma\}\ .\]
	
	The families of subsets
	\[\Big(\bigcup_{\{\gamma\in \Gamma\mid\gamma F\cap F^{\prime}\not=\emptyset\}} (\gamma F\cap F^{\prime})\times \{\gamma\}\Big)_{F^{\prime}\in \cF({P_U(X)})}\quad \text{and} \quad \Big(\bigcup_{\gamma\in B} \gamma F\times \{\gamma\}\Big)_{B\subseteq \Gamma, |B|<\infty}\] generate the same bornologies. This finishes the proof of the lemma.
\end{proof}

Let $X$ be a $\Gamma$-bornological coarse space.
\begin{lem}
	\label{lem:comparison4}
	Assume:
	\begin{enumerate}
		\item $X$ has strongly bounded geometry.
		\item $X$ is $\Gamma$-finite.
	\end{enumerate}
	Then the morphisms
	$C^s(\beta_{X,\Gamma_{can,max}})$ and $C^s(\beta_{X,\Gamma_{max,max}})$ are equivalent.
\end{lem}
\begin{proof}
	The morphism of $\Gamma$-bornological coarse spaces $\Gamma_{can,max}\to \Gamma_{max,max}$ induces the commutative diagram
	\begin{equation}
	\label{eq24354re3}\xymatrix{
	C^s(\cO^\infty(P_U(X)_{bdu})\otimes \Yo^s(\Gamma_{can,max}))\ar[r]\ar[d]& \Sigma \Yo_c^s(P_U(X)_{bd}\otimes \Gamma_{can,max})\ar[d]\\
		C^s(\cO^\infty(P_U(X)_{bdu})\otimes \Yo^s(\Gamma_{max,max}))\ar[r]& \Sigma \Yo_c^s(P_U(X)_{bd}\otimes \Gamma_{max,max})			
	}
	\end{equation}
	The functor $\cO^{\infty}$ from $\Gamma\UBC$ to $\Sp\cX$ is
	homotopy invariant and excisive for equivariant uniform decompositions. Since $X$ is $\Gamma$-finite and has strongly bounded geometry, for every invariant entourage $U$ the complex
	$P_{U}(X)$ is a $\Gamma$-finite simplicial complex.
	Using excision and homotopy invariance we conclude that the left vertical map in the above diagram is an equivalence if 
	\[C^s(\cO^\infty(S)\otimes \Yo^s(\Gamma_{can,max}))\to C^s(\cO^\infty(S)\otimes \Yo^s(\Gamma_{max,max}))\]
	is an equivalence for every $\Gamma$-uniform bornological coarse space $S$ which is a transitive $\Gamma$-set that has the minimal bornology and the discrete uniform structure.
	Note that in this case
	\[\cO^{\infty}(S)\simeq \cO^\infty(S_{disc,min,min})\simeq \Sigma \Yo^{s}(S_{min,min})\ ,\]
	since $S_{disc,min,min}\to S$ is a coarsening and $\cO(S_{disc,min,min})$ is flasque.
	
	If we can show that for every $\Gamma$-bounded $\Gamma$-bornological coarse space $X$ the map
	\[\Yo_c^s(X\otimes\Gamma_{can,max})\to \Yo_c^s(X\otimes\Gamma_{max,max})\]
	is an equivalence, then we can conclude that the right vertical map in \eqref{eq24354re3}  is an equivalence, and by the above argument  the left vertical map is an equivalence, too. The lemma then follows by taking the colimit over all invariant entourages $U$ of $X$.

	By continuity, it suffices to show that if $F$ is a locally finite, invariant subset of $X\times \Gamma_{?,max}$ {(the coarse structure on $\Gamma$ does not matter since local finiteness is a bornological concept)}, then the bornological coarse structures on $F$ induced by $X\times \Gamma_{can,max}$ and $X\times \Gamma_{max,max}$, respectively, agree. Since the bornologies are the same, we only have to care about the coarse structures.
	
	{Every entourage of $F_{\Gamma_{can,max}}$ is an entourage of $F_{\Gamma_{max,max}}$. So it remains to show the other inclusion.}
	We choose a bounded subset $A$ of $X$ such that $\Gamma A=X$. Let $U$ be an invariant entourage of $X$ containing the diagonal.
	Then 
	$U[A]$ is bounded. Furthermore, we have $U\subseteq \Gamma(A\times  U[A])$. The set
	$W':=F\cap (U[A]\times \Gamma)$ is finite since $F$ is locally finite and $U[A]\times \Gamma$ is bounded. We let $W$ denote the projection of $W'$ to $\Gamma$.
	Then we have
	\[(U\times \Gamma\times \Gamma)\cap (F\times F)\subseteq
	(U\times \Gamma (W\times W))\cap (F\times F)\ .\]
	Now note that $\Gamma(W\times W)$ is an entourage of $\Gamma_{can,max}$. This shows that every entourage of $F_{\Gamma_{max,max}}$ is an entourage of $F_{\Gamma_{can,max}}$.
\end{proof}

\cref{oreorpgergregeg} follows from combining \cref{lem:comparison1} (with $L=\Gamma_{can,min}$), \cref{lem:comparison2}, \cref{lem:comparison3} and \cref{lem:comparison4}.

\subsection{Homological properties of pull-backs by the cone}
\label{secjkfwe0002323}

Recall \cref{gkirjoiergegergergerg} of the three functors
$F$, $F^{0}$, and $F^{\infty}$. In this section we analyze the homological properties of the functor  
$F^{\infty} $. It turns out that this functor is almost a coarse homology theory. 
The only problematic axiom is   vanishing  on flasques. In order to improve on this point recall the definition of $\Gamma\Sp\cX_{\wfl}$ from \cref{def:wfl} and consider the composition
\[F^{\infty}_{\wfl}\colon \Gamma\BC\xrightarrow{F^{\infty}} \Gamma\Sp\cX\to \Gamma\Sp\cX_{\wfl}\ .\]
In a similar manner, we derive functors
$F^{0}_{\wfl}$ and $F_{\wfl}$ from $F^{0}$ and $F$, respectively.
For every $\Gamma$-bornological coarse space $X$ we have a fiber sequence in $\Gamma\Sp\cX_{\wfl}$
\begin{equation}\label{fr3f3f3f4ff}
F^{0}_{\wfl}(X)\to F_{\wfl}(X)\to F^{\infty}_{\wfl}(X) \xrightarrow{\beta_{X,\wfl}} \Sigma F^{0}_{\wfl}(X)\ .
\end{equation}
The morphism $\beta_{X,\wfl}$ is a version of the forget-control morphism from \cref{rlkjflewf89233iurz2i3urkjf2ff}.
\begin{prop}\label{roielrgergergg}
The functor
$F_{\wfl}^{\infty}$
is an   equivariant $\Gamma\Sp\cX_{\wfl}$-valued  coarse homology theory.
\end{prop}

\begin{proof}
We verify the axioms. 
\begin{enumerate}
\item\label{woifewfewfewfe} (Coarse invariance)
We consider a $\Gamma$-bornological coarse space $X$.
For $i $ in $\{0,1\}$ let 
 \[\iota_{i}\colon X\to \{0,1\}_{max,max}\otimes X\] denote  the  corresponding inclusions. It suffices to show that
 $F_{\wfl}^{\infty}(\iota_{0})$ and 
 $F_{\wfl}^{\infty}(\iota_{1})$ are equivalent.
For every invariant entourage $U $ of $X$ we consider the invariant entourage $\tilde U:=\{0,1\}^{2}\times U$ of $\{0,1\}_{max,max}\otimes X$. Then the
 map \[[0,1]_{du}\otimes  P_{U}(X)_{bdu} \to P_{\tilde U}(\{0,1\}_{max,max}\otimes X)_{bdu}\]
given by 
\[(t,\mu)\mapsto (1-t)\iota_{0,*}\mu +t\iota_{1,*}\mu\]
is a homotopy between the morphisms of $\Gamma$-uniform bornological coarse spaces
\[P_{U}(X)_{bdu}\to P_{\tilde U}(\{0,1\}_{max,max}\otimes X)_{bdu}\] induced by $\iota_{0}$ and $\iota_{1}$.
 
By the homotopy invariance of the functor $\cO^{\infty}$ (\cref{fijofiwewefewf}) we  conclude that
the   morphisms
\[\cO^{\infty}(P_{U}(X)_{bdu})\to \cO^{\infty}(P_{\tilde U}(\{0,1\}_{max,max}\otimes X)_{bdu})\]  induced by $\iota_{0}$ and $\iota_{1}$ are equivalent.
Since the entourages of the form $\tilde U$ for all $U$ in $\cC^{\Gamma}$ are cofinal in the entourages of $\{0,1\}\otimes X$, we get the equivalence of
$F^{\infty}_{\wfl}(\iota_{0})$ and $F^{\infty}_{\wfl}(\iota_{1})$  as desired.
\item (Excision)
Let $X$ be a $\Gamma$-bornological coarse space and $Z$ an invariant subset. For an invariant entourage $U$ of $X$ the subset $P_{U}(Z)$ of $P_{U}(X)_{bdu}$ is invariant and closed.

Let $(\cY,Z)$ be an equivariant complementary pair on $X$ with $\cY=(Y_{i})_{i\in I}$.
Let $i_{0}$ in $I$ be such that $Y_{i_{0}}\cup Z=X$. Let $i_{1}$ in $I$ be such that
$U[Y_{i_{0}}]\subseteq Y_{i_{1}}$. Then for every $i$ in $I$ with $i\ge \max\{i_{0},i_{1}\}$ we have
$P_{U}(Y_{i})\cup P_{U}(Z)=P_{U}(X)$. The pair of invariant subsets
$(P_{U}(Y_{i}), P_{U}(Z))$ is then an equivariant uniform   decomposition of $P_{U}(X)_{bdu}$. By \cref{kldjedjoiewufowe23435335} and \cref{oerjgroigrgregegergeg} the functor   $\cO^{\infty}$ sends equivariant uniform   decompositions to push-outs.  We conclude that for $i $ in $I$ with $i\ge \max\{i_{0},i_{1}\}$
we have a push-out
\[\xymatrix{\cO^{\infty}(P_{U}(Z\cap Y_{i})_{bdu})\ar[r]\ar[d]&\cO^{\infty}(P_{U}(Y_{i})_{bdu})\ar[d]\\ \cO^{\infty}(P_{U}(Z)_{bdu})\ar[r]&\cO^{\infty}(P_{U}(X)_{bdu})}\]
{Since colimits of push-out squares are push-out squares, we now take the colimits over the invariant entourages $U$ in $\cC^{\Gamma}$ and over $i$ in $I$ to get the push-out square}
\[\xymatrix{F^{\infty}( Z\cap \cY)\ar[r]\ar[d]&F^{\infty}( \cY)\ar[d]\\ F^{\infty} (Z)\ar[r]&F^{\infty}(X)}\]
We get  the desired push-out \[\xymatrix{F_{\wfl}^{\infty}( Z\cap \cY)\ar[r]\ar[d]&F_{\wfl}^{\infty}( \cY)\ar[d]\\ F_{\wfl}^{\infty} (Z)\ar[r]&F_{\wfl}^{\infty}(X)}\]
\item (Flasqueness)
We assume that $X$ is flasque with the flasqueness implemented by the equivariant map 
$f\colon X\to X$. For an invariant entourage $U$ of $X$  with the property
$(\id,f)(\diag_{X})\subseteq U$
we form the entourage of $X$
\[\tilde U:=\bigcup_{n\in \nat} (f^{n}\times f^{n})(U)\ .\] Note that  
$(f\times f)(\tilde U)\subseteq \tilde U$. Therefore we have a morphism
\[P_{\tilde U}(f)\colon P_{\tilde U}(X)_{bdu}\to P_{\tilde U}(X)_{bdu}\ .\] Like every simplicial map 
  it is   distance decreasing.  Moreover,
for every $\mu\in P_{\tilde U}(X)$ we have \[d(\mu,P_{\tilde U}(f)(\mu))\le 2\ .\]
Finally, if $B$ is a bounded subset of $X$  and $n$ is an integer  such that $\tilde U[B]\cap f^{n}(X)=\emptyset$, then \[P_{\tilde U}(f^{n})(P_{\tilde U}(X))\cap P_{\tilde U}(B)=\emptyset\ .\]

We conclude that $P_{\tilde U}(f)$ implements flasqueness of the bornological coarse space $P_{\tilde U}(X)_{bd}$.
The set of invariant entourages of the form $\tilde U$ as above is cofinal in all invariant entourages of $X$. Therefore, we get $F^{0}(X)\simeq 0$ and hence $F^{0}_{\wfl}(X)\simeq 0$ by taking the colimit over these entourages.

We  now  claim that
$\cO(P_{\tilde U}(f))$ implements weak flasqueness of $\cO(P_{\tilde U}(X)_{bdu})$. In the following we verify the conditions stated in \cref{iogegergeger}. 

Since $f$ is $U$-close to $\id_{X}$, as in \ref{woifewfewfewfe} we can conclude that the map
$P_{\tilde U}(f)_{bdu}$ is uniformly homotopic to $\id_{P_{\tilde U}(X)_{bdu} }$.
By the homotopy invariance of
$\cO$ we conclude that
\[\Yo^{s}(\cO(P_{\tilde U}(f)))\simeq \id_{\Yo^{s}(\cO(P_{\tilde U}(f)_{bdu}))}\]
as required in \cref{iogegergeger}.\ref{iogegergeger1}.

In order to save notation we define the map 
 \[Q\colon \cP([0,\infty)\times P_{\tilde U}(X)\times [0,\infty)\times P_{\tilde U}(X))\to \cP([0,\infty)\times P_{\tilde U}(X)\times [0,\infty)\times P_{\tilde U}(X) )\]  by  \[Q(V):=\bigcup_{n\in \nat} \big(([0,\infty)\times  P_{\tilde U}(f))^{n}\times ([0,\infty)\times  P_{\tilde U}(f))^{n} \big) (V)\ .\]

 Let now $V$ be an entourage of $\cO(P_{\tilde U}(X)_{bdu})$.  We must show that $Q(V)$ is again 
an entourage of $\cO(P_{\tilde U}(X)_{bdu})$.
After enlarging $V$ we can assume that it is of the form $V=U_{\psi}\cap W_{r}$ as in the proof of \cref{ifeioewfueoiwfwfewfwf}, {where the function $\phi$ (which is the first component of $\psi$) is such that $\phi(i)$ is a uniform entourage of the form $U_{r(i)}$ for every $i$ in $\nat$, see \eqref{fhbeufhiu24r}.} Since $P_{\tilde U}(f)$ is distance decreasing we see that $Q(W_{r})\subseteq W_{r}$. Since $P_{\tilde U}(f)$ preserves the first coordinate of the cone {and is distance decreasing} we also see that $Q(U_{\psi})\subseteq {U_{\psi}}$. Hence we actually get $Q(V)\subseteq V$.
 
 Finally, for every bounded subset $A$ of $\cO(P_{\tilde U}(X)_{bdu})$ there exists $r$ in $(0,\infty)$ and a bounded subset $B$ of $X$ such that $A\subseteq [0,r]\times P_{\tilde U}(B)$. We can choose an integer $n$  such that $f^{n}(X)\cap B =\emptyset$.  Then $\cO(P_{\tilde U}(f))^{n}(\cO(P_{\tilde U}(X)_{bdu}))\cap A=\emptyset$.

We conclude that
\[\Yo^{s}_{\wfl}(\cO(P_{\tilde U}(X)_{bdu})) \simeq 0\ .\]
 Taking the colimit over the invariant entourages entourages $U$ and again using the cofinality of the resulting family of entourages $\tilde U$
 we get $F_{\wfl}(X)\simeq 0$. 
 
From  the fiber sequence \eqref{fr3f3f3f4ff}
we now conclude that
\[F^{\infty}_{\wfl}(X)\simeq 0\ .\]

\item ($u$-continuity) This is just a cofinality check:
\[\mathclap{
\colim_{U\in \cC^{\Gamma}} F^{\infty}_{\wfl}(X_{U})\simeq  \colim_{U\in \cC^{\Gamma}} \colim_{V\in \cC\langle U\rangle^{\Gamma}} \cO^{\infty}_{\wfl}(P_{V}(X)_{bdu})\simeq \colim_{V\in \cC^{\Gamma}} \cO^{\infty}_{\wfl}(P_{V}(X)_{bdu})\simeq F^{\infty}_{\wfl}(X) \ .
}\]
\end{enumerate}
This finishes the proof of \cref{roielrgergergg}.
\end{proof}

\begin{rem}
Let $X$ be flasque. In the above proof we have shown that $F_{\wfl}(X)\simeq 0$. Note that we do not expect that $F(X)\simeq 0$.
\end{rem}

Let $E$ be a strong $\Gamma$-equivariant $\bC$-valued coarse homology theory. Then we have an essentially unique factorization $E_{\wfl}\colon \Gamma\Sp\cX_{\wfl} \to \bC$. The composition
\[E_{\wfl}\circ F^{\infty}_{\wfl}\colon \Gamma\BC\to \bC\]
is then a $\Gamma$-equivariant $\bC$-valued coarse homology theory. We have an equivalence
\[E\circ F^{\infty}\simeq E_{\wfl}\circ F^{\infty}_{\wfl}\ .\]

\begin{kor}
If $E$ is a strong  equivariant $\bC$-valued coarse homology theory, then 
 \[E\circ F^{\infty}\colon \Gamma \BC\to  \bC\] is a $\Gamma$-equivariant coarse homology theory. 
\end{kor}

Let $E$ be an equivariant coarse homology theory and $Q$ be a $\Gamma$-bornological coarse space.

\begin{lem}
If $E$ is strong, then $E_{Q}$ is also strong. 
\end{lem}
\begin{proof}
If $X$ is a weakly flasque $\Gamma$-bornological coarse space with weak flasqueness implemented by $f\colon X\to X$, then
$f\otimes \id_{Q}$ implements weak flasqueness of $
X\otimes Q$. This implies the lemma.
\end{proof}

If the underlying $\Gamma$-set of $Q$ is free, then by \cref{wefewfefewfwfwfw} we have
\[E_{Q}(F^{0}(X))\simeq E(F^{0}(X)\otimes \Yo^{s}(Q))\simeq E(X\otimes Q)\simeq E_{Q}(X)\ .\] In particular,
the functor
\[E_{Q}\circ F^{0}\colon \Gamma\BC\to \bC\]
is an equivariant coarse homology theory.

Let $Q$ be a $\Gamma$-bornological coarse space and $E$ be an   equivariant   coarse homology theory.

\begin{kor} 
If $E$ is   strong  and  the underlying $\Gamma$-set of $Q$ is free, then the forget-control map
\[\beta \colon E_{Q}\circ F^{\infty}\to \Sigma E_{Q}\circ F^{0}\]
is a transformation between equivariant coarse homology theories.\end{kor}

This aspect of the theory (in the case of a trivial group $\Gamma$) is further studied in \cite{ass}.

\bibliographystyle{alpha}
\bibliography{born-equiv}
\end{document}